\documentclass[12pt, reqno]{amsart}

\usepackage{amsmath,amssymb,amsfonts,amsthm}
\usepackage{color}

\usepackage[colorlinks=true,linkcolor=blue,citecolor=black,urlcolor=black,pdfborder={0 0 0}]{hyperref}
\usepackage{graphicx}
\usepackage{cleveref}

\usepackage{bbm}

\usepackage{caption}
\usepackage{thmtools}
\usepackage{commath}
\usepackage{mathrsfs}

\usepackage{enumitem}

\crefname{theorem}{Theorem}{Theorems}
\crefname{thm}{Theorem}{Theorems}
\crefname{lemma}{Lemma}{Lemmas}
\crefname{lem}{Lemma}{Lemmas}
\crefname{remark}{Remark}{Remarks}
\crefname{prop}{Proposition}{Propositions}
\crefname{defn}{Definition}{Definitions}
\crefname{corollary}{Corollary}{Corollaries}
\crefname{conjecture}{Conjecture}{Conjectures}
\crefname{question}{Question}{Questions}
\crefname{chapter}{Chapter}{Chapters}
\crefname{section}{Section}{Sections}
\crefname{figure}{Figure}{Figures}

\theoremstyle{plain}
\newtheorem{thm}{Theorem}[section]
\newtheorem*{thm*}{Theorem}
\newtheorem{lemma}[thm]{Lemma}
\newtheorem{lem}[thm]{Lemma}

\newtheorem{prop}[thm]{Proposition}
\newtheorem{conjecture}[thm]{Conjecture}

\newtheorem{question}[thm]{Question}
\theoremstyle{definition}

\theoremstyle{remark}
\newtheorem*{remark}{Remark}

\numberwithin{equation}{section}

\renewcommand{\P}{\mathbb P}
\newcommand{\E}{\mathbb E}

\newcommand{\R}{\mathbb R}
\newcommand{\Z}{\mathbb Z}

\newcommand{\F}{\mathfrak F}

\newcommand{\cC}{\mathcal C}

\newcommand{\cP}{\mathcal P}

\newcommand{\cU}{\mathcal U}

\newcommand{\cX}{\mathcal X}

\newcommand{\sA}{\mathscr A}
\newcommand{\sB}{\mathscr B}
\newcommand{\sC}{\mathscr C}
\newcommand{\sD}{\mathscr D}
\newcommand{\sE}{\mathscr E}
\newcommand{\sF}{\mathscr F}

\newcommand{\sI}{\mathscr I}

\newcommand{\sR}{\mathscr R}

\newcommand{\sW}{\mathscr W}

\newcommand{\sY}{\mathscr Y}

\newcommand{\bbG}{\mathbb G}

\newcommand{\bbV}{\mathbb V}
\newcommand{\bbW}{\mathbb W}

\newcommand{\WUSF}{\mathsf{WUSF}}
\newcommand{\FUSF}{\mathsf{FUSF}}

\newcommand{\UST}{\mathsf{UST}}

\newcommand{\LE}{\mathsf{LE}}

\usepackage[medium]{titlesec}
\titleformat{\section}
{\normalfont\Large}{\thesection}{1em}{}
\titleformat{\subsection}
{\normalfont\large
}{\thesubsection}{1em}{}
\titleformat{\subsubsection}
{\normalfont\bfseries}{\thesubsubsection}{1em}{}
\titlespacing*{\section}{0pt}{3.5ex plus 1ex minus .2ex}{2.3ex plus .2ex}
\titlespacing*{\subsection}{0pt}{3.25ex plus 1ex minus .2ex}{1.5ex plus .2ex}
\titlespacing*{\subsubsection}{0pt}{3.25ex plus 1ex minus .2ex}{1.5ex plus .2ex}

\newcommand{\bP}{\mathbf{P}}
\newcommand{\eps}{\varepsilon}
\newcommand{\Comp}{\mathcal{C}}

\newcommand{\diam}{\textrm{diam}}

\newcommand{\symdif}{\hspace{.1em}\triangle\hspace{.1em}}

\newcommand{\Witness}{\mathscr W}
\newcommand{\bareta}{\hat \eta}

\newcommand{\coarse}[2]{#1/\!#2}

\newcommand{\myfrac}[3][0pt]{\genfrac{}{}{}{}{\raisebox{#1}{$#2$}}{\raisebox{-#1}{$#3$}}}

\usepackage[margin=1in]{geometry}
\title[The component graph of the uniform spanning forest]{The component graph of the uniform spanning forest: transitions in dimensions $9,10,11,\ldots$}

\author{Tom Hutchcroft}
\address{University of Cambridge}
\email{t.hutchcroft@maths.cam.ac.uk}
\author{Yuval Peres}
\address{Microsoft Research}
\email{peres@microsoft.com}

\begin{document}

\begin{abstract}
We prove that the uniform spanning forests of $\Z^d$ and $\Z^{\ell}$ have qualitatively different connectivity properties whenever $\ell >d \geq 4$.  In particular, we consider the graph formed by contracting each tree of the uniform spanning forest down to a single vertex, which we call the \emph{component graph}.  We introduce the notion of \emph{ubiquitous subgraphs} and show that the set of ubiquitous subgraphs of the component graph changes whenever the dimension changes and is above $8$. To separate dimensions $5,6,7,$ and $8$, we prove a similar result concerning \emph{ubiquitous subhypergraphs} in the \emph{component hypergraph}.
%
%
% Let $\F^d$ be the uniform spanning forest of $\Z^d$. We consider the graph $\Comp(\F^d)$ formed by contracting each tree in $\F^d$ down to a point, which we call the component graph of $\F^d$. 
% We prove that the component graphs $\Comp(\F^{d_1})$ and $\Comp(\F^{d_2})$
% \Z^{d'}/\F^{d'}$
 % have qualitatively different connectivity properties whenever $d_1>d_2 \geq 9$.
% 
% 
% 
Our result sharpens a theorem of Benjamini, Kesten, Peres, and Schramm, 
 who proved that the diameter of the component graph increases by one every time the dimension increases by four. 
\end{abstract}

\maketitle

% \vspace{-2.3em}

\section{Introduction}

The \textbf{uniform spanning forests} of an infinite, connected, locally finite graph $G$ are defined to be distributional limits of uniform spanning trees of large finite subgraphs of $G$. These limits can be taken with either free or wired boundary conditions, yielding the \textbf{free uniform spanning forest} (FUSF) and \textbf{wired uniform spanning forest} (WUSF) respectively. 
Although they are defined as limits of trees, the USFs are not necessarily connected. 
Indeed, Pemantle \cite{Pem91} proved that the FUSF and WUSF of $\Z^d$ coincide for all $d$ (so that we can refer to both simply as the USF of $\Z^d$), and are a single tree almost surely (a.s.) if and only if $d\leq 4$. A complete characterization of the connectivity of the WUSF was given by Benjamini, Lyons, Peres, and Schramm \cite{BLPS}, who proved that the WUSF of a graph is connected if and only if two independent random walks on $G$ intersect infinitely often a.s.

Extending Pemantle's result, Benjamini, Kesten, Peres, and Schramm \cite{BeKePeSc04} (henceforth referred to as BKPS) discovered the following surprising theorem. 

\begin{thm*}[BKPS \cite{BeKePeSc04}]
	Let $\F$ be a sample of the USF of $\Z^d$. For each
	$x,y \in \Z^d$,  let $N(x,y)$ be the minimal number of edges that are not in $\F$ used by a path from $x$ to $y$ in $\Z^d$.
	Then 
	\[\max_{x,y \in \Z^d}N(x,y) = \left\lceil \frac{d-4}{4} \right\rceil\]
	almost surely. 
\end{thm*}

In particular, this theorem shows that every two trees in the uniform spanning forest of $\Z^d$ are adjacent almost surely if and only if $d\leq 8$. 
Similar results have since been obtained for other models \cite{procaccia2011geometry,rath2010connectivity,broman2016connectedness,li2016percolative,procaccia2016connectivity}. 
The purpose of this paper is to show that, once $d\geq 5$,
  the uniform spanning forest undergoes qualitative changes to its connectivity \emph{every} time the dimension increases, rather than just every four dimensions.

% \medskip

% \vspace{-1.5em}

In order to formulate such a theorem, we introduce the \emph{component graph} of the uniform spanning forest.
Let $G$ be a graph and let $\omega$ be a subgraph of $G$. The \textbf{component graph} $\Comp_1(\omega)$ of $\omega$ is defined to be the simple graph that has the connected components of $\omega$ as its vertices, and has an edge between two connected components $k_1$ and $k_2$ of $\omega$ if and only if 
there exists an edge $e$ of $G$ that has one endpoint in $k_1$ and the other endpoint in $k_2$. More generally, for each $r\geq 1$, we define the \textbf{distance $r$ component graph} $\Comp_r(\omega)$ to be the graph which has the components of $\omega$ as its vertices, and has an edge between two components $k_1$ and $k_2$ of $\omega$ if and only if there is path in $G$ from $k_1$ to $k_2$ that has length at most $r$.

% \vspace{0.3em}

When formulated in terms of the component graph, the result of BKPS states that
 the diameter of $\Comp_1(\F)$ is almost surely $\lceil (d-4)/4\rceil$ for every $d\geq 1$. In particular, it implies that $\Comp_1(\F)$ is almost surely a single point for all $1\leq d \leq 4$ (as follows from Pemantle's theorem), and is almost surely a complete graph on a countably infinite number of vertices for all $5\leq d\leq 8$.

% \vspace{0.3em}

We now introduce the notion of \emph{ubiquitous subgraphs}. 
We define a \textbf{graph with boundary} $H=(\partial V, V_\circ,E)=(\partial V(H),V_\circ (H),E(H))$ to be a graph $H=(V,E)$ whose vertex set $V$ is partitioned into two disjoint sets, $V=\partial V \cup V_\circ$, which we call the \textbf{boundary} and \textbf{interior} vertices of $H$, such that $\partial V \neq \emptyset$. 
Given a graph $G$, a graph with boundary $H$, and collection of distinct vertices $(x_u)_{u \in \partial V}$ of $G$ indexed by the boundary vertices of $H$, 
we say that $H$ is \textbf{present} at $(x_u)_{u \in \partial V}$ if there exists a collection of vertices $(x_u)_{u \in V_\circ }$ of $G$ indexed by the interior vertices of $H$ 
such that $x_u \sim x_v$ or $x_u=x_v$ for every $u\sim v$ in $H$. (Note that, in this definition, we do \emph{not} require that $x_u$ and $x_v$ are not adjacent in $G$ if $u$ and $v$ are not adjacent in $H$.) We say that $H$ is \textbf{faithfully present} at $(x_u)_{u\in \partial V}$ if there exists a collection of \emph{distinct} vertices $(x_u)_{u \in V_\circ}$ of $G$, disjoint from $(x_u)_{u \in \partial V}$, indexed by the interior vertices of $H$ 
such that $x_u \sim x_v$ for every $u\sim v$ in $H$. In figures, we will use the convention that boundary vertices are white and interior vertices are black.

% \vspace{0.3em}
We say that $H$ is \textbf{ubiquitous} in $G$ if it is present at every collection of distinct vertices  $(x_u)_{u\in \partial V}$ in $G$, and that $H$ is \textbf{faithfully ubiquitous} in $G$ if it is faithfully present at every collection of distinct vertices  $(x_u)_{u\in \partial V}$ in $G$.

% \vspace{0.3em}

For example, if $H$ is a path of length $n$ with the endpoints of the path as its boundary, then $H$ is ubiquitous in a graph $G$ if and only if $G$ has diameter less than or equal to $n$. The same graph is \emph{faithfully} ubiquitous in $G$ if and only if every two vertices of $G$ can be connected by a simple path of length \emph{exactly} $n$. If $H$ is a star with $k$ leaves set to be in the boundary and the central vertex set to be in the interior, then $H$ is ubiquitous in a graph $G$ if and only if every $k$ vertices of $G$ share a common neighbour, and in this case $H$ is also faithfully ubiquitous. 

\begin{figure}[t]
	\centering
	\includegraphics[width=0.7\textwidth]{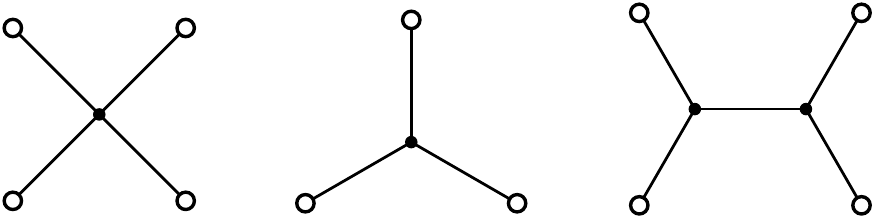}
	\caption{
	\label{fig:sepfamilybasic}
	% \small{
	Three trees with boundary that can be used to distinguish the component graphs of the uniform spanning forest in dimensions $9,10,11,$ and $12$. Boundary vertices are white, interior vertices are black.}
	\vspace{-0.5cm}
	% }
\end{figure}

% \vspace{0.3em}

The main result of this paper is the following theorem. We say that a transitive graph $\bbG$ is \textbf{$d$-dimensional} if there exist positive constants $c$ and $C$ such that $cn^d \leq |B(x,n)|\leq Cn^d$ for every vertex $x$ of $\bbG$ and every $n\geq 1$, where $B(x,n)$ denotes the graph-distance ball of radius $n$ around $x$ in $\bbG$. The WUSF and FUSF of any $d$-dimensional transitive graph coincide~\cite{BLPS}, and we speak simply of the USF of $\bbG$. Note that the geometry of a $d$-dimensional transitive graph may be very different from that of $\Z^d$. (Working at this level of generality does not add any substantial complications to the proof, however.)

\begin{thm}\label{thm:mainsimple}
	Let $\bbG_1$ and $\bbG_2$ be transitive graphs of dimension $d_1 $ and $d_2$ respectively, and let $\F_1$ and $\F_2$ be uniform spanning forests of $\bbG_1$ and $\bbG_2$ respectively. Then the following claims hold for every $r_1,r_2\geq 1$:
	\vspace{0.3em}
	\begin{enumerate}[leftmargin=0.9cm]\itemsep0.5em
		\item 
			\emph{(\textbf{Universality and monotonicity.})}
			If $d_1 \geq d_2 \geq 9$, then every finite graph with boundary that is ubiquitous in $\Comp_{r_1}(\F_1)$ is also ubiquitous in $\Comp_{r_2}(\F_2)$ almost surely.
		
		\item 
			\emph{(\textbf{Distinguishability of different dimensions.})} 
			If $d_1 > d_2 \geq 9$,  then there exists a finite graph with boundary $H$ such that $H$ is almost surely ubiquitous in $\Comp_{r_2}(\F_2)$  but not in $\Comp_{r_1}(\F_1)$.
	\end{enumerate}
	Moreover, the same result holds with `ubiquitous' replaced by `faithfully ubiquitous'.
\end{thm}

% Note that, by ergodicity of the uniform spanning forest, the probability that a finite graph with boundary is faithfully ubiquitous in the component graph of the uniform spanning forest of a transitive graph $\bbG$ is either zero or one. 

% \medskip

In order to prove item $(2)$ of \cref{thm:mainsimple}, it will suffice to consider the case that $H$ is a tree. In this case,
the following theorem allows us to calculate the dimensions for which $H$ is ubiquitous in the component graph of the uniform spanning forest.  The corresponding result for general $H$ is given in \cref{thm:main}. Examples of trees that can be used to distinguish between different dimensions using \cref{thm:maintree} are given in \cref{fig:sepfamilybasic,fig:sepfamily}.

\begin{thm}\label{thm:maintree}
	Let $\bbG$ be a $d$-dimensional transitive graph for some $d >8$, let $\F$ be a uniform spanning forest of $\bbG$,  let $r\geq 1$, and let $T$ be a finite tree with boundary. Then $T$ is almost surely ubiquitous in $\Comp_r(\F)$ if and only if $T$ is almost surely faithfully ubiquitous in $\Comp_r(\F)$, if and only if
	\[
		\max\left\{\frac{|E(T')|}{|V_\circ(T')|}\,:\, T' \text{ is a subgraph of $T$}\right\} \leq \frac{d-4}{d-8}.
	\]
\end{thm}

\medskip

Note that $(d-4)/(d-8)$ is a decreasing function of $d$ for $d>8$. The theorem of BKPS follows as a special case of \cref{thm:maintree} by taking $T$ to be a path. 
\cref{fig:sepfamily} gives an example of a family of trees that can be used to deduce item $(2)$ of \cref{thm:mainsimple} from \cref{thm:maintree}. See \cref{fig:unbalanced} for another example application.

% \medskip

\begin{figure}[t]
	\centering
	\includegraphics[width=0.7\textwidth]{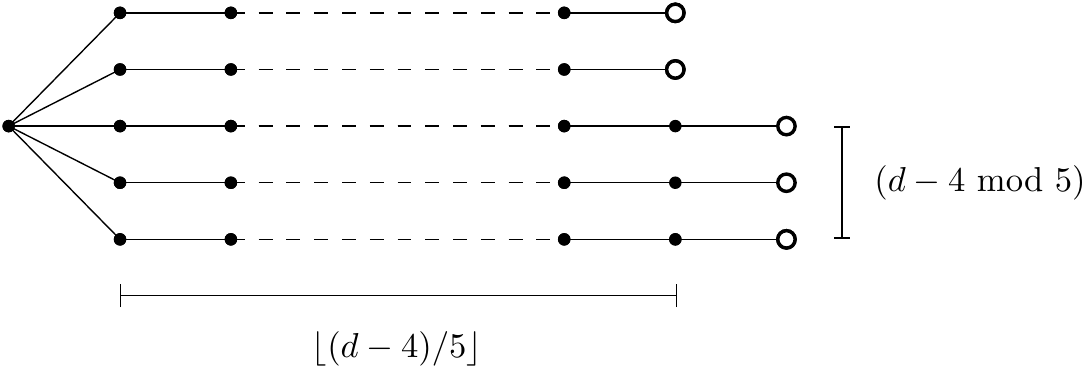}
	\caption{
	\label{fig:sepfamily}
	% \small{
	A family of trees with boundary that can distinguish between $d$ and $d+1$ for any $d\geq 9$. 
	% Boundary vertices are white, interior vertices are black.  
	% }
	}
\end{figure}

\begin{figure}[h!]
\includegraphics[width=0.7\textwidth]{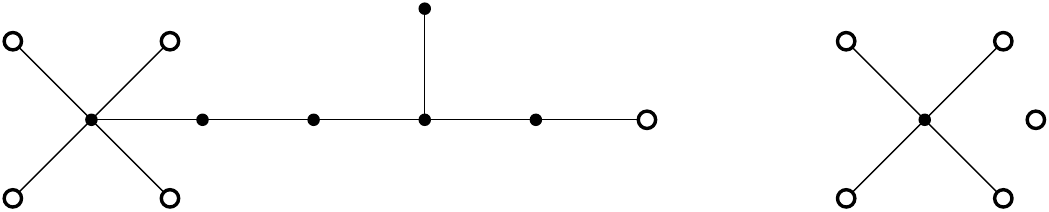}
\caption{
% \small{
Left: a finite tree with boundary $T$. Right: the subgraph $T'$ of $T$ maximizing $|E(T')|/|V_\circ(T')|$. By \cref{thm:maintree}, $T$ is almost surely faithfully ubiquitous in the component graph of the uniform spanning forest of $\Z^d$ if and only if $d\leq 9$.
% }
}
\label{fig:unbalanced}
\end{figure}

% \medskip

The next theorem shows that uniform spanning forests in different dimensions between $5$ and $8$ also have qualitatively different connectivity properties. 
 The result is more naturally stated in terms of \emph{ubiquitous subhypergraphs} in the \emph{component hypergraph} of the USF; see the following section for definitions and \cref{fig:hyper} for an illustration of the relevant hypergraphs.

\begin{thm}[Distinguishing dimensions $5,6,7,$ and $8$.]
	\label{thm:5678}
	 Let $\bbG$ be a $d$-dimensional transitive graph and let $\F$ be a uniform spanning forest of $\bbG$.  The following hold almost surely.
	\begin{enumerate}[leftmargin=*]
		\itemsep0.25em
		\item
		If $d=5$, then there exists a constant $r_0$ such that for every five trees of $\F$, there exists a ball of radius $r_0$ in $\bbG$ that is intersected by each of the five trees. On the other hand, if $d\geq 6$, then for every $r\geq 1$, there exists a set of four trees in $\F$ such that there does not exist a ball of radius $r$ in $\bbG$ intersecting all four trees. 
		\item 
		If $d=5$ or $6$, then there exists a constant $r_0$ such that for every three trees of $\F$, there exists a ball of radius $r_0$ in $\bbG$ that is intersected by each of the three trees. On the other hand, if $d\geq 7$, then for every $r\geq 1$, there exists a set of three trees in $\F$ such that there does not exist a ball of radius $r$ in $\bbG$ intersecting all three trees.
		\item 
		If $d=5,6,$ or $7$, then there exists a constant $r_0$ such that for every $r\geq r_0$, every set of three pairs of trees of $\F$ have the following property: There exist three trees $T_{1},T_{2},T_{3}$ in $\F$ 
		such that $T_i$ and the $i$th pair of trees all intersect some ball $B_i$ of radius $r$ in $G$ for each $i=1,2,3$, and the trees $T_1,T_2,T_3$ all intersect some ball $B_0$ of radius $r$ in $G$.
		On the other hand, if $d \geq 8$, then for every $r\geq 1$ there exists a set of three pairs of trees of $\F$ that do not have this property. 
	\end{enumerate}
\end{thm}

\begin{figure}
\includegraphics[width=0.775\textwidth]{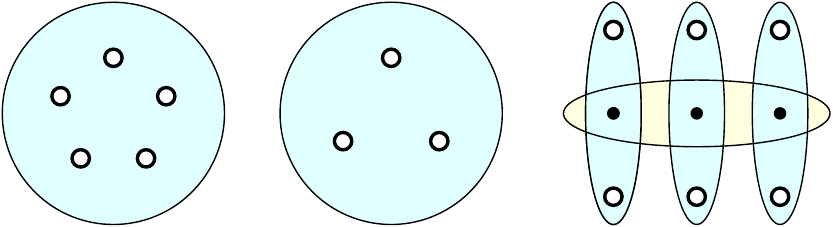}
\caption{
% \small{
Three hypergraphs with boundary that can be used to distinguish the component hypergraphs of the uniform spanning forest in dimensions $5,6,7,$ and $8$.  Edges are represented by shaded regions.
% }
}
\label{fig:hyper}
\end{figure}

\subsection{Ubiquity of general graphs and hypergraphs in the component graph.}
\label{subsec:introgeneral}

\begin{figure}
\includegraphics[width=0.5\textwidth]{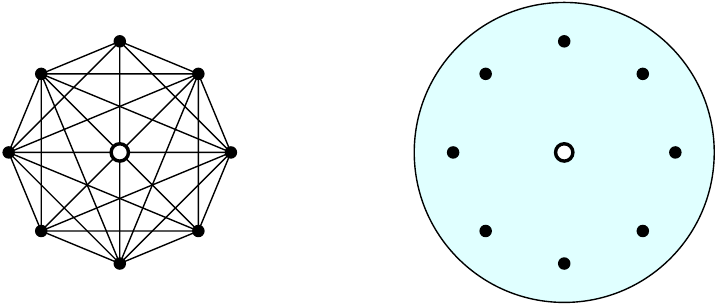}
\caption{
\small{
Considering the coarsening in which all edges of a hypergraph are merged into one shows that any connected graph with $|\partial V|\in \{0,1\}$ is faithfully ubiquitous in the component graph of the uniform spanning forest of $\Z^d$ for every $d>4$, since every subhypergraph of this coarsening has $d$-apparent weight either $-d$ or $-4$.  
}
}
\label{fig:degenerate}
\end{figure}

In this section, we extend \cref{thm:maintree} to the case that $H$ is not a tree. In order to formulate this extension, it is convenient  to consider the even more general setting in which $H$ is a \emph{hypergraph} with boundary. 
Indeed, it is a surprising feature of the resulting theory that one is forced to consider hypergraphs even if one is interested only in graphs.

% \medskip
We define a \textbf{hypergraph} $H=(V,E,\perp)$ to be a triple consisting of a set of vertices $V$, a set of edges $E$, and a binary relation $\perp \subseteq V \times E$ such that the set $\{v\in V : (v,e)\in \perp\}$ is nonempty for every $e\in E$. We write $v \perp e$ or $e \perp v$ and say that $v$ is \textbf{incident} to $e$ if $(v,e)\in \perp$. Note that this definition is somewhat nonstandard, as it allows multiple edges with the same set of incident vertices. We say that a hypergraph is \textbf{simple} if it does not contain two distinct edges whose sets of incident vertices are equal. 
Every graph is also a hypergraph. 
A \textbf{hypergraph with boundary} $H=(\partial V,V_\circ,E,\perp)$ is defined to be a hypergraph $H=(V,E,\perp)$ together with a partition of $V$ into disjoint subsets, $V=\partial V \cup V_\circ$, the \textbf{boundary} and \textbf{interior} vertices of $H$, such that $\partial V \neq \emptyset$. The degree of a vertex in a hypergraph is the number of edges that are incident to it, and the degree of an edge in a hypergraph is the number of vertices it is incident to.
To lighten notation, we will often write simply $H=(\partial V, V_\circ, E)$ for a hypergraph with boundary, leaving the incidence relation $\perp$ implicit. 

% \medskip

If $H=(\partial V, V_\circ,E,\perp)$ is a hypergraph with boundary, a \textbf{subhypergraph} (with boundary) of $H$ is defined to be a hypergraph with boundary of the form $H'=(\partial V',V_\circ',E',\perp')$, where \[
\text{$\partial V' \subseteq \partial V$,\, $V_\circ' \subseteq V_\circ$,\, $E' \subseteq E$,\, $V'=\partial V' \cup V_\circ'$,\, and\, $\perp' = \perp \cap\, (V'\times E')$.}\] 
% 
% Let $H_1=(V_1,E_1,\perp_1)$ and $H_2=(V_2,E_2,\perp_2)$. A \textbf{homomorphism} $\phi:H_1\to H_2$ is a pair $\phi=(\phi_V,\phi_E)$ of functions $\phi_V: V_1 \to V_2$  and $\phi_E : E_1 \to E_2$ for which 
% $\{v $
We say that a hypergraph with boundary $H'=(\partial V',V_\circ',E',\perp')$  is a \textbf{quotient} of a hypergraph with boundary  $H=(\partial V,V_\circ,E,\perp)$  if there exists a surjective function $\phi_V : V\to V'$ mapping $\partial V$ bijectively onto $\partial V'$ and a bijective function $\phi_E:E\to E'$ such that
\[
\{v' : v' \in V',\, v' \perp' \phi_E(e) \} = \{ \phi_V(v) : v\in V,\, v\perp e\}
\]
for every $e\in E$.
% $\{\phi(v) : v \perp e\}$ for some edge $e$ of $H$. 
Similarly, we say that $H'$ is a \textbf{coarsening} of $H$ (and call $H$ a \textbf{refinement} of $H'$) if there exists a bijection $\phi_V: V \to V'$ mapping $\partial V$ bijectively onto $\partial V'$ and a surjection $\phi_E : E \to E'$ such that 
\[
\{e' : e'\in E',\, e' \perp' \phi_V(v)\}= \{ \phi_E(e) : e\in E,\, e\perp v\}
\]
% $\phi_V(v) \in \phi_E(e)$ for every edge $e\in E(H)$  and every vertex $v \perp e$.
for every $v\in V$.
In other words, $H'$ is a \emph{quotient} of $H$ if 
it can be 
obtained from $H$ by merging together some of the \emph{vertices} of $H$, while $H'$ is a \emph{coarsening} of $H$ if it can be obtained by merging together some of the \emph{edges} of $H$.

% \medskip

The following theorem allows us to calculate the dimensions for which an arbitrary finite simple graph $H$ is ubiquitous in the component graph of the uniform spanning forest.  It will be used to deduce \cref{thm:mainsimple,thm:maintree}. See \cref{fig:example,fig:degenerate} for example applications. For each finite hypergraph with boundary $H=(\partial V, V_\circ,E)$ and $d \in \R$, we define the \textbf{weight} of $H$, denoted $\Delta(H)$, and the $d$-\textbf{apparent weight} of $H$, denoted by $\eta_d(H)$, by setting
\[\Delta(H) := \sum_{e\in E} \deg(e) = \sum_{v\in V} \deg(v) \quad \text{ and } \quad \eta_d(H) := (d-4)\Delta-d|E| -(d-4)|V_\circ|\]
respectively. 
We say that $H$ is $d$-\textbf{buoyant} if $\eta_d(H) \leq 0$, i.e., if its $d$-apparent weight is non-positive.  
 If $H$ is a simple graph then $\Delta=2|E|$ and so $\eta_d(H) = (d-8)|E| -(d-4)|V_\circ|$.

\begin{samepage}
\begin{thm}\label{thm:main}
Let $\bbG$ be a $d$-dimensional transitive graph for some $d>4$, let $\F$ be the uniform spanning forest of $\bbG$, let $H$ be finite simple graph with boundary, and let $r\geq 1$. Then $H$ is faithfully ubiquitous in $\Comp_r(\F)$ almost surely  if and only
 % if $\eta_d(H')$ is non-positive for some $H'$ that is a coarsening of a subgraph of $H$.
% satisfies the sentence $\cS_H$ almost surely if and only if $H$ is reducible to safe.
\begin{equation}
\label{eq:thmmaincriterion}
 \min\left\{ \max \left\{\eta_d(H'') : H'' \text{ is a subhypergraph of $H'$}\right\} : H' \text{ is a coarsening of $H$}\right\}\leq 0,
\end{equation}
that is, if and only if $H$ has a coarsening all of whose subhypergraphs are $d$-buoyant. 
% and is ubiquitous in $\Comp_r(\F)$ if and only if
% \[ \max\left\{ \min \left\{\eta_d(H'') : H'' \text{ is a coarsening of $H'$}\right\} : H' \text{ is a subgraph of $H$}\right\}\leq 0\]
Moreover, $H$ is ubiquitous in $\Comp_r(\F)$ if and only if it has a quotient that is faithfully ubiquitous in $\Comp_r(\F)$ almost surely.
% If $H$ is a graph, we may take $r_H=1$.
\end{thm}
\end{samepage}

This terminology used here arises from the following analogy: We imagine that from each vertex-edge pair $(v,e)$ of $H$ with $v\perp e$ we hang a weight exerting a downward force of $(d-4)$, while from each edge and each interior vertex of $H$ we attach a balloon exerting an upward force of either $d$ or $(d-4)$ respectively. The net force is equal to the apparent weight. The hypergraph is buoyant (i.e., floats) if the apparent weight is non-positive.

% \medskip
\cref{thm:main} is best understood as a special case of a more general theorem concerning the \emph{component hypergraph}. 
Given a subset $\omega$ of a graph $G$ and $r\geq 1$, we define the \textbf{component hypergraph} $\Comp^{hyp}_r(\omega)$ to be the simple hypergraph that has the components of $\omega$ as vertices, and where a finite set of components $W$  is an edge of $\Comp^{hyp}_r(\omega)$ if and only if there exists a set of diameter $r$ in $G$ that intersects every component of $\omega$ in the set $W$.
Presence, faithful presence, ubiquity and faithful ubiquity of a hypergraph with boundary $H$ in a hypergraph $G$ are defined similarly to the graph case. For example, we say that a finite hypergraph with boundary $H=(\partial V, V_\circ, E)$ is \textbf{faithfully present} at $(x_u)_{u\in \partial V}$ in $G$ if there exists a collection of distinct vertices $(x_u)_{u \in V_\circ}$ of $G$, disjoint from $(x_u)_{u \in \partial V}$, indexed by the interior vertices of $H$ 
such that for each $e\in E$ there exists an edge $f$ of $G$ that is incident to all of the vertices in the set $\{x_v : v \perp e\}$.
Given a $d$-dimensional graph $\bbG$ and $M\geq 1$, we let $R_\bbG(M)$ be minimal such that there exists a set of vertices in $\bbG$ of diameter $R_\bbG(M)$ that intersects $M$ distinct components of the uniform spanning forest of $\bbG$ with positive probability. Given a hypergraph with boundary $H$, we let $R_\bbG(H) = R_\bbG(\max_{e\in E} \deg(e))$.

\begin{figure}
\centering
\includegraphics[width=1\textwidth]{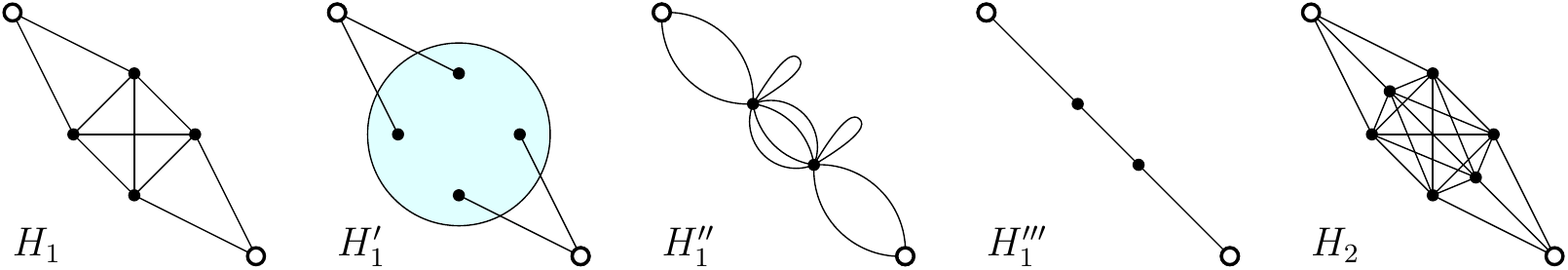}
% \qquad
% \includegraphics[height=0.2\textwidth]{subgraph.pdf}
\caption{
% \footnotesize{
\small{
The graph with boundary $H_1$, far left, has $d$-apparent weight $\eta_d(H_1)=6d-64$, and is therefore $d$-buoyant if and only if $d \leq  10 + 2/3$. Meanwhile, it has a coarsening $H'_1$, centre left, that has $d$-apparent weight $\eta_d(H'_1) = 5d-64$, so that $H'_1$ is $d$-buoyant if and only if $d \leq 12 + 4/5$. In fact, using \cref{thm:main}, it can be verified  that $H_1$ is almost surely faithfully ubiquitous in the component graph of the uniform spanning forest of $\Z^d$ if and only if $d  \leq 12$. 
% (For this we also need to check the values of $\eta_d$ for coarsenings of subgraphs of $H$.) 
On the other hand, considering the quotient $H''_1$ of $H_1$, centre, and the coarsening $H_1'''$ of $H''_1$, centre right, along with other possible coarsenings of quotients, it can be verified that $H_1$ is almost surely ubiquitous in the component graph of the uniform spanning forest of $\Z^d$ if and only if $d \leq 16$. Thus, for $13 \leq d \leq 16$ the graph $H_1$ is ubiquitous but not faithfully ubiquitous in the component graph of the uniform spanning forest of $\Z^d$ a.s. A similar analysis shows that the graph $H_2$, far right, is faithfully ubiquitous in the component graph of the uniform spanning forest of $\Z^d$ almost surely if and only if $d \leq 9$, and is ubiquitous almost surely if and only if $d\leq 16$.
}
}
\label{fig:example}
\end{figure}

\begin{thm}\label{thm:mainhyper}
Let $\bbG$ be a $d$-dimensional transitive graph for some $d>4$, let $\F$ be the uniform spanning forest of $\bbG$, and let $H$ be finite hypergraph with boundary. If
\begin{equation}
\label{eq:thmmainhypercriterion} 
 \min\left\{ \max \left\{\eta_d(H'') : H'' \text{ is a subhypergraph of $H'$}\right\} : H' \text{ is a coarsening of $H$}\right\}\leq 0
\end{equation}
then $H$ is faithfully ubiquitous in $\Comp^{hyp}_r(\F)$ almost surely for every $r\geq R_\bbG(H)$. Otherwise, $H$ is not faithfully ubiquitous in $\Comp^{hyp}_r(\F)$ for any $r\geq 1$ almost surely. Moreover, $H$ is ubiquitous in $\Comp^{hyp}_r(\F)$ if and only if it has a quotient that is faithfully ubiquitous in $\Comp^{hyp}_r(\F)$ almost surely.
% If $H$ is a graph, we may take $r_H=1$.
\end{thm}

$H$ clearly cannot be faithfully ubiquitous in $\cC^{hyp}_r(\F)$ if $r < R_\bbG(H)$, so this condition is necessary.
Note that $R_\bbG(2)=1$ for any $d$-dimensional transitive graph with $d >4$, so that \cref{thm:main} follows as a special case of \cref{thm:mainhyper} as claimed. 
 \cref{thm:5678} follows immediately by applying \cref{thm:mainhyper} to the hypergraphs pictured in \cref{fig:hyper}. The $\min\max$ problem arising in \eqref{eq:thmmaincriterion} and \eqref{eq:thmmainhypercriterion} is studied in \cref{subsec:optimalcoarsenings}.

\subsection{Organisation}

In \cref{sec:background}, we give background on uniform spanning forests, establish notation, and prove some simple preliminaries that will be used throughout the rest of the paper. 
In \cref{sec:notationsandoutline}, we outline some of the key steps in the proof of the main theorems; this section is optional if the reader prefers to go straight to the fully detailed proofs.
 \cref{sec:moments} is the computational heart of the paper, where the quantitative estimates needed for the proof of the main theorems are established.   In \cref{sec:wrappingup}, we deduce the main theorems from the estimates of \cref{sec:moments} together with the multicomponent indistinguishability theorem of \cite{MulticomponentIndistinguishability}, which is used as a zero-one law. This section is quite short, most the work having already been done in \cref{sec:moments}. 
 We conclude with some open problems and remarks in \cref{sec:closing}.

\section{Background, definitions, and preliminaries}
\label{sec:background}

\subsection{Basic notation}
Let $\bbG$  be a $d$-dimensional transitive graph with vertex set $\bbV$, and let $\F$  be the uniform spanning forest of $\bbG$. For each set $W\subseteq \bbV$, we write 
write $\sF(W)$ for the event that the vertices of $W$ are all in the same component of $\F$. Let $r \geq 1$ and let $H = (\partial V, V_\circ, E)$ be a finite hypergraph with boundary. We define
\[
	\bareta_d(H) = \min\left\{\eta_d(H'): H' \text{ is a coarsening of $H$}\right\}.
\]

We write $\preceq,\succeq$, and $\asymp$ for inequalities or equalities that hold up to a positive multiplicative constant depending only on some fixed data that will be clear from the context, usually $\bbG, H$, and $r$, and write $\lesssim,\gtrsim$ and $\approx$ for inequalities or equalities that hold up to an additive constant depending only on the same data. In particular
\[a \asymp b \text{ if and only if } \log_2 a \approx \log_2 b.\] 
We sometimes write $\exp_2 (a)$ to mean $2^a$.

  For each two vertices $x$ and $y$ of $\bbG$, we write 
$\langle xy \rangle = d_\bbG(x,y)+1$,
where $d_\bbG$ is the graph metric on $\bbG$.  
For each vertex $x$ of $\bbG$ and  $\infty \geq N > n \geq 0$, we define the dyadic shell
\[\Lambda_x(n,N) := \left\{y \in \bbV : 2^n \leq \langle x y \rangle \leq 2^N \right\}.\]
If $x=(x_u)_{u\in \partial V}$ is a collection of vertices in $\bbG$, we choose one such point $x_0$ arbitrarily and set $\Lambda_x(n,N)=\Lambda_{x_0}(n,N)$ for every $N> n \geq 0$.
Since $\bbG$ is $d$-dimensional, we have that
\begin{equation}
	\label{eq:LambdaVolume}
	\log_2 |\Lambda_x(n,N)| \approx dN 
\end{equation}
for all $n \geq 0$ and $N \geq n+1$. The upper bound is immediate, while the lower bound follows because $\Lambda_x(n,N)$ contains both some point $y$ with $\langle x_0y \rangle = 2^{N-1}+2^{N-2}$ and the ball of radius $2^{N-2}$ around this point $y$.

% \end{proof}
\subsection{Uniform spanning forests}
\label{subsec:USFbackground}

Given a finite connected graph $G$, we define $\UST_G$ to be the uniform probability measure on the set of spanning trees of $G$, that is, connected subgraphs of $G$ that contain every vertex of $G$ and do not contain any cycles. 
Now suppose that $G=(V,E)$ is an infinite, connected, locally finite graph, and let $(V_i)_{i\geq 1}$ be an \textbf{exhaustion} of $V$ by finite sets, that is, an increasing sequence of finite, connected subsets of $V$ such that $\bigcup_{i\geq1}V_i=V$. For each $i\geq 1$, let $G_i$ be the subgraph of $G$ induced\footnote{Given a graph $G=(V,E)$ and a set of vertices $W \subseteq V$, the subgraph of $G$ \emph{induced} by $W$ is defined to be the graph with vertex set $W$ and with edge set given by the set of edges of $G$ that have both endpoints in $W$.} by $V_i$, and let $G_i^*$ be the graph formed from $G$ by contracting $V \setminus V_i$ down to a single vertex and deleting all of the self-loops that are created by this contraction. The \textbf{free} and \textbf{wired uniform spanning forest} (FUSF and WUSF) measures of $G$, denoted $\FUSF_G$ and $\WUSF_G$, are defined to be the weak limits of the uniform spanning tree measures of $G_i$ and $G_i^*$ respectively. That is, for every finite set $S \subset E$,
\[\FUSF_G\bigl(\{\omega  \subseteq E : S \subseteq \omega \}\bigr) = \lim_{i\to\infty} \UST_{G_i}\bigl(\{\omega  \subseteq E : S \subseteq \omega \}\bigr)\]
and
\[\WUSF_G\bigl(\{\omega  \subseteq E : S \subseteq \omega \}\bigr) = \lim_{i\to\infty} \UST_{G^*_i}\bigl(\{\omega  \subseteq E : S \subseteq \omega \}\bigr).\]
Both limits were proven to exist by Pemantle \cite{Pem91} (although the WUSF was not considered explicitly until the work of H\"aggstr\"om \cite{Hagg95}), and do not depend on the choice of exhaustion. 
% We say that a random spanning forest of $G$ is a free (resp.\ wired) uniform spanning forest of $G$ if it has law $\FUSF_G$ (resp.\ $\WUSF_G$).

Benjamini, Lyons, Peres, and Schramm \cite{BLPS} proved that the WUSF and FUSF of $G$ coincide if and only if $G$ does not admit harmonic functions of finite Dirichlet energy, from which they deduced that the WUSF and FUSF coincide on any amenable transitive graph. In particular, it follows that the WUSF and FUSF coincide for every transitive $d$-dimensional graph, and in this context we refer to both the FUSF and WUSF measures on $G$ as simply the uniform spanning forest measure, $\mathsf{USF}_G$, on $G$. We say that a random spanning forest of $G$ is a uniform spanning forest of $G$ if it has law $\mathsf{USF}_G$. 
% for every finite set $S \subset E$. Similarly, the \textbf{wired uniform spanning forest} measure of $G$ is defined to be the weak limit of the uniform spanning tree measures of $G$. That is,

\subsection{Wilson's Algorithm}
\label{subsec:Wilson}

\textbf{Wilson's algorithm} \cite{Wilson96} is a way of generating the uniform spanning tree of a finite graph by joining together loop-erased random walks. It was extended to generate the wired uniform spanning forests of infinite, transient graphs by Benjamini, Lyons, Peres, and Schramm \cite{BLPS}.

 Recall that, given a path $\gamma = ( \gamma_n )_{n\geq0}$ in a graph $G$ that is either finite or visits each vertex of $G$ at most finitely often, the \textbf{loop-erasure} of $\gamma$ is defined by deleting loops from $\gamma$ chronologically as they are created. The loop-erasure of a simple random walk path is known as \textbf{loop-erased random walk} and was first studied by Lawler \cite{Lawler80}. Formally, we define the loop-erasure of $\gamma$ to be $\mathsf{LE}(\gamma)=( \gamma_{\tau_i} )_{i\geq0}$, where $\tau_i$ is defined recursively by setting $\tau_0=0$ and
\[\tau_{i+1} = 1+ \sup\{t \geq \tau_i : \gamma_t = \gamma_{\tau_i}\}.\]
(If $G$ is not simple, then we also keep track of which edges are used by $\mathsf{LE}(\gamma)$.)

Let $G$ be an infinite, connected, transient, locally finite graph. Wilson's algorithm rooted at infinity allows us to sample the wired uniform spanning forest of $G$ as follows. Let $( v_i )_{i\geq1}$ be an enumeration of the vertices of $G$. Let $\F_0=\emptyset$, and define a sequence of random subforests $( \F_i )_{i\geq 0}$ of $G$ as recursively follows. 
\begin{enumerate}

\item Given $\F_i$, let $X^{i+1}$ be a random walk started at $v_{i+1}$, independent of $\F_i$. 

\item 
Let $T_{i+1}$ be the first time $X^{i+1}$ hits the set of vertices already included in $\F_i$, where $T_{i+1}=\infty$ if $X^{i+1}$ never hits this set. 
% If $X^{i+1}$ hits the set of vertices already included in $\F_i$, stop it the first time it does so 
Note that $T_{i+1}$ will be zero if $v_{i+1}$ is already included in $\F_i$. 
% Otherwise, let $X^{i+1}$ run forever. 

\item Let $\F_{i+1}$ be the union of $\F_i$ with the loop-erasure of the stopped random walk path $( X^{i+1}_n )_{n=0}^{T_{i+1}}$.

\end{enumerate}
Finally, let $\F = \bigcup_{i\geq1}\F_i$. This is Wilson's algorithm rooted at infinity: the resulting random forest $\F$ is a wired uniform spanning forest of $G$.

\subsection{The main connectivity estimate}

Let $K$ be a finite set of vertices of $\bbG$. Following \cite{BeKePeSc04}, we define the \textbf{spread} of $K$, denoted $\langle K \rangle$, to be 
\[\langle K \rangle = \min \Big\{\prod_{\{x,y\}\in E(\tau)} \langle x y \rangle : \tau=(W,E) \text{ is a tree with vertex set $K$}\Big\}. \]
Note that the tree $\tau$ being minimized over in the definition of $\langle K \rangle$ need not be a subgraph of $\bbG$. If we enumerate the vertices of $K$ as $x_1,\ldots,x_n$, then we have the simple estimate \cite[Lemma 2.6]{BeKePeSc04}
\begin{equation}
\label{eq:spread}
\langle K \rangle \asymp \prod_{i=1}^n 
\min \left\{ \langle x_i x_j \rangle : 1 \leq j<i \right\},
\end{equation}
where the implied constant depends on the cardinality of $K$. In practice we will always use \eqref{eq:spread}, rather than the definition, to estimate the spread.

The main tool in our analysis of the USF is the following estimate of BKPS. Recall that $\sF(K)$ is the event that every vertex of $K$ is in the same component of the uniform spanning forest $\F$.

\begin{thm}[BKPS \cite{BeKePeSc04}]
\label{thm:sdim1}
 Let $\bbG$ be a $d$-dimensional transitive graph with $d>4$, let $\F$ be the uniform spanning forest of $G$, and let $K$ be a finite set of vertices of $\bbG$. Then there exists a constant $C=C(\bbG,|K|)$ such that
\begin{equation} \P\big(\sF(K)\big) \leq C \big\langle  K \big\rangle^{-(d-4)}.  \end{equation}
% where the sum is taken over partitions $\cJ$ of $I$, and
 % where the implied constant may depend on the cardinality of $K$. 
\end{thm}

BKPS proved the theorem in the case $\bbG=\Z^d$. The general case follows from the same proof by applying the heat kernel estimates of Hebisch and Saloff-Coste \cite{HebSaCo93} (see \cref{thm:HSCGreen}), as stated in \cite[Remark 6.12]{BeKePeSc04}. These heat kernel estimates imply in particular that the Greens function estimate
\begin{equation}
\label{eq:HSC}
\sum_{n\geq 0} p_n(u,v) \asymp \langle uv \rangle^{-(d-2)}
\end{equation}
holds for every $d$-dimensional transitive graph $\bbG$ with $d>2$ and every pair $u,v \in \bbV$.
% We shall also require the following estimate.

\begin{prop}\label{prop:sdim2}
Let $\bbG$ be a $d$-dimensional transitive graph, let $\F$ be the uniform spanning forest of $G$, and let $K_i$ be a collection of finite sets of vertices of $\bbG$ indexed by some finite set $I$. Then there exists a constant $C=C(\bbG,|I|,\{|K_i|:i \in I\})$ such that
\begin{equation} \P\big(\sF(K_i \cup K_j) \text{ if and only if $i=j$}\big) \leq C \prod_{i \in I} \big\langle  K_i \big\rangle^{-(d-4)}.  \end{equation}
% where the sum is taken over partitions $\cJ$ of $I$, and
 % where the implied constants may depend on the cardinality of $I$ and of the sets $K_i$. 
\end{prop}

\begin{proof}
We may assume that $I=\{1,\ldots,k\}$ for some $k\geq 1$. 
Given a collection of independent random walks $X^1,\ldots,X^n$, let $A(X^1,\ldots,X^n)$ be the indicator of the event that the forest generated by running the first $n$ steps of Wilson's algorithm using the walks $X^1,\ldots,X^n$, in that order, is connected.
% Let $K$ be a finite set. If we generate a sample $\F$ of the USF using Wilson's algorithm, starting with independent random walks $X^1,\ldots,X^n$ started at the vertices of $K$, then $K$ is connected in $\F$ if and only if $I(W$ occurs. 
 % Thus, for every finite set of vertices $K$, 
 Thus, given a finite set $K \subset \bbV$,  we have
\[\P(\sF(K)) = \P\left(A\left(X^1,\ldots,X^{|K|}\right) =1\right)\]
where $X^1,\ldots,X^{|K|}$ are independent random walks started at the vertices of $K$. Now suppose that $(K_i)_{i \in I}$ is a collection of finite sets, and suppose we generate a sample $\F$  of the USF, starting with independent random walks $X^{1,1},\ldots,X^{1,|K_1|},X^{2,1},\ldots,X^{k,|K_k|}$, where $X^{i,j}$ starts from the $j$th element of $K_i$. Then we observe that
\[\left\{\sF(K_i \cup K_j) \text{ if and only if $i=j$} \right\} \subseteq \bigcap_{i \in I} \left\{ A(X^{i,1},\ldots,X^{i,|K_i|}) =1\right\}, \]
and hence that 
\begin{equation}
\label{eq:negdep}\P\big(\sF(K_i \cup K_j) \text{ if and only if $i=j$}\big) \leq \prod_{i \in I}\P(\sF(K_i)).\end{equation}
The claim now follows from \cref{thm:sdim1}. It is also possible to prove \eqref{eq:negdep} using the negative association property of the USF, see e.g.\ \cite{MR2060630}.
\end{proof}

\subsection{Witnesses}

 % We write $G(\cdot,\cdot)$ for the Green's function of $\bbG$, that is, $G(x,y)$ is the expected number of times that a simple random walk on $G$ started at $x$ visits $y$.

% Throughout this section, $\F$, and, for each set $W\subseteq \bbV$, write $\sF(W)$ for the event that the vertices of $W$ are all in the same component of $\F$.

\begin{figure}
\includegraphics[width=0.8\textwidth]{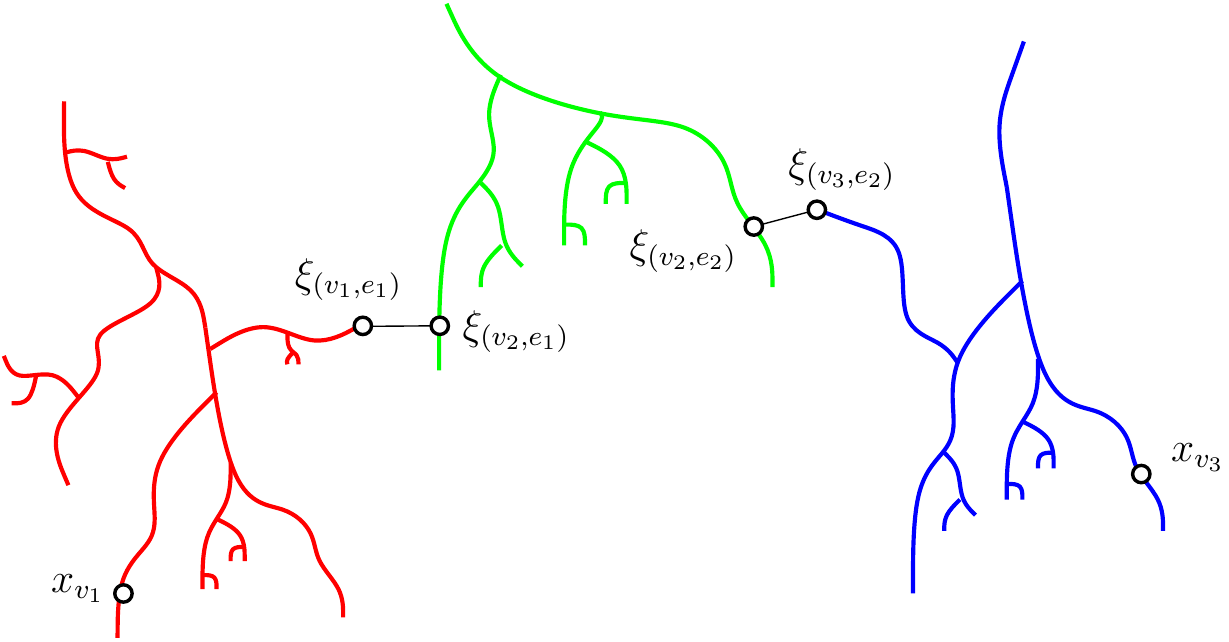}
\caption{Schematic illustration of a witness for the faithful presence of a path of length two with endpoints as boundary points. Let $H$ be the graph with boundary defined by $V(H)=\{v_1,v_2,v_3\}$, $\partial V(H)=\{v_1,v_3\}$ and with edges $e_1=\{v_1,v_2\}$ and $e_2=\{v_2,v_3\}$. The configuration $(\xi_{(v,e)})_{(v,e)\in E_\bullet}$ is a witness for the $1$-faithful presence of $H$ at $(x_{v_1},x_{v_3})$ if 
$\{\xi_{(v_1,e_1)},\xi_{(v_2,e_1)}\}$ and $\{\xi_{(v_2,e_2)},\xi_{(v_3,e_2)}\}$ are edges of $\bbG$ and 
there exist three distinct trees of $\F$ each containing one of the sets $\{x_{v_1},\xi_{(v_1,e_1)}\}$, $\{\xi_{(v_1,e_1)},\xi_{(v_2,e_1)}\}$, and $\{\xi_{(v_3,e_2)}, x_{v_3}\}$.}
\label{fig:witness}
\end{figure}

Let $H$ be a finite hypergraph with boundary, let $r\geq 1$, and let $x=(x_v)_{v\in \partial V}$ be a collection of vertices in $\bbG$. We say that $H$ is $r$-faithfully present at $x$ if it is faithfully present at the components of $x$ in $\cC_r^{hyp}(\F)$. We define $r$-presence of $H$ at $x$ similarly. Let $E_\bullet$ be the set of pairs $(e,v)$, where $e\in E$ is an edge of $H$ and $v\perp e$ is a vertex of $H$ incident to $e$.  We say that $\xi= (\xi_{(e,v)})_{(e,v)\in E_\bullet}\in V^{E_\bullet}$ is a \textbf{witness} for the $r$-faithful presence of $H$ at $x$ if the following conditions hold:
% 
% 
% 
 % For each $(x_v)_{v\in \partial V}$ and $\xi = (\xi_{(e,v)})_{(e,v) \in E_\bullet} \in \bbV^{E_\bullet}$, let $\Witness(x,\xi)$ be the event that the following conditions hold:
 \begin{enumerate}[leftmargin=*]
 \itemsep0.3em
 \item   For every $e \in E$ and every $u,v \perp e$ we have that $\langle \xi_{(e,v)} \xi_{(e,u)} \rangle \leq r-1$.
 \item For each boundary vertex $v \in \partial V$, every point in the set $\{x_v\} \cup \{\xi_{(e,v)} : e\perp v\}$ is in the same component of $\F$,
 \item for each interior vertex $v \in V_\circ$, every point in the set $\{\xi_{(e,v)} : e\perp v\}$ is in the same component of $\F$, and
 \item for any two distinct vertices $u,v \in V$, the components of $\F$ containing $\{\xi_{(e,u)} : e\perp u\}$ and $\{\xi_{(e,v)} : e\perp v\}$ are distinct. 
 \end{enumerate}
 See \cref{fig:witness} for an illustrated example. 
 We write $\Witness(x,\xi)=\Witness^H_r(x,\xi)$ for the event that $\xi$ is a witness for the $r$-faithful presence of $H$ at $x$.
 Thus, 
on the event that all the vertices of $x$ are in distinct components of $\F$, 
 $H$ is $r$-faithfully present at $x$ if and only if $\Witness_r^H(x,\xi)$ occurs for some $\xi\in V^{E_\bullet}$, and is present at $x$ if and only if $\Witness_r^{H'}(x,\xi)$ occurs for some quotient $H'$ of $H$ and some $\xi\in  V^{E_\bullet(H')}$.

We say that $H$ is $r$-\textbf{robustly faithfully present} at $x=(x_v)_{v\in V}$  if there is an infinite collection  $\{ \xi^i = (\xi^i_{(e,v)})_{(e,v)\in E_\bullet} : i \geq 1 \}$ 
such that $\xi^i$ is a witnesses for the $r$-faithful presence of $H$ at $x$ for every $i$, and $\xi^j_{(e,v)} \neq \xi^j_{(e',v')}$ for every $i > j \geq 1$ and $(e,v),(e',v') \in E_\bullet$.

Often, $x$, $r$ and $H$ will be fixed. In this case we will speak simply of  `faithful presence'  to mean `$r$-faithful presence', `robustly faithfully present' to mean `$r$-robustly faithfully present', `witnesses' to mean `witnesses for the $r$-faithful presence of $H$ at $x$', and so on.

It will be useful to define the following sets in which witnesses must live. 
For every $(x_v)_{v\in \partial V}$, $n\geq 0$ and $N > n$, let
\[\Xi_x(n,N) =  \Xi^H_x(n,N) = \left(\Lambda_x(n,N)\right)^{E}\]
and let
 $\Xi_{\bullet x}(n,N)=\Xi^{H,r}_{\bullet x}(n,N)$ be the set
\begin{multline*}\Xi_{\bullet x}(n,N) =\\ \left\{(\xi_{(e,v)})_{(e,v)\in E_\bullet} \in \left(\Lambda_x(n,N)\right)^{E_\bullet} : \, \langle \xi_{(e,v)} \xi_{(e,u)} \rangle \leq r-1 \text{ for every $e \in E$ and every $u,v \perp e$}\right\},\end{multline*}
so that $\xi \in \Lambda_x(n,N)^{E_\bullet}$ is a witness for the faithful presence of $H$ if and only if $\xi \in \Xi_{\bullet x}(n,N)$ and conditions $(2)$, $(3)$, and $(4)$ in the definition of witnesses, above, hold.

% \subsection{Basic preliminaries for hypergraphs with boundary}

% Let $H=(\partial V, V_\circ, E)$ be a finite hypergraph with boundary.

\subsection{Indistinguishability of tuples of trees}

% We now review the \emph{multicomponent indistinguishability theorem} of \cite{MulticomponentIndistinguishability}, which will play a central part in our analysis. 
In this section we provide background on the notion of \emph{indistinguishability theorems}, including the indistinguishability theorem of \cite{MulticomponentIndistinguishability} which will play a major role in the proofs of our main theorems.

Indistinguishability theorems tell us that, roughly speaking, `all infinite components look alike'. The first such theorem was proven in the context of Bernoulli percolation by Lyons and Schramm \cite{LS99}. Indistinguishability of components in uniform spanning forests was conjectured by Benjamini, Lyons, Peres, and Schramm \cite{BLPS} and proven by Hutchcroft and Nachmias \cite{HutNach2016a}. (Partial progress was made independently at the same time by Tim\'ar \cite{timar2015indistinguishability}.) All of the results just mentioned apply to \emph{individual} components. In this paper, we will instead apply the indistinguishability theorem of \cite{MulticomponentIndistinguishability}, which yields a form of indistinguishability for \emph{multiple} components in the uniform spanning forest. 
We will use this theorem as a zero-one law that allows us to pass from an estimate showing that certain events occur with positive probability to knowing that these events must occur with probability one.

% \medskip

% \medskip

We now give the definitions required to state this theorem. 
Let $G=(V,E)$ be a graph, and let $k\geq 1$. 
We define $\Omega_k(G) =\{0,1\}^E \times V^k$, which we equip with its product $\sigma$-algebra and think of as the set of subgraphs of $G$ rooted at an ordered $k$-tuple of vertices. A measurable set $\sA \subseteq \Omega_k(G)$ is said to be a $k$-\textbf{component property} if 
\[ (\omega,(u_i)_{i=1}^k)\in \sA \Longrightarrow (\omega,(v_i)_{i=1}^k)\in \sA \,\,\,
\begin{array}{l}
\text{for all } (v_i)_{i=1}^k \in V^k \text{ such that $u_i$ is} \\\text{connected to $v_i$ in $\omega$ for each $i=1,\ldots,k$}.
\end{array}
\]
That is, $\sA$ is a $k$-component property if it is stable under replacing the root vertices with other root vertices from within the same components. 
Given a $k$-component property $\sA$, we say that a $k$-tuple of components $(K_1,\ldots,K_k)$ of a configuration $\omega \in \{0,1\}^E$ \textbf{has property} $\sA$ if $(\omega,(u_i)_{i=1}^k) \in \sA$ whenever $u_1,\ldots,u_k$ are vertices of $G$ such that $u_i \in K_i$ for every $1 \leq i \leq k$. 

% \medskip

Given a vertex $v$ of $G$ and a configuration $\omega \in \{0,1\}^E$, let $K_\omega(v)$ denote the connected component of $\omega$ containing $v$. We say that a $k$-component property $\sA$ is a \textbf{tail} $k$-component property if 
\[ (\omega,(v_i)_{i=1}^k)\in \sA \Longrightarrow (\omega',(v_i)_{i=1}^k)\in \sA \,\,\,
\begin{array}{l}
\forall \omega' \in \{0,1\}^E \text{ such that } \omega \symdif \omega' \text{ is finite and }\\
 K_\omega(v_i)\symdif K_{\omega'}(v_i) \text{ is finite for every $ i =1,\ldots,k$,}
\end{array}
\]
where $\symdif$ denotes the symmetric difference. 
In other words, tail multicomponent properties are stable under finite modifications to $\omega$ that result in finite modifications to each of the components of interest $K_\omega(v_1),\ldots,K_\omega(v_k)$.

\begin{thm}[\cite{MulticomponentIndistinguishability}]
\label{thm:indist}
Let $\bbG$ be a $d$-dimensional transitive graph with $d>4$ and with vertex set $\bbV$, and let $\F$ be the uniform spanning forest of $\bbG$.
 Then for each $k\geq 1$ and each tail $k$-component property $\sA \subseteq \Omega_k(\bbG)$, either every $k$-tuple of distinct connected  components of $\F$ has property $\sA$ almost surely or no $k$-tuple of distinct connected  components of $\F$ has property $\sA$ almost surely.
\end{thm}

We say that $\sA$ is a \textbf{multicomponent property} if it is a $k$-component property for some $k\geq 1$. 
For our purposes, the key example of a tail multicomponent property is the property that some finite hypergraph with boundary $H$ is $r$-robustly faithfully present at $(x_v)_{v\in \partial V}$. Applying \cref{thm:indist}, we will deduce that if $H$ is $r$-robustly faithfully present at some $(x_v)_{v\in \partial V}$ with positive probability then it must be almost surely $r$-robustly faithfully present at \emph{every} $(x_v)_{v\in \partial V}$ for which  the vertices $\{x_v\}_{v\in \partial V}$ are all in distinct components of $\F$.

\subsection{Optimal Coarsenings}
\label{subsec:optimalcoarsenings}

In this section we study the $\min\max$ problem appearing in \cref{thm:main,thm:mainhyper}, proving the following.

\begin{lemma}
	\label{lem:maxminswap}
	Let $H$ be a finite hypergraph with boundary and let $d\geq 4$. Then
	\begin{multline} 
		\label{eq:maxminswap}
	\max\left\{ \min \left\{\eta_d (H'') : H'' \text{ is a coarsening of $H'$}\right\} : H' \text{ is a subhypergraph of } H\right\} =\\
	\min\left\{ \max \left\{\eta_d (H'') : H'' \text{ is a subhypergraph of $H'$}\right\} : H' \text{ is a coarsening of } H\right\}.
	\end{multline}
	In particular, $H$ has a coarsening all of whose subhypergraphs are $d$-buoyant if and only if every subhypergraph of $H$ has a $d$-buoyant coarsening.
\end{lemma}

Given a hypergraph with boundary $H=(\partial V, V_\circ, E,\perp)$ and an equivalence relation $\bowtie$ on $E$, we can form a coarsening $\coarse{H}{\bowtie}$ of $H$ by taking $\partial V(\coarse{H}{\bowtie})=\partial V(H)$ and $V_\circ (\coarse{H}{\bowtie})=V_\circ(H)$, taking $E(H/\bowtie)$ to be the set of equivalence classes of $\bowtie$, and defining 
\[
\perp\!(H/\bowtie) = \Bigl\{(v,[e]) : v\in V,\, [e] \in E(H/\bowtie), \text{ and } \exists f\in E \text{ such that } [f]=[e] \text{ and } v \perp f\Bigr\},
\]
where $[e]$ denotes the equivalence class of $e$ under $\bowtie$. 
It is easily seen that every coarsening of $H$ can be uniquely represented in this way. We say that a coarsening $\coarse{H}{\bowtie}$ of a hypergraph with boundary $H$ is \textbf{proper} if there exist at least two non-identical edges of $H$ that are related under $\bowtie$.

Let $H=(\partial V, V_\circ,E,\perp)$ be a finite hypergraph with boundary. 
We say that a subhypergraph $H'$ of $H$ is \textbf{subordinate} to an equivalence relation $\bowtie$ on $E$ if every equivalence class of $\bowtie$ is either contained in or disjoint from the edge set of $H'$.  
Given an equivalence relation $\bowtie$ on $E$ and a subhypergraph $H'$ of $H$, we write $\coarse{H'}{\bowtie}$ for the coarsening $\coarse{H'}{\,(\,\bowtie |_{E'})}$, where $\bowtie|_{E'}$ is the restriction of $\bowtie$ to the edge set of $H'$.
The function $H' \mapsto \coarse{H'}{\bowtie}$ is a bijection from subhypergraphs of $H$ subordinate to $\bowtie$ to subhypergraphs of $\coarse{H}{\bowtie}$.
% surjective function from subhypergraphs of $H$ to subhypergraphs of $\coarse{H}{\bowtie}$, and is bijective when restricted to subhypergraphs of $H$ subordinate to $\bowtie$.
% 
We say that an equivalence relation $\bowtie$ on $E$ is $d$\textbf{-optimal} if 
\[ \eta_d(\coarse{H}{\bowtie}) = \min \bigl\{\eta_d(\coarse{H}{\bowtie'}) : \text{$\coarse{H}{\bowtie'}$ a coarsening of $H$} \bigr\}. \]
 We call a coarsening $H'=\coarse{H}{\bowtie}$ of $H$ $d$-optimal if $\bowtie$ is $d$-optimal. We say that a subhypergraph $H'=(\partial V',V_\circ',E',\perp')$ of $H$ is \textbf{full} if $\{v\in V: v \perp e\}\subseteq V'$ for every $e \in E'$.

\begin{lemma}
\label{lem:optimal}
Let $H$ be a finite hypergraph with boundary, let $d\in \R$, and let $\coarse{H}{\bowtie}$ be a $d$-optimal coarsening of $H$. Then  $\coarse{H'}{\bowtie}$ is a  $d$-optimal coarsening of $H'$ for every full subhypergraph $H'$ of $H$ subordinate to $\bowtie$.
\end{lemma}

Recall from \cref{subsec:introgeneral} that $\eta_d(H)$ is defined to be $(d-4)\Delta(H)-d|E|-(d-4)|V_\circ|$.

\begin{proof}

Let $H'$ be a subhypergraph of $H$ subordinate to $\bowtie$. Let $\bowtie'$ be an equivalence relation on $E'$, and let $\bowtie''$ be the equivalence relation on $E$ defined by
\begin{equation*}
e  \, \bowtie'' \, e' 
% \iff (\text{$e,e' \in E\setminus E'$ and $e = e'$}) \text{ or } (\text{$e,e' \in  E'$ and $e \bowtie' e'$})\\
\iff (\text{$e,e' \in E\setminus E'$ and $e \bowtie e'$}) \text{ or } (\text{$e,e' \in  E'$ and $e \bowtie' e'$}).
\end{equation*}
% where the equivalence of the two definitions follows since $H'$ is subordinate to $\bowtie$.
Thus, $H'$ is subordinate to $\bowtie''$ and $\coarse{H'}{\bowtie''}=\coarse{H'}{\bowtie'}$. It is easily verified that, since $\bowtie$ and $\bowtie''$ differ only on edges of $H'$ and $H'$ is subordinate to $\bowtie$,
\begin{align*}
 |V_\circ(\coarse{H'}{\bowtie'})| - |V_\circ(\coarse{H'}{\bowtie})| &= |V_\circ(\coarse{H}{\bowtie''})| - |V_\circ(\coarse{H}{\bowtie})|,\\
 |E(\coarse{H'}{\bowtie'})| - |E(\coarse{H'}{\bowtie})| &= |E(\coarse{H}{\bowtie''})| - |E(\coarse{H}{\bowtie})|,  \quad \text{ and }\\
\Delta(\coarse{H'}{\bowtie'}) - \Delta(\coarse{H'}{\bowtie}) &= \Delta(\coarse{H}{\bowtie''}) - \Delta(\coarse{H}{\bowtie}).
\end{align*}
 We deduce that, since $\coarse{H}{\bowtie}$ is $d$-optimal,
\[\eta_d(\coarse{H'}{\bowtie'}) - \eta_d(\coarse{H'}{\bowtie}) = \eta_d(\coarse{H}{\bowtie''}) - \eta_d(\coarse{H}{\bowtie}) \geq 0,
\vspace{0.2em}
\]
% where the final inequality follows .
and the result follows  since $\bowtie'$ was arbitrary. \qedhere

\end{proof}

\begin{lem}
\label{lem:optimalisbest}
Let $H$ be a finite hypergraph with boundary, let $d\geq 4$, and let $\coarse{H}{\bowtie}$ be a $d$-optimal coarsening of $H$. Then
\begin{multline}
\label{eq:optimalisbest}
\max \left\{ \min \left\{\eta_d(\coarse{H'}{\bowtie'}) : \coarse{H'}{\bowtie'} 
\text{  a coarsening of $H'$}\right\}: H'
\text{ a subhypergraph of $H$} \right\}\\
=
\max\left\{\eta_d(\coarse{H'}{\bowtie}) : H' \text{ a subhypergraph of $H$} \right\}.
% \geq \min \left\{ \max\left\{\eta_d(\coarse{H'}{\bowtie'}) : H' \text{ a subhypergraph of $H$} \right\} : \coarse{H}{\bowtie'} \text{ a coarsening of $H$}\right\}.
\end{multline}
\end{lem}

\begin{proof}
Let $\coarse{H}{\bowtie}$ be a $d$-optimal coarsening of $H$. 
Let $H'$ be a subhypergraph of $H$, and let $H''$ be the smallest full subhypergraph of $H$ that contains $H'$ and is subordinate to $\bowtie$. That is, $H''$ is obtained from $H'$ by adding every edge of $H$ that is contained in equivalence class of $\bowtie$ intersecting $E'$  and every vertex of $H$ that is incident to either an edge of $H'$ or one of these added edges. 
% In particular, $\coarse{H''}{\bowtie}$ is the smallest full subhypergraph of $\coarse{H}{\bowtie}$ that contains $\coarse{H'}{\bowtie}$. 
% Thus, w
Writing $\deg_{\coarse{H''}{\;\bowtie}}(v)$ for the degree of a vertex in $\coarse{H''}{\bowtie}$, we compute that
\[\eta_d(\coarse{H''}{\bowtie}) - \eta_d(\coarse{H'}{\bowtie}) = (d-4)\sum_{v\in V(H'') \setminus V(H') } \left(\deg_{\coarse{H''}{\;\bowtie}}(v) -1\right) \geq 0.\]
It follows that
\begin{multline}
\label{eq:optimalisbest2}
\min\{ \eta_d(\coarse{H'}{\bowtie'}) : \text{$\coarse{H'}{\bowtie'}$ a coarsening of $H'$}\} \leq 
\eta_d(\coarse{H'}{\bowtie}) \leq \eta_d(\coarse{H''}{\bowtie})\\ = \min \{ \eta_d(\coarse{H''}{\bowtie'}) : \coarse{H''}{\bowtie'} \text{ is a coarsening of $H''$}\},
\end{multline}
where the equality on the second line follows from \cref{lem:optimal}. Taking the maximum over $H'$, we obtain that
\begin{align*}
\max &\left\{ \min \left\{\eta_d(\coarse{H'}{\bowtie'}) : \coarse{H'}{\bowtie'} 
\text{  a coarsening of $H'$}\right\}: H'
\text{ a subhypergraph of $H$} \right\}\\ 
&\leq 
\max \left\{ \eta_d(\coarse{H'}{\bowtie}) 
: H'
\text{ a subhypergraph of $H$} \right\}\\
% &= \max \left\{ \eta_d(\coarse{H'}{\bowtie}) 
% :\begin{array}{l}\text{$H'$ a subhypergraph of}\\ \text{$H$ subordinate to $\bowtie$} \end{array} \right\}\\
&\leq
\max \left\{ \min \left\{\eta_d(\coarse{H''}{\bowtie'}) : \coarse{H''}{\bowtie'} 
\text{  a coarsening of $H''$}\right\}: 
\begin{array}{l}\text{$H''$ a subhypergraph of}\\ \text{$H$ subordinate to $\bowtie$} \end{array}\right\},
\end{align*}
where the second equality follows from \eqref{eq:optimalisbest2}. 
The the final line of this display is clearly less than or equal to the first line, so that all the lines must be equal, completing the proof.
% Since $H'$ is $d$-pessimal these inequalities must be equalities, and the result follows easily.
% The claimed equality \eqref{eq:optimalisbest} follows immediately by taking the maximum over $H$ in \eqref{eq:optimalisbest2}.
\end{proof}
\begin{proof}[Proof of \cref{lem:maxminswap}]
It follows immediately from \cref{lem:optimalisbest} that
\begin{multline*}
\max \left\{ \min \left\{\eta_d(\coarse{H'}{\bowtie'}) : \coarse{H'}{\bowtie'} 
\text{  a coarsening of $H'$}\right\}: H'
\text{ a subhypergraph of $H$} \right\}\\
\geq \min \left\{ \max\left\{\eta_d(\coarse{H'}{\bowtie'}) : H' \text{ a subhypergraph of $H$} \right\} : \coarse{H}{\bowtie'} \text{ a coarsening of $H$}\right\},
\end{multline*}
and the reverse inequality is trivial. \qedhere

\end{proof}

\begin{remark}
\cref{lem:optimalisbest} yields a brute force algorithm for computing the value of the relevant $\max \min$ problem that is exponentially faster than the trivial brute force algorithm, although still taking superexponential time in the number of edges of $H$.  
\end{remark}

\section{Sketch of the proof}
\label{sec:notationsandoutline}

In this section we give a detail-free overview of the most important components of the proof. This section is completely optional; all the arguments and definitions mentioned here will be repeated in full detail later on.

% \begin{itemize}
	% \item

	\subsection{Non-ubiquity in high dimensions}
	\label{subsec:nonubiqsketch}

	Let $\bbG$ be a $d$-dimensional transitive graph, let $H$ be a finite hypergraph with boundary, and let $\F$ be the uniform spanning forest of $\bbG$. 
	We wish to show that 
	if every coarsening of $H$ has a subhypergraph that is not $d$-buoyant, then 
	$H$ is not faithfully ubiquitous in $\cC^{hyp}_r(\F)$ for any $r \geq 1$ a.s. By \cref{lem:maxminswap}, this condition is equivalent to there existing a subhypergraph of $H$ none of whose coarsenings are $d$-buoyant. If $H$ is faithfully ubiquitous then so are all of its subhypergraphs, and so it suffices to consider the case that $H$ does not have any $d$-buoyant coarsenings, i.e., that $\bareta_d(H) >0$.

	 To show that $H$ is not faithfully ubiquitous, it would suffice to show that if the vertices $x=(x_v)_{v\in \partial V}$  are far apart from each other, then the expected total number of witnesses for the faithful presence of $H$ at $x$ is small. As it happens, we are not able to control the total number of witnesses without making further assumptions on $H$. 
	 Nevertheless, the most important step in our argument is to show
	that if $x$ is contained in $\Lambda_x(0,n-1)$, then the expected number of witnesses in $\Lambda_x(n,n+1)$ is exponentially small as a function of $n$.  
	 Once we have done this, we will control the expected number of witnesses that occur `at the same scale' as $x$  by a similar argument. We are not finished at this point, of course, since we have not ruled out the existence of witnesses that are spread out across multiple scales. However, given the single-scale estimates, we are able to handle multi-scale witnesses of this form via an inductive argument on the size of $H$ (\cref{lem:inductionestimate,lem:firstmoment2,lem:firstmoment3}), which allows us to reduce from the multi-scale setting to the single-scale setting.

	 % \medskip

	 Let us briefly discuss how the single-scale estimate is attained. Write $\Xi = \Xi_x(n,n+1)$. \cref{prop:sdim2} implies that the expected number of witnesses in $\Lambda_x(n,n+1)$ is at most a constant multiple of 
	 \[
	 \sum_{\xi \in \Xi} 
	 \prod_{u \in \partial V} W(x,\xi),\]
	 where 
	 \[ W(x,\xi) = \prod_{u\in \partial V}\langle x_u, \{\xi_e: e \perp u\} \rangle^{-(d-4)} \prod_{u \in V_\circ} \langle \{\xi_e: e \perp u\} \rangle^{-(d-4)}. \]
	 To control this sum, we split it as follows. Let $L$ be the set of symmetric functions $\ell:E^2 \to \{0,\ldots,n\}$ such that $\ell(e,e)=0$ for every $e\in E$. 
	% such that $\ell(\bullet,\bullet)=0$, $\ell(e,\bullet) \geq n$ and  $\ell(e,e)=0$ for every $e\in E$.
 	For each $\ell \in L$, let 
	\[\Xi_\ell = \left\{\xi \in \Xi :
	\begin{array}{l}
	% 2^{\ell(e,\bullet)-1} \leq \langle \xi_e x_1 \rangle \leq 2^{\ell(e,\bullet)} \text{ for all $e\in E$ and }\\  
	 2^{\ell(e,e')} \leq \langle \xi_e \xi_{e'} \rangle \leq 2^{\ell(e,e')+2} \text{ for all $e,e' \in E$}
	\end{array}
	 \right\},\]
	 so that $\Xi = \bigcup_{\ell \in L} \Xi_\ell$. The advantage of this decomposition is that $W$ is approximately constant on each set $\Xi_\ell$:
	 \[\log_2 W(x,\xi) \approx  -(d-4)|\partial V| n  -(d-4)\sum_{i=1}^{|E|}\sum_{u \perp e}\min \left\{\ell(e_i,e_j) : j<i,\, e_j \perp u\right\}\]
	 for every $\xi \in \Xi_\ell$.

	On the other hand, by considering the number of choices we have for $\xi_{e_i}$ at each step given our previous choices, it follows that
	\begin{align}\log_2 |\Xi_\ell| \lesssim  
	dn+ d\sum_{i=1}^{|E|}\min\left\{\hat \ell(e_i,e_j) : j<i\right\},
	\end{align}
	where $\hat \ell$ is the largest ultrametric on $E$ that is dominated by $\ell$. ($\Xi_\ell$ could be much smaller than this of course -- it could even be empty.) We deduce that
	\begin{multline*}
	 \sum_{\xi \in \Xi} 
	  W(x,\xi) \preceq \exp_2\left( dn-(d-4)|\partial V|n\right) \\ \cdot \sum_{\ell \in L} \exp_2 \left[ d\sum_{i=1}^{|E|}\min\left\{\hat \ell(e_i,e_j) : j<i\right\}    -(d-4)\sum_{i=1}^{|E|}\sum_{u \perp e}\min \left\{\ell(e_i,e_j) : j<i,\, e_j \perp u\right\} \right]
	 \end{multline*}
	 and hence that
	 \begin{multline*}
	 \log_2 \sum_{\xi \in \Xi} 
	  W(x,\xi) \lesssim \log_2|L| + dn-(d-4)|\partial V|n
	 \\ + \max_{\ell \in L} \left[ d\sum_{i=1}^{|E|}\min\left\{\hat \ell(e_i,e_j) : j<i\right\}    -(d-4)\sum_{i=1}^{|E|}\sum_{u \perp e}\min \left\{\ell(e_i,e_j) : j<i,\, e_j \perp u\right\} \right].
	 \end{multline*}
	 We have that $\log_2 |L| = E^2 \log_2(n+1)$, which will be negligible compared with the rest of the expression in the case that $\bareta_d(H) >0$.   From here, the problem is to identify the $\ell \in L$ achieving the maximum above. We will argue, by invoking a general lemma (\cref{lem:ultrametric1}) about optimizing linear combinations of minima of distances on the ultrametric polytope, that there is an $\ell \in L$ maximizing the expression such that $\ell$ is an ultrametric and $\ell(e,e') \in \{0,n\}$ for every $e,e' \in E$. The set of such functions $\ell$ are in bijection with the set of coarsenings $H'$ of $H$, where two edges of $H$ are identified in $H'$ if and only if $\ell(e,e')=0$. Choosing such a coarsening optimally, it is not hard to deduce that
	 \[  \log_2 \sum_{\xi \in \Xi} 
	  W(x,\xi) \lesssim  - \bareta_d(H)\, n + |E|^2\log_2 n, \]
	  giving the desired exponential decay.

	  \subsection{Ubiquity in low dimensions}

	  We now sketch the proof of ubiquity in low dimensions. Here we will only discuss the case in which $d/(d-4)$ is not an integer (i.e., $d\notin \{5,6,8\}$). The case in which $d/(d-4)$ is an integer raises several additional technical complications, see \cref{sec:speciald}. 

	  Let $\bbG$ be a $d$-dimensional transitive graph with $d\in \{7\} \cup \{9,10,\ldots\}$, let $H$ be a finite hypergraph with boundary, and let $\F$ be the uniform spanning forest of $\bbG$. Recall the definition of $R_\bbG(H)$ from \cref{subsec:introgeneral}. 
	  Working in the opposite direction to the previous subsection, we wish to prove that if $H$ has a coarsening all of whose subhypergraphs are $d$-buoyant, then $H$ is faithfully ubiquitous in the component hypergraph $\cC^{hyp}_{r}(\F)$ for every $r \geq R_\bbG(H)$ a.s.
	   We say that $H$ is $r$-\textbf{robustly faithfully present} at $x=(x_v)_{v\in V}$  if there are infinitely many disjoint witnesses for the faithful presence of $H$ at $x$. The event that $H$ is $r$-robustly faithfully present at $x$ is a tail $|\partial V|$-component property. Thus, by \cref{thm:indist}, it suffices to prove that there exists an $x$ such that, with positive probability, the points of $x$ are all in different components of $\F$ and $H$ is $R_\bbG(H)$-robustly faithfully present at $x$. 

	   Let us suppose for now that every subhypergraph of $H$ is $d$-buoyant
	    (i.e., that we do not have to pass to a coarsening for this to be true). 
	   To prove that $H$ has a positive probability of being robustly faithfully present at some $x$, we perform a first and second moment analysis on the number of witnesses in dyadic shells. Suppose that $x$ is contained in $\Lambda_x(0,n-1)$. Since we are now interested in existence rather than nonexistence, we can make things easier for ourselves by considering only  $\xi$ that are both contained in a dyadic shell $\Lambda_x(n,n+1)$, and such that $\langle \xi_{(e,u)} \xi_{(e',u')} \rangle \geq 2^{n-C_1}$ whenever $e\neq e'$, for some appropriate chosen constant $C_1$. Furthermore, for each $e\in E$ the points $\{ \xi_{(e,u)} : u \perp e\}$  must be sufficiently well separated that there are not local obstructions to $\xi$ being a witness -- this is where we need that $r \geq R_\bbG(H)$. Call such a $\xi$ \textbf{good}, and denote the set of good $\xi$ by $\Omega_x(n)$. We then argue that for good $\xi$, the probability that $\xi$ is a witness is comparable to
	   \begin{multline*} 
	   		W(x,\xi) = \prod_{u \in \partial V} \langle x_u, \{\xi_e: e \perp u\} \rangle^{-(d-4)} \prod_{u \in V_\circ} \langle \{\xi_e: e \perp u\} \rangle^{-(d-4)}\\ \asymp \exp_2 \left[ -(d-4)(\Delta -|V_\circ|) \, n \right],
	   \end{multline*}
	   where $\xi_e$ is chosen arbitrarily from $\{\xi_{(e,u)} : u \perp e\}$ for each $e$, and hence that
	   % \[ \log_2 \sum_{\xi \in \Omega_x(n)} W(x,\xi) \]
	   the expected number of witnesses in $\Omega_x(n)$ is comparable to $2^{-\eta_d(H) n}$.
	   In other words, we have that the upper bound on the probability that $\xi$ is a witness provided by \cref{prop:sdim2}  is comparable to the true probability when $\xi$ is good. Our proof of this estimate appears in \cref{Sec:technical}; unfortunately it is quite long.
	    % (In fact, we prove a more general statement about \emph{constellations}, see below.) 

	   Taking this lower bound on trust for now, the rest of the analysis proceeds similarly to that sketched in \cref{subsec:nonubiqsketch}, and is in fact somewhat simpler thanks to our restriction to good configurations. The bound implies that the expected number of good witnesses in $\Lambda_x(n,n+1)$ is  comparable to
	   $\exp_2\left[ - \eta_d(H) \, n \right]$. Estimating the second moment is equivalent to estimating the expected number of pairs $\xi,\zeta$ such that $\xi$ and $\zeta$ are both good witnesses. Observe that if $\xi$ and $\zeta$ are both good witnesses then the following hold:
	   \begin{enumerate}
		   	\item For each $v \in V$, there is at most one $v' \in V$ such that $\xi_{(e,v)}$ and  $\zeta_{(e',v')}$ 
	   	 % \rangle \leq 2^{n-3}$. 
	   	 are in the same component of $\F$ for some (and hence every) $e \perp v$ and $e' \perp v'$. 
	   	    	\item For each $e \in E$, there is at most one $e'\in E$ such that $\langle \xi_e \zeta_{e'} \rangle \leq 2^{n-C_1-1}$. 
	   \end{enumerate}
	   To account for the degrees of freedom given by (1), we define $\Phi$  to be the set of functions $\phi: V_\circ \to V_\circ \cup \{\star\}$  such that the preimage of $\phi^{-1}(v)$ has at most one element for each $v\in V_\circ$. 
	   (Here and elsewhere, we use $\star$ as a dummy symbol so that we can encode partial bijections by functions.)
	   For each $\phi \in \Phi$, we define $\tilde \sW_\phi(\xi,\zeta)$ to be the event that $\xi$ and $\zeta$ are both witnesses, and that $\xi_{(e,v)}$ and  $\zeta_{(e',v')}$ 
	   	 % \rangle \leq 2^{n-3}$. 
	   	 are in the same component of $\F$ if and only if $e'=\phi(e)$.
	   	Thus, to control the expected number of pairs of good witnesses, it suffices to control
	   	 \[\sum_{\phi \in \Phi} \sum_{\xi,\zeta \text{ good}} \P\left(\tilde \sW_\phi(\xi,\zeta)\right) \preceq \max_{\phi \in \Phi} \sum_{\xi,\zeta \text{ good}} \P\left(\tilde \sW_\phi(\xi,\zeta)\right).\]
	   	Next, to account for the degrees of freedom given by (2), 	we define $\Psi$  to be the set of functions $\psi: E \to E \cup \{\star\}$  such that the preimage $\psi^{-1}(e)$ has at most one element for each $e\in E$.    For each $\psi \in \Psi$ and $k = (k_e)_{e \in E} \in \{0,\ldots,n\}^{E}$, let
\begin{multline*}\Omega^{\psi,k} = \\\left\{(\xi,\zeta) \in (\Omega_x(n))^2 : 
\begin{array}{l} 2^{n-k_e} \leq \langle \zeta_e \xi_{\psi(e)} \rangle \leq 2^{n-k_e+2} \text{ for all $e\in E$ such that $\psi(e) \neq \star$,} 
\vspace{0.3em} \\
 \text{and }\langle \zeta_e \xi_{e'} \rangle \geq 2^{n-C_1-2} \text{ for all $e,e'\in E$ such that $e' \neq \psi(e)$}
\end{array}
\right\}.
\end{multline*}
We can easily upper bound the volume
\begin{equation*}\log_2|\Omega^{\psi,k}(x)| \lesssim 2d|E|n - d\sum_{\psi(e)\neq\star} k_e. \end{equation*}
Using this together with \cref{prop:sdim2}, is is straightforward to calculate that
\begin{multline*}
\log_2 \sum_{(\xi,\zeta) \in \Omega^{\psi,k}}\P\left(\tilde \sW_\phi(\xi,\zeta)\right) \lesssim
-2\eta_d(H)\, n - (d-4)|\{ u \in V_\circ : \phi(u) \neq \star\}|\,n
 \\+ (d-4)\sum_{\psi(e)\neq \star}|\{u \perp e : \phi(u) \perp \psi(e)\}|k_e - d \sum_{\psi(e)\neq\star}k_e
  \end{multline*}
 for every $\phi\in \Phi$, $\psi \in \Psi$ and $k \in \{0,\ldots,n\}^E$. 
% \end{itemize}

We now come to some case analysis. Observe that for every $\psi\in \Psi$ and $e\in E$, we have that
\begin{multline*}
\sum_{k_e=0}^n \exp_2\left[(d-4)|\{u\perp e:\phi(u)\perp \psi(e)\}| - d\right]k_e  \\\preceq \begin{cases}
\exp_2\left[(d-4)|\{u\perp e:\phi(u)\perp \psi(e)\}| - d\right]n & \text{ if $(d-4)|\{u\perp e:\phi(u)\perp \psi(e)\}|  >d$}\\
n & \text{ if $(d-4)|\{u\perp e:\phi(u)\perp \psi(e)\}|  =d$}\\
1 & \text{ if $(d-4)|\{u\perp e:\phi(u)\perp \psi(e)\}| <d$}.
\end{cases}
\end{multline*}
Since $d/(d-4)$ is not an integer, the middle case cannot occur and we obtain that
\begin{multline*}
\log_2 \sum_{k}\sum_{(\xi,\zeta) \in \Omega^{\psi,k}} \P(\tilde \sW_\phi(\xi,\zeta))
\lesssim
 -2\eta_d(H)\, n - (d-4)|\{ u \in V_\circ : \phi(u) \neq \star\}|\,n \\
 % \ni v \text{ for some $e \perp v$}\}|
 % \\+ (d-4)\sum_{e\in E}\mathbbm{1}[\psi(e) \neq e,\, \text{ $\psi(e)$ shares an endpoint with $e$}]k_e
% + (d-4)\sum_{e}\mathbbm{1}[\psi(e) = e]\deg(e) -d\sum_{e}\mathbbm{1}[\psi(e) \neq \star].
+\sum_{e}\left[(d-4)|\{u \perp e : \phi(u) \perp \psi(e)\}|-d\right]
\mathbbm{1}\left(|\{u \perp e : \phi(u) \perp \psi(e)\}| >  d/(d-4) \right) n.
% - (d-4)|\{ u \in V_\circ : \phi(v) \neq \star\}| n.
 \end{multline*}
 From here, our task is to show that the expression on the right hand side is maximized when $\phi \equiv \star$ and $\psi \equiv \star$, in which case it is equal to $-2d \eta_d(H) n$. To do this, we identify optimal choices of $\phi$ and $\psi$ with subhypergraphs of $H$, and use the assumption that every subhypergraph of $H$ is $d$-buoyant. This should be compared to how, in the proof of non-ubiquity sketched in the previous subsection, we identified optimal choices of $\ell$ with coarsenings of~$H$.

 Once we have this,  since there are only a constant number of choices for $\phi$ and $\psi$, we deduce that the second moment of the number of good witnesses is comparable to the square of the first moment. Thus, it follows from the Cauchy-Schwarz inequality that the probability of there being a good witness in each sufficiently large  dyadic shell is bounded from below by some $\eps>0$, and we deduce from Fatou's lemma that there are good witnesses in infinitely many dyadic shells with probability at least $\eps$. This completes the proof that robust faithful presence occurs with positive probability. 

It remains to remove the simplifying assumption we placed on $H$, i.e., to allow ourselves to pass to a coarsening of $H$ all of whose subhypergraphs are $d$-buoyant before proving faithful ubiquity. To do this,  
we introduce the notion of \emph{constellations of witnesses}. These are larger collections of points, defined in such a a way that every constellation of witness for $H$ contains a witness for each refinement of $H$.  In the actual, fully detailed proof we will work with constellations from the beginning. This does not add many complications.

% The more serious assumption was that $H$ does not contain any edges of degree $d/(d-4)$. Note that this is only a potential problem when $d= 5, 6,$ or $8$. In this case, the second moment of the number of good witnesses in a dyadic shell can differ from the square of the first moment by a polynomial factor (i.e., a logarithm in the exponent). Fortunately, it suffices to consider the case that $H$ is a single edge of degree $d/(d-4)$, since the ubiquity of other hypergraphs containing such an edge follows from this case together with the previous case in which these edges were not present. This final case of a single edge of degree $d/(d-4)$ can be handled by controlling both the second moment of the number of good witnesses in a single dyadic shell and the correlation between different shells, and then applying the second moment method to the number of witnesses in $\bigcup_{k=n}^{2n} \Lambda_x(2k,2k+1)$. 

\section{Moment Estimates}

\label{sec:moments}

\subsection{Non-ubiquity in high dimensions}\label{sec:1stupper}

The goal of this section is to prove the following. 

\begin{prop}\label{prop:nonubiquity}
	Let $\bbG$ be a $d$-dimensional transitive graph with $d>4$, let $\F$ be the uniform spanning forest of $G$, let $H$ be a finite hypergraph with boundary, and let $r\geq 1$. Then the following hold:
	\begin{enumerate}[leftmargin=*]
	\itemsep0.2em
		\item 
	If $H$ has a subhypergraph that does not have any $d$-buoyant coarsenings,
	% \[
		 % \max\left\{ \bareta_d(H') : H' \text{ is a subhypergraph of $H$}\right\}> 0
	% \]
	then $H$ is  not faithfully ubiquitous in $\Comp^{hyp}_r(\F)$ almost surely. 
	\item
	If every quotient $H'$ of $H$ such that $R_\bbG(H') \leq r$ has a subhypergraph that does not have any $d$-buoyant coarsenings, 
	% \[ \min\left\{\max\left\{ \bareta_d(H'') : H'' \text{ is a subhypergraph of $H'$}\right\} : H' \text{ is a quotient of $H$}\right\}> 0\]
	then $H$ is  not  ubiquitous in $\Comp^{hyp}_r(\F)$ almost surely.
	\end{enumerate}
\end{prop}

Let $\bbG$ be a $d$-dimensional graph with $d > 4$, and let $\F$ be the uniform spanning forest of $\bbG$. Let $H=(\partial V,V_\circ,E)$ be a finite hypergraph with boundary such that $E \neq \emptyset$, and let $r\geq 1$. Recall that $\sW(x,\xi)$ is defined to be the event that $\xi$ is a witness for the faithful presence of $H$ at $x$. 
For each $N> n$, we define 
\[
	 S^H_x(n,N) = \sum_{\xi \in \Xi_{\bullet x}(n,N)} \mathbbm{1}\left[\Witness(x,\xi)\right].
\]
For each $(\xi_e )_{e\in E} \in \bbV^{E}$, we also define
\[
	 W^H(x,\xi) = \prod_{u \in \partial V} \langle x_u, \{\xi_e: e \perp u\} \rangle^{-(d-4)} \prod_{u \in V_\circ} \langle \{\xi_e: e \perp u\} \rangle^{-(d-4)} 
\]
and
\[
	\mathbb{W}^H_x(n,N)  = \sum_{\xi\in \Xi_x(n,N)} W^H(x,\xi), 
\]
so that, if we choose a vertex $u(e) \perp e$ arbitrarily for each $e\in E$ and set $(\xi_e)_{e\in E} = (\xi_{(e,u(e))})_{e\in E}$, it follows from \cref{prop:sdim2} that
\[
	\E\left[ S^H_{x}(n,N)\right] 
	=
	\sum_{\xi\in \Xi_{\bullet x}(n,N)} \P(\Witness(x,\xi)) \preceq \sum_{\xi\in \Xi_{\bullet x}(n,N)^{E}} W^H(x,\xi) = \bbW^H_{x}(n,N)
\]
for every $x$, $n$, and $N$.

To avoid trivialities, in the case that $H$ does not have any edges we define $\bbW^H_x(n,N)=1$ for every $x\in \bbV^{\partial V}$ and $N>n$. 

\medskip

In order to prove \cref{prop:nonubiquity}, it will suffice to show that if $H$ has a subhypergraph with boundary that does not have any $d$-buoyant coarsenings, then for every $\eps>0$ there exists a collection of vertices $(x_u)_{u\in \partial V}$ such that all the vertices $x_u$ are in a different component of $\F$ with probability at least $1/2$ (which, by \cref{thm:sdim1}, will be the case if the vertices are all far away from each other), but $\P(H$ is faithfully present at $x)= \P(S^H_{x}(0,\infty) >0) \leq \eps$.
% 
% \medskip
% 
In order to prove this, we seek to obtain upper bounds  on the quantity $\bbW^H_{x}(n,N)$. We begin by considering the case of a single distant scale. That is, the case that $|N-n|$ is a constant and all the points of $x$ are contained in $\Lambda_x(0,n-1)$. Recall that 
	$\bareta_{d}(H)$ is defined to be $\min \{\eta_{d}(H') : \text{ $H'$ is a coarsening of $H$}\}$.

% \medskip

\begin{lem}[A single distant scale]
	\label{lem:firstmoment}
	Let $\bbG$ be a $d$-dimensional transitive graph and let $H$ be a finite hypergraph with boundary. Then for every $m \geq 0$, there exists a constant $c=c(\bbG,H,m)$ such that
	\begin{equation*}
		\log_2\bbW^H_x(n,n+m) \leq   -\bareta_d(H) \, n + |E|^2\log_2 n + c 
	\end{equation*}
	for all $x=(x_u)_{u\in \partial V} \in \bbV^{\partial V}$ and all $n$ such that $\langle x_u x_v \rangle \leq 2^{n-1}$ for all $u,v \in \partial V$.
\end{lem}

It will be useful for applications in \cref{Sec:technical} to prove a more general result. A graph $\bbG$ is said to be \textbf{$d$-Ahlfors regular}  if there exists a positive constant $c$ such that $c^{-1} r^d \leq |B(x,r)| \leq cr^d$ for every $r\geq 1$ and every $x \in V$ (in which case we say $\bbG$ is $d$-Ahlfors regular with constant $c$).
 Given $\alpha>0$ and a finite hypergraph with boundary $H$, we define
\[
	\eta_{d,\alpha}(H) = (d-2\alpha)\Delta - d|E| - (d-2\alpha)|V_\circ|,
\]
where we recall that $\Delta =\sum_{e\in E}\deg(e) = \sum_{v\in V} \deg(v)$, and define $\bareta_{d,\alpha}(H) = \min \{\eta_{d,\alpha}(H') : \text{ $H'$ is a coarsening of $H$}\}$.
Given a graph $\bbG$, a finite hypergraph with boundary $H=(\partial V, V_\circ, E)$, and points $(x_v)_{v\in \partial V}$, $(\xi_e)_{e\in E}$ we also define 
\[ 
	W_\alpha^H(x,\xi) = \prod_{u \in \partial V} \langle x_u, \{\xi_e: e \perp u\} \rangle^{-(d-2\alpha)} \prod_{u \in V_\circ} \langle \{\xi_e: e \perp u\} \rangle^{-(d-2\alpha)} 
\]
and, for each $N> n$,
\[
	\mathbb{W}^{H,\alpha}_{x}(n,N)  = \sum_{\xi\in \Xi_x(n,N)} W_\alpha^H(x,\xi). 
\]
Note that $\eta_d=\eta_{d,2}$ and $\bbW_x^H=\bbW_x^{H,2}$, so that \cref{lem:firstmoment} follows as a special case of the following lemma.

\begin{lem}[A single distant scale, generalised]
	\label{lem:firstmomentgeneral}
	Let $\bbG$ be a $d$-Ahlfors regular graph with constant $c'$, let $H$ be a finite hypergraph with boundary, and let $\alpha \in \R$ be such that $d\geq 2\alpha$. Then for every $m \geq 0$, there exists a constant $c=c(c',H,\alpha,d,m)$ such that
	\begin{equation*}
		\log_2\bbW^{H,\alpha}_{x}(n,n+m) \leq   -\bareta_{d,\alpha}(H) \, n + |E|^2\log_2 n + c 
	\end{equation*}
	for all $x=(x_u)_{u\in \partial V} \in \bbV^{\partial V}$ and all $n$ such that $\langle x_u x_v \rangle \leq 2^{n-1}$ for all $u,v \in \partial V$.
\end{lem}

Before proving this lemma, we will require a quick detour to analyze a relevant optimization problem. 

\subsubsection*{Optimization on the ultrametric polytope}
\label{subsec:ultrametric}

% We will require the following simple lemma.
 Recall that a (semi)metric space $(X,d)$ is an \textbf{ultrametric} space if $d(x,y) \leq \max \{d(x,z),d(z,y)\}$ for every three points $x,y,z\in X$. For each finite set $A$, the \textbf{ultrametric polytope} on $A$ is defined to be 
\[ \mathcal{U}_A = \left\{(x_{a,b})_{a,b \in A}  \in [0,1]^{A^2} : 
\begin{array}{l}
x_{a,a}=0 \text{ for all $a \in A$},\, x_{a,b}=x_{b,a} \text{ for all $a,b\in A$},
\vspace{0.25em}
\\
 \text{and } x_{a,b} \leq \max\left\{x_{a,c},x_{c,b}\right\} \text{for all $a,b,c \in A$}
\end{array}
\right\},  \]
which is a closed convex subset of $\R^{A^2}$. 
We consider $\cU_A$ to be the set of all ultrametrics on $A$ with distances bounded by $1$.  
We write $\cP(A^2)$ for the set of subsets of $A^2$.

\begin{lemma}\label{lem:ultrametric1}
	Let $A$ be a finite non-empty set, and let $F:\R^{A^2}\to \R$ be of the form
	\[
		F(x) = \sum_{k=1}^K c_k \min\{x_{a,b} : (a,b) \in W_k\},
	\]
	where $K<\infty$, $c_1,\ldots,c_K \in \R$, and $W_1,\ldots, W_K \in \cP(A^2)$. Then the maximum of $F$ on $\cU_A$ is obtained by an ultrametric for which all distances are either zero or one. That is,
	\[
		\max\{F(x) : x \in \cU_A\} = \max\left\{F(x) : x \in \cU_A,\, x_{a,b} \in \{0,1\} \text{ for all $a,b\in A$}\right\}.
	\]
\end{lemma}

\begin{proof}
We prove the claim by induction on $|A|$. The case $|A|=1$ is trivial. Suppose that the claim holds for all sets with cardinality less than that of $A$. 
We may assume that $(a,a) \notin W_k$ for every $1\leq k \leq K$ and $i \in A$, since if $(a,a)\in W_k$ for some $1\leq k \leq K$ then the term $c_k \min \{x_{a,b} : (a,b) \in W_k\}$ is identically zero on $\cU_A$. 
We write $\mathbf{1}$ for the vector
\[\mathbf{1}_{(a,b)}=\mathbbm{1}(a\neq b).\]
% $\mathbf{c}\in\R^{A^2}$ for the vector $\mathbf{c}_{i,j}=c$ for all $i,j$.
% \[\mathbf{1}\]
% 
 % Observe that 
 It is easily verified that
\[F(\lambda x) = \lambda F(x)
% when $\lambda >0$ is a non-negative constant, and that 
% and
 \quad \text{ and } \quad F(x+\alpha\mathbf{1}) = F(x) + \alpha F(\mathbf{1})  \]
% for every constant $c \in \R$.
for every $x\in \R^{A^2}$, every $\lambda\geq 0$, and every $\alpha \in \R$. 

Suppose $y\in\cU_A$ is such that $F(y) = \max_{x\in \cU_A} F(x)$. 
We may assume that $F(y)>F(\mathbf{1})$ and that $F(y)>F(0)=0$, since otherwise the claim is trivial. 
Let $m = \min \{ y_{a,b} : a,b \in A, a\neq b\}$, which is less than one by assumption.
We have that
% Let $a= \min\{x_{i,j}\}$. We claim that $a=0$. By assumption, $a\neq 1$, and we have
\[\frac{y}{1-m}- \frac{m}{1-m}\mathbf{1} \in \cU_A\]
and
\[F\left(\frac{y}{1-m}- \frac{m}{1-m}\mathbf{1}\right) = \frac{F(y)}{1-m} - \frac{mF(\mathbf{1})}{1-m} = F(y) + \frac{m}{1-m}(F(y)-F(\mathbf{1})), \]
and so we must have $m=0$ since $y$ maximizes $F$.

Define an equivalence relation $\bowtie$ on $A$ by letting $a$ and $b$ be related if and only if $y_{a,b}=0$. We write $\hat a$ for the equivalence class of $b$ under $\bowtie$. Let $C$ be the set of equivalence classes of $\bowtie$, and let $\phi: \cU_C \to \cU_A$ be the function defined by
\[\phi(x)_{a,b} = x_{\hat a, \hat b}\]
for every $x\in \cU_n$.
For each $1\leq k \leq K$, let $\hat W_k$ be the set of pairs $\hat a, \hat b \in C$ such that $(a,b) \in W_k$ for some $a$ in the equivalence class $\hat a$ and $b$ in the equivalence class $\hat b$. Let 
 $\hat F : \cU_C \to \R$ be defined by 
\[\hat F(x) = \sum_{k=1}^K c_k \min \{x_{a,b} : (\hat a, \hat b) \in \hat W_k \}. \]
We have that $\hat F = F \circ \phi$, and, since $y$ maximized $F$, we deduce that, by the induction hypothesis,
\begin{align*}\max\{F(x): x \in \cU_A\} &= \max\{\hat F(x) : x \in \cU_C\}\\ &= \max\{\hat F(x) : x \in \cU_C,\, x_{\hat a, \hat b} \in \{0,1\} \text{ for all $\hat a, \hat b \in C$}\},\end{align*}
completing the proof.
\end{proof}

We will also require the following generalisation of \cref{lem:ultrametric1}.
For each finite collection of disjoint finite sets $\{A_i\}_{i\in I}$ with union $A = \bigcup_{i\in I} A_i$, we define
\[ \mathcal{U}_{\{A_i\}_{i\in I}} = \{ x \in \cU_{A} : x_{a,b}=1 \text{ for every distinct $i,j \in I$ and every $a \in A_i$ and $b \in A_j$.}\}.  \]

\begin{lemma}\label{lem:ultrametric2}
	Let $\{A_i\}_{i\in I}$ be a finite collection of disjoint, finite, non-empty sets with union $A = \bigcup_{i\in I}A_i$, and let $F:\R^{A^2}\to \R$ be of the form
	\[
		F(x) = \sum_{k=1}^K c_k \min\{x_{(i,j)} : (i,j) \in W_k\},
	\]
	where $K<\infty$, $c_1,\ldots,c_K \in \R$, and $W_1,\ldots, W_K \in \cP(A^2)$. Then the maximum of $F$ on $\cU_A$ is obtained by an ultrametric for which all distances are either zero or one. That is,
	\[
		\max\{F(x) : x \in \cU_{\{A_i\}_{i\in I}}\} = \max\left\{F(x) : x \in \cU_{\{A_i\}_{i\in I}},\, x_{a,b} \in \{0,1\} \text{ for all $a,b\in A$}\right\}.
	\]
\end{lemma}

\begin{proof}
We prove the claim by fixing the index set $I$ and inducting on $|A|$. The case $|A|=|I|$ is trivial. Suppose that the claim holds for all collections of finite disjoint sets indexed by $I$ with total cardinality less than that of $A$. 
We may assume that $(i,i) \notin W_k$ for every $1\leq k \leq K$ and $i \in A$, since if $(i,i)\in W_k$ for some $1\leq k \leq K$ then the term $c_k \min \{x_{i,j} : (i,j) \in W_k\}$ is identically zero on $\cU_A$. Furthermore, we may assume that $W_k$ contains more than one element of at least one of the sets $A_i$ for each $1 \leq k \leq K$, since otherwise the term $c_k \min \{x_{i,j} : (i,j) \in W_k\}$ is equal to the constant $c_k$ on $\cU_{\{A_i\}_{i\in I}}$. 
We write $\mathbf{1}$ and $\mathbf{i}$ for the vectors
\[\mathbf{1}_{a,b} = \mathbbm{1}(a\neq b)\]
and
\[\mathbf{i}_{a,b} = \mathbbm{1}(\text{$a\neq b$, and $a,b\in A_i$ for some $i \in I$}).\]
% $\mathbf{1}$ for the vector $\mathbf{1}_{i,j}\equiv1$, and write $\delta$ for the vector $\delta_{i,j}=\mathbbm{1}(i=j)$.
% $\mathbf{c}\in\R^{A^2}$ for the vector $\mathbf{c}_{i,j}=c$ for all $i,j$.
% \[\mathbf{1}\]
% 
 % Observe that 
 It is easily verified that
\[F(\lambda x) = \lambda F(x)
% when $\lambda >0$ is a non-negative constant, and that 
% and
 \quad \text{ and } \quad F(x+\alpha\mathbf{i}) = F(x) + \alpha F(\mathbf{1})  \]
% for every constant $c \in \R$.
for every $x\in \cU_{\{A_i\}_{i\in I}}$, every $\lambda\geq 0$, and every $\alpha \in \R$ such that $x + \alpha \mathbf{i} \in \cU_{\{A_i\}_{i\in I}} $. 

The rest of the proof is similar to that of \cref{lem:ultrametric1}. \qedhere

\end{proof}

% That is, $\cU_A$ is the set of all ultrametrics on $A$ with distances bounded by $1$. We write $\cP(A^2)$ for the set of subsets of $A^2$.

\subsubsection*{Back to the uniform spanning forest}

% \medskip

We now return to the proofs of \cref{prop:nonubiquity} and \cref{lem:firstmomentgeneral}.

\begin{proof}[Proof of \cref{lem:firstmomentgeneral}]
	In this proof, implicit constants will be functions of $c',H,\alpha,d$ and $m$. The case that $E = \emptyset$ is trivial (by the assumption that $d \geq 2 \alpha$), so we may assume that $|E|\geq 1$.

	Write $\Xi=\Xi_x(n,n+m)$. 
	First, observe that
	\[
		\langle x_u, \{\xi_e: e \perp u\} \rangle \asymp 2^{n} \langle \{\xi_e: e \perp u\} \rangle \]
	for every $\xi \in \Xi$ and $u \in \partial V$, and hence that
	\begin{align*}
	\bbW^{H,\alpha}_{x}(n,n+m) &= \sum_{\xi\in \Xi} \prod_{u \in \partial V} \langle x_u, \{\xi_e: e \perp u\} \rangle^{-(d-2\alpha)} \prod_{u \in V_\circ} \langle \{\xi_e: e \perp u\} \rangle^{-(d-2\alpha)} \\
	% &\leq 
	% 2^{-(d-2\alpha)|P|(n-m)} \sum_{\xi \in \Lambda(n,n+m)} \prod_{u \in V} \langle \{\xi_e: e \perp u\} \rangle^{-(d-2\alpha)}\\
	&\preceq 
	2^{-(d-2\alpha)|\partial V|n} \sum_{\xi \in \Xi} \prod_{u \in V} \langle \{\xi_e: e \perp u\} \rangle^{-(d-2\alpha)}.
	\end{align*}

	Let $L$ be the set of symmetric functions $\ell:E^2 \to \{0,\ldots,n\}$ such that $\ell(e,e)=0$ for every $e\in E$. 
	 For each $\ell \in L$, let 
	\[
		\Xi_\ell = \left\{\xi \in \Xi :
			\begin{array}{l}
		 		2^{\ell(e,e')} \leq \langle \xi_e \xi_{e'} \rangle \leq 2^{\ell(e,e')+m+1} \text{ for all $e,e' \in E$}
			\end{array}
		\right\},
	\]
	 so that $\Xi = \bigcup_{\ell \in L} \Xi_\ell$. 
	For each $\ell$ in $L$, let
	\[
		\hat \ell(e,e') = \min \{\ell(e,e')\}\cup\left\{\max \{\ell(e,e_1),\ldots,\ell(e_k,e')\}: k\geq 1 \text{ and } e_1,\ldots,e_k \in E \right\}.
	\]
	In other words, $\hat \ell$ is the largest ultrametric on $E$ that is dominated by $\ell$. 
	Observe that for every $\ell \in L$, every $\xi \in \Xi_\ell$, end every $e,e',e'' \in E$, we have that
	\begin{align*} 
		\log_2 \langle \xi_e \xi_{e'} \rangle &\leq \log_2\left[ \langle \xi_e \xi_{e''} \rangle + \langle \xi_{e''} \xi_{e'} \rangle\right] 
		 \leq \log_2\max\{\langle \xi_e \xi_{e''} \rangle,\, \langle \xi_{e''} \xi_{e'} \rangle\} +1\\
		& \leq \max\{\ell(e,e''),\, \ell(e'',e')\}+2m +3,
	\end{align*}
	and hence, by induction, that
	\[
		\log_2 \langle \xi_e \xi_{e'} \rangle \leq \hat \ell(e,e')+(2m+3)|E| \approx \hat \ell(e,e').
	\]

	Let $e_1,\ldots,e_{|E|}$ be an enumeration of $E$.
	For every $\ell \in L$, every $1 \leq j < i \leq |E|$ and every $\xi \in \Xi_\ell$ we have that
	\[
		\xi_{e_i} \in B\left(\xi_{e_j},\, 2^{\hat\ell(e_i,e_j)+(2m+3)|E|}\right) \text{ and } \left|B\left(\xi_{e_j},\, 2^{\hat\ell(e_i,e_j)+(2m+3)|E|}\right)\right| \preceq 2^{d\hat\ell(e_i,e_j)}.
	\]
	By considering the number of choices we have for $\xi_{e_i}$ at each step given our previous choices, it follows that
	\begin{align}
		\log_2 |\Xi_\ell| \lesssim  
		dn+ d\sum_{i=2}^{|E|}\min\left\{\hat \ell(e_i,e_j) : j<i\right\}. 
		\label{eq:Lambdaell}
	\end{align}
	% where we define $\min \emptyset := 0$. 

	Now, for every $\xi \in \Xi_\ell$, we have that
	\begin{align}
	\log_2 \prod_{u \in V} \langle \{\xi_e: e \perp u\} \rangle^{-(d-2\alpha)}
	&\approx -(d-2\alpha)\sum_{u \in V}\sum_{i=2}^{|E|}\mathbbm{1}(e_i \in u) \min \left\{\ell(e_i,e_j) : j<i,\, e_j \perp u\right\}
	\nonumber
	\\
	&=
	-(d-2\alpha)\sum_{i=2}^{|E|}\sum_{u \perp e}\min \left\{\ell(e_i,e_j) : j<i,\, e_j \perp u\right\}.
	\label{eq:Lambdaellsfriend}
	\end{align}
	Thus, from \eqref{eq:Lambdaell} and \eqref{eq:Lambdaellsfriend} we have that
	\begin{multline}
		\log_2 \sum_{\xi \in \Xi_\ell} \prod_{u \in V} \langle \{\xi_e: e \perp u \} \rangle^{-(d-2\alpha)}\\ \lesssim  
		dn+ \sum_{i=2}^{|E|}\left[ d\min\{\hat \ell(e_i,e_j) : j<i\} -(d-2\alpha) \sum_{u \perp e}\min \{\ell(e_i,e_j) : j<i,\, e_j \perp u\}\right]. \label{eq:Qfirstmoment1}
	\end{multline}

	Let  $Q: L \to \R$ be defined to be the expression on the right hand side of \eqref{eq:Qfirstmoment1}.
	We clearly have that $Q(\hat \ell) \geq Q(\ell)$ for every $\ell \in L$, and so there 
	exists $\ell\in L$ maximizing $Q$ such that $\ell$ is an ultrametric. It follows from \cref{lem:ultrametric1} (applied to the normalized ultrametric $\ell/n$) that there exists $\ell\in L$ maximizing $Q$ such that $\ell$ is an ultrametric and every value of $\ell$ is in $\{0,n\}$.
	Fix one such $\ell$, and define an equivalence relation $\bowtie$ on $E$ by letting $e \bowtie e'$  if and only if $\ell(e,e')=0$, which is an equivalence relation since $\ell$ is an ultrametric.
	Observe that, for every $2 \leq i \leq |E|$,
	\[ 
		\min \{\ell(e_i,e_j) : j < i\} = \mathbbm{1}[\text{$e_j$ is not in the equivalence class of $e_i$ for any $j<i$}]\, n, 
	\]
	and hence that
	\[ 
		dn + \sum_{i=2}^{|E|} \min \{\ell(e_i,e_j) : j < i\} = |\{\text{equivalence classes of $\bowtie$}\}| \, n. 
	\]
	Similarly, we have that, for every vertex $u$ of $H$,
	\[ 
		\sum_{i=2}^{|E|} \min \{\ell(e_i,e_j) : j < i, e_j \perp u\} = \left(|\{\text{equivalence classes of $\bowtie$ incident to $u$}\}|-1\right)\,n,
	\]
	where we say that an equivalence class of $\bowtie$ is incident to $u$ if it contains an edge that is incident to $u$. Thus, we have that
	\begin{multline}
	\label{eq:Qequiv}
		Q(\ell) = 
	 	d|\{\text{equivalence classes of $\bowtie$}\}|\, n
		\\-(d-2\alpha)\sum_{u \in V} (|\{\text{equivalence classes of $\bowtie$ incident to $u$}\}|-1)\, n.
	\end{multline}

	Let $H'=\coarse{H}{\bowtie}$ be the coarsening of $H$ associated to $\bowtie$ as in \cref{subsec:optimalcoarsenings}. 
	We can rewrite \eqref{eq:Qequiv} as
	\begin{align*}
		Q(\ell) &=  d|E(H')|\, n   -(d-2\alpha)\Delta(H')\, n+(d-2\alpha)|V(H)|\, n
		= -\eta_{d,\alpha}(H')\, n+(d-2\alpha)|\partial V| \, n.
		% \leq - \bareta_{d,\alpha}(H)\, n  +(d-2\alpha)|\partial V|\, n.
	\end{align*}
	Since $|L| \leq (n+1)^{|E|^2}$, we deduce that
	\begin{multline*}
		\log_2\bbW^{H,\alpha}_{x}(n,n+m) \lesssim -(d-2\alpha)|\partial V|\, n + \log_2 \sum_{\ell \in L} Q(\ell) 
		\\
		\leq
		\max_{\ell\in L}Q(\ell) - (d-2\alpha)|\partial V|\, n + \log_2|L| \lesssim  - \bareta_{d,\alpha}(H)\, n + |E|^2 \log_2 n
	 \end{multline*}
	as claimed. \qedhere

\end{proof}

% \medskip

Next, we consider the case that the points $x_v$ are roughly equally spaced and we are summing over points $\xi$ that are on the same scale as the spacing of the $x_v$.

\begin{lem}[The close scale]
\label{lem:firstmomentclose}
	Let $\bbG$ be a $d$-dimensional transitive graph with $d>4$ and let $H$ be a finite hypergraph with boundary. Let $m_1,m_2\geq 0$. Then there exists a constant $c=c(\bbG,H,m_1,m_2)$ such that
\begin{equation*}\log_2\bbW^H_x(0,n+m_2) \leq    -\bareta_d(H)\, n + |E \cup \partial V|^2\log_2 n + c \end{equation*}
 for every $n \geq 1$ and every $x=(x_u)_{u\in \partial V} \in \bbV^{\partial V}$ are such that $2^{n-m_1} \leq \langle x_u x_v \rangle \leq 2^n$ for every $u,v\in V$.
\end{lem}

\begin{proof}
 We may assume that $E \neq \emptyset$, the case $E=\emptyset$ being trivial. 
For notational convenience, we will write $\xi_v=x_v$, and consider $v \perp v$ for every vertex $v\in \partial V$. Write $\Xi=\Xi_x(0,n+m_2)$, and observe that for each $\xi \in \Xi$ and $e \in E$ there exists at most one $v\in \partial V$ for which $\log_2 \langle \xi_e \xi_v\rangle < n-m_1-1$. To account for these degrees of freedom, we define $\Phi$ to be the set of functions $\phi:E \cup \partial V \to \partial V \cup \{\star\}$ such that $\phi(v)=v$ for every $v\in \partial V$. For each $\phi \in \Phi$, 
% and $v\in \partial V$,
let $L_\phi$ be the set of symmetric functions $\ell: (E \cup \partial V)^2 \to \{0,\ldots,n\}$ such that $\ell(e,e)=0$ for every $e\in E \cup \partial V$ and $\ell(e,e')=n$ for every $e,e' \in E \cup \partial V$ such that $\phi(e)\neq \phi(e')$.
  For each $\phi \in \Phi$ and $\ell \in L_\phi$,
 % and $\ell = (\ell_v)_{v\in \partial V\cup \{\star\}} \in L_\phi$, 
 let 
\begin{multline*}\hspace{-0.25cm}\Xi_{\phi,\ell} = \left\{\xi \in \Xi : 
\begin{array}{l}
\ell(e,e')- m_1 - 1 \leq \log_2 \langle \xi_e \xi_{e'} \rangle \leq \ell(e,e') + m_2+1 \text{ for every $e,e' \in E \cup \partial V$}
\end{array}
\hspace{-0.15cm} \right\},\end{multline*}
 and observe that
$\Xi = \bigcup_{\phi\in \Phi} \bigcup_{\ell \in L_\phi} \Xi_{\phi,\ell}$.

Now, for each $\phi \in \Phi$ and $\ell \in L_\phi$, let $\hat \ell$ be the largest ultrametric on $E \cup \partial V$ that is dominated by $\ell$. Observe that $\hat \ell \in L_\phi$, and that, as in the previous lemma, we have that
% 
% For each $\phi\in \Phi$, 
% Given $\ell\in L_\phi$, we set $\ell(e,e')=\ell_{\phi(e)}(e,e')$ and $\hat \ell(e,e')=\hat\ell_{\phi(e)}(e,e')$ if $\phi(e)=\phi(e')$,  
% 
% $\ell(e,e')=\hat\ell(e,e')=n$ if $\phi(e)\neq \phi(e')$, and set $\ell($ 
% As before, we have that
% Observe that  for every $\ell \in L$, every $\xi \in \Xi_\ell$, end every $e,e',e'' \in E \cup \partial V$, we have that
% \begin{align*} \log_2 \langle \xi_e \xi_{e'} \rangle &\leq \log_2\left[ \langle \xi_e \xi_{e''} \rangle + \langle \xi_{e''} \xi_{e'} \rangle\right] \\
% & \leq \log_2\max\{\langle \xi_e \xi_{e''} \rangle,\, \langle \xi_{e''} \xi_{e'} \rangle\} +1\\
% & \leq \max\{\ell(e,e''),\, \ell(e'',e')\}+1,\end{align*}
% and hence, by induction, that
\[\log_2 \langle \xi_e \xi_{e'} \rangle \lesssim \hat \ell(e,e')\]
for every $e,e' \in E \cup \partial V$. 
% for every $v\in \partial V\cup\{\star\}$, $\xi \in \Xi_{\phi,\ell}$ and every $e,e' \in \phi^{-1}(v)$. Similarly, we have
% \[\log_2 \langle \xi_e x_v \rangle \lesssim \hat \ell_v(e,v)\]
% for every $v\in \partial V$, $\xi \in \Xi_{\phi,\ell}$ and every $e \in \phi^{-1}(v)$.

% \medskip

Let $e_1,\ldots,e_{|E|}$ be an enumeration of $E$, and let $e_0,e_{-1},\ldots,e_{-|\partial V|+1}$ be an enumeration of $\partial V$. As in the proof of the previous lemma, we have the volume estimate
% For every $\ell \in L$, every $1 \leq j < i \leq |E|$ and every $\xi \in \Xi_\ell$ we have that
% \[\xi_{e_i} \in B\left(\xi_{e_j},\, 2^{\hat\ell(e_i,e_j)+|E|}\right) \text{ and } \left|B\left(\xi_{e_j},\, 2^{\hat\ell(e_i,e_j)+|E|}\right)\right| \preceq 2^{d\hat\ell(e_i,e_j)}\]
% and it follows that
\begin{align}\log_2 |\Xi_{\phi,\ell}| \lesssim  
% dn+ d\sum_{i=1}^{|E|}\min\{\hat \ell(e_i,e_j) +|E| : j<i\}
% \approx
d\sum_{i=1}^{|E|}\min\{\hat \ell(e_i,e_j) : j<i\}
% \end{align}
  \label{eq:Lambdaell}\end{align}
% where we define $\min \emptyset := 0$. 

Now, for every $\xi \in \Xi_{\phi,\ell}$, we have that, similarly to the previous proof,
\begin{align*}
 \log_2 W(x,\xi)
% \\
% &\approx -(d-2\alpha)\sum_{u \in V}\sum_{i=1}^{|E|}\mathbbm{1}(e_i \in u) \min \{\ell(e_i,e_j) : j<i,\, e_j \perp u\}
% \\
&\approx
-(d-4)\sum_{i=1}^{|E|}\sum_{u \perp e}\min \{\ell(e_i,e_j) : j<i,\, e_j \perp u\}.
\end{align*}
(Recall that we are considering $u\perp u$ for each $u \in \partial V$.)
Thus, we have
\begin{multline}\log_2 \sum_{\xi \in \Xi_{\phi,\ell}} W(x,\xi)\\ \lesssim  
 \sum_{i=1}^{|E|}\left[ d\min\{\hat \ell(e_i,e_j) : j<i\} -(d-4) \sum_{u \perp e}\min \{\ell(e_i,e_j) : j<i,\, e_j \perp u\}\right]. \label{eq:Qfirstmoment}
\end{multline}

Let  $Q: L_\phi \to \R$ be defined to be the expression on the right hand side of \eqref{eq:Qfirstmoment}.
Similarly to the previous proof but applying \cref{lem:ultrametric2} instead of \cref{lem:ultrametric1}, there is an $\ell\in L_\phi$ maximizing $Q$ such that $\ell$ is an ultrametric and $\ell(e,e') \in \{0,n\}$ for all $e,e' \in E \cup \partial V$.
Fix one such $\ell$, and define an equivalence relation $\bowtie$ on $E \cup \partial V$ by letting $e \bowtie e'$  if and only if $\ell(e,e')=0$, which is an equivalence relation since $\ell$ is an ultrametric.
Similarly to the proof of the previous lemma, we can compute that
\begin{multline*}Q(\ell) = 
 dn\left|\{\text{equivalence classes of $\bowtie$ that are contained in $E$}\}\right|
\\
-(d-4)n\sum_{u \in \partial V} \left|\{\text{equivalence classes of $\bowtie$ incident to $u$ that do not contain $u$}\}\right|
\\
-(d-4)n \sum_{u \in V_\circ} \left(\left|\{\text{equivalence classes of $\bowtie$ incident to $u$}\}\right| -1\right).
% \\-(d-4)n\sum_{u \in  V_\circ} (\#\{\text{equivalence classes of $\bowtie$ adjacent to $u$}\}
\end{multline*}
Since $d>4$ and each equivalence class of $\bowtie$ can contain at most one vertex of $v$, we see that $Q$ increases if we remove a vertex $v\in \partial V$ from its equivalence class. Since $\ell$ was chosen to maximize $Q$, we deduce that the equivalence class of $v$ under $\bowtie$ is a singleton for every $v\in \partial V$. Thus, there exists an ultrametric $\ell \in L_\phi$ maximizing $Q$ such that $\ell(e,e')\in \{0,n\}$ for every $e,e' \in E$ and $\ell(e,v)=n$ for every $e\in E$ and $v\in \partial V$. Letting $\bowtie'$ be the equivalence relation on $E$ (rather than $E \cup \partial V$) corresponding to such an optimal $\ell$, we have
\begin{multline}
\label{eq:Qell2}
Q(\ell) = 
 dn\left|\{\text{equivalence classes of $\bowtie'$}\}\right|
\\
-(d-4)n\sum_{u \in \partial V} \left|\{\text{equivalence classes of $\bowtie'$ incident to $u$}\}\right|.
\\
-(d-4)n \sum_{u \in V_\circ} \left(\left|\{\text{equivalence classes of $\bowtie'$ incident to $u$}\}\right| -1\right).
% \\-(d-4)n\sum_{u \in  V_\circ} (\#\{\text{equivalence classes of $\bowtie$ adjacent to $u$}\}
\end{multline}
 The rest of the proof is similar to the proof of \cref{lem:firstmoment}.
 % , by rewriting \eqref{eq:Qell2} in terms of the coarsening $\coarse{H}{\bowtie'}$ and using the fact that $|\Phi|$ is a constant and $|L_\phi| \leq (n+1)^{|E\cup \partial V|^2}$ for every $\phi \in \Phi$.
% the maximum of $Q$ is the same as the maximum of the function $Q$ from \cref{lem:firstmoment}, and we conclude the proof as before.
\qedhere
\end{proof}

We can now bootstrap from the single scale estimates \cref{lem:firstmoment,lem:firstmomentclose} to a multi-scale estimate. Given a hypergraph with boundary $H=(\partial V, V_\circ, E)$ and a set of edges $E'\subseteq E$, we write 
$V_\circ(E')=\bigcup_{e\in E'}\{v\in V_\circ : v \perp e\}$ and define
$H(E') = (\partial V, V_\circ(E'), E')$.
% The key to this is the following simple consequence of \cref{lem:firstmoment}.
 % We say that a subhypergraph $H'$ of a hypergraph with boundary $H$ is \textbf{proper} if it is not equal to $H$, and say that $H'$ is \textbf{edge-induced} if its vertex set is equal to $\bigcup_{e\in E'} \{v\in V : v \perp e\}$.
  % and recall that $H'$ is said to be full if $\{v\in V: v \perp e\}\subseteq V'$ for every $e \in E'$.

\begin{lem}[Induction estimate]
\label{lem:inductionestimate}
Let $\bbG$ be a $d$-dimensional transitive graph and let $H$ be a finite hypergraph with boundary. Then there exists a constant $c=c(\bbG,H)$ such that
	\begin{multline*}
		\log_2\left[\bbW^H_x(0,N+|E|+2) - \bbW^H_x(0,N)\right] \leq 
		 \\
		   % \max_{\substack{H' \text{\emph{ a proper edge-induced}} \\ \text{\emph{subhypergraph of }} H }}
		   \max_{E' \subsetneq E}
		   \left\{
		    \log_2 \bbW^{H(E')}_x(0,N+|E|+2) 
		     - \left[\bareta_d(H) -\bareta_d(H(E')) \right]\, N + |E\setminus E'|^2\log_2 N 
			\right\} + c
	\end{multline*}
	for every $x=(x_u)_{u\in \partial V} \in \bbV^{\partial V}$ and every $N$ such that $\langle x_u x_v \rangle \leq 2^{N-1}$ for all $u,v \in \partial V$.
\end{lem}

Note that when $|E|\geq 1$ we must consider the term $E'=\emptyset$ when taking the maximum in this lemma, which gives $-\bareta_d(H) N + |E|^2 \log_2 N$.

\begin{proof}
The claim is trivial in the case $E=\emptyset$, so suppose that $|E|\geq 1$. 
	Let
	$\Xi = \Xi_x(0,N+|E|+2) \setminus \Xi_x(0,N)$ 
	so that
	\[ \bbW^H_x(0,N+|E|+2) - \bbW^H_x(0,N) \leq \sum_{\xi \in \Xi} W^H(x,\xi).\]
			For each $E'\subsetneq E$ and every $1 \leq m \leq |E|+1$, let
	\[ \Xi^{E', m} = \left(\Lambda_x(0,N+m-1)\right)^{E'} \times \left(\Lambda_x(N+m,N+|E|+2)\right)^{E \setminus E'}.
	\]
	Observe that if $\xi \in \Xi$
	% \[\xi \in \Xi_x(0,N+1) \setminus \left(\Xi_x(0,N) \cup \Xi_x(N-|E|-1,N+1)\right),\]
	then, by the Pigeonhole Principle, there must exist $1 \leq m \leq |E|+2$ such that $\xi_e$ is not in $\Lambda_x(N-m-1,N-m)$ for any $e \in E$, and we deduce that
	\[
	\Xi = \bigcup \left\{ \Xi^{E',m} :  E'\subsetneq E,\, 1\leq m \leq |E|+2 \right\}.
	\]
	Thus, to prove the lemma it suffices to show that
	\begin{equation}
	\label{eq:inductionestimate1}
	\log_2 \sum_{\xi \in \Xi^{H(E'),m}} W^H(x,\xi) \lesssim 
		     \log_2 \bbW^{H(E')}_x(0,N) - \left(\bareta_d(H) -\bareta_d(H(E')) \right)\, N + |E \setminus E'|^2\log_2 N 
	\end{equation}
	whenever $1\leq m \leq |E|+2$ and $E' \subsetneq E$. If $E'=\emptyset$ then this follows immediately from \cref{lem:firstmoment}, so we may suppose not.
	% Let $m(\xi)$ be the maximal such $m$, and let 
	% \[J(\xi) = \{e \in E : \xi_e \in \Lambda(n,N-m(\xi)-1)\}.\]

	\medskip

	To this end, fix $E' \subsetneq E$ with $|E'|\geq 1$ and write $H'=H(E') = (\partial V,V_\circ(E'),E') = (\partial V, V_\circ',E')$.  Choose some $v_0 \in \partial V$ arbitrarily, and write $x_v = x_{v_0}$ for every $v \in V_\circ'$. Then for every $\xi \in \Xi^{E',m}$, using the fact that we have the empty scale $\Lambda_x(N-m-1,N-m)$ separating $\{\xi_e : e \in E'\}$ from $\{\xi_e : e \notin E'\}$, we have that
	\[
		\left\langle \{x_u\} \cup \{\xi_e: e \perp u\} \right\rangle 
		\asymp \left\langle \{x_u\}\cup \{\xi_e: e\in E',\,  e \perp u\} \right\rangle  \left\langle \{x_u\}\cup\{\xi_e: e \notin E',\,  e \perp u\} \right\rangle
	\]
	for every vertex $u\in \partial V$,
	\[
	\left\langle \{\xi_e: e \perp u\} \right\rangle \asymp
	 \left\langle \{\xi_e: e\in E',\, e \perp u\} \right\rangle  \left\langle \{x_u\}\cup \{\xi_e: e \notin E',\,  e \perp u\} \right\rangle
	\]
	for every vertex $u \in V'_\circ$, and that, trivially,
	\[
		\left\langle \{\xi_e: e \perp u\} \right\rangle =
	  \left\langle  \{\xi_e: e \notin E',\,  e \perp u\} \right\rangle
	\]
	for every vertex $u \in V_\circ \setminus V'_\circ$. Define a hypergraph with boundary $H'' =(\partial V'', V_\circ'',  E'',\perp'')$ by setting \begin{multline*}\text{$\partial V'' = \partial V \cup V_\circ'$,\qquad $V_\circ'' = V_\circ \setminus V_\circ',$\qquad $V''= \partial V'' \cup V_\circ''=V$,\qquad $E'' = E \setminus E'$,}\\ \text{and 
	$\perp''=\perp \cap\, (V'' \cap E'')$.}\end{multline*}
	For each $\xi \in \Xi^{E',m}$, let $\xi'=(\xi'_e)_{e\in E'}=(\xi_e)_{e\in E'}$ and $ \xi''=( \xi''_e)_{e\in E''} = (\xi_e)_{e\in E''}$.
	Then the above displays imply that
	\[ W^H(x,\xi) \asymp W^{H'}\left(x,\xi'\right) \cdot W^{H''}\big(x, \xi''\big) \]
	for every $\xi \in \Xi^{E',m}$. Thus, summing over $\xi' \in (\Lambda_x(0,N+m-1))^{E'}$ and $ \xi'' \in (\Lambda_x(N+m,N+|E|+2))^{E''}$, we obtain that
	\begin{align} 
	\log_2 \sum_{\xi \in \Xi^{E',m}} W^H(x,\xi) &\lesssim \log_2 \bbW^{H'}_x(0,N+m-1) + \log_2 \bbW^{H''}_x(N+m,N+|E|+2)
	\nonumber
	\\
	&\lesssim \log_2 \bbW^{H'}_x(0,N+|E|+2) - \bareta_d(H'') N + |E''|^2 \log_2 N,
	\label{eq:inductionestimate2}
	\end{align}
	where the second inequality follows from \cref{lem:firstmoment}.

	To deduce \eqref{eq:inductionestimate1} from \eqref{eq:inductionestimate2}, it suffices to show that
	\begin{equation}
	\label{eq:HHH}
	 \bareta_d(H) \leq \bareta_d(H') + \bareta_d(H'').
	\end{equation}
	To this end, let $\bowtie'$ be an equivalence relation on $E'$ and let $\bowtie''$ be an equivalence relation on $E''$. We can define an equivalence relation $\bowtie$ on $E$ by setting $e \bowtie e'$ if and only if either $e,e' \in E'$ and $e \bowtie'  e'$ or $e,e' \in E''$ and $e \, \bowtie'' \,e'$. 
 We easily verify that 
	$\Delta(\coarse{H}{\bowtie})$ $=$ $\Delta(\coarse{H'}{\bowtie'})$ $+$ $\Delta(\coarse{H''}{\bowtie''})$,
	$|V_\circ(\coarse{H}{\bowtie})| = |V_\circ(\coarse{H'}{\bowtie'})| + |V_\circ(\coarse{H''}{\bowtie''})|$, and 
	$|E(\coarse{H}{\bowtie})| = |E(\coarse{H'}{\bowtie'})|$ $+ |E(\coarse{H''}{\bowtie''})|$, so that
	\begin{align*}
	\eta_d(\coarse{H}{\bowtie}) = \eta_d(\coarse{H'}{\bowtie'}) + \eta_d(\coarse{H''}{\bowtie''}),
	\end{align*}
	and the inequality \eqref{eq:HHH} follows by taking the minimum over $\bowtie'$ and $\bowtie''$. \qedhere

\end{proof}

We now use \cref{lem:inductionestimate} and \cref{lem:firstmomentclose} to perform an inductive analysis of $\bbW$. Although we are mostly interested in the non-buoyant case, we begin by controlling the buoyant case.

\begin{lem}[Many scales, buoyant case]\label{lem:firstmoment2}
% Let $H$ be a finite hypergraph with boundary, and let $n\geq 1$.
Let $H$ be a finite hypergraph with boundary. Let $m \geq 1$, and suppose that $x=(x_u)_{u\in \partial V} \in \bbV^{\partial V}$ are such that $2^{n-m} \leq \langle x_u x_v \rangle \leq 2^{n-1}$.  
If every subhypergraph of $H$ has a $d$-buoyant coarsening, 
then there exists a constant $c=c(\bbG,H,m)$ such that
\[\log_2\bbW^H_x(0,N) \leq  -{\bareta}_d(H) \, N + (|E \cup \partial V|^2+1)\log_2 N + c\]
for all $N\geq n$.
% if $\nu(H_0)\geq 0$.
% If $n$ is such that $\langle x \rangle$
\end{lem}

\begin{proof}
We induct on the number of edges in $H$. The claim is trivial when $E = \emptyset$. 
% For the base case, suppose that $H$ has a single edge, so that ${\bareta}_d(H)=\eta_d(H)$ and
% \[\bbW^H_x(0,N) \leq \bbW^H_x(0,n+1) + \sum_{i=n+2}^N\bbW^H_x(i-1,i).\]
% By \cref{lem:firstmoment} and \cref{lem:firstmomentclose}, we have that
% \begin{align*}\bbW^H_x(0,N) &\preceq \exp_2\left[-\bareta_d(H) n + |E \cup \partial V|^2 \log_2 n \right] + \sum_{i=n}^N \exp_2\left[-\bareta_d(H) i +  |E|^2\log_2 i \right]\\
% &\preceq \sum_{i=n}^N \exp_2\left[-\bareta_d(H) i +  |E\cup \partial V|^2\log_2 i\right] \preceq N^{|E \cup \partial V|^2} \sum_{i=n}^N 2^{-\bareta_d(H) i }
% \end{align*}
% and hence that
% \begin{equation*}
% \bbW^H_x(0,N) \preceq 
% \begin{cases}
% N^{|E\cup\partial V|^2+1} & \text{ if }\bareta_d(H)=0 \\
% N^{|E\cup \partial V|^2} 2^{-\bareta_d(H)\, N} & \text{ if } \bareta_d(H)<0.
% \end{cases}
% \end{equation*}
% The bound
% \[\log_2\bbW^H_x(0,N) \lesssim -\bareta_d(H)\,N +(|E\cup \partial V|^2+1)\log_2 N 
% \vspace{0.6em}
% \]
% holds in either case.
% This establishes the case in which $H$ has a single edge. 

% \medskip

 % where the sum is taken over proper, non-empty subsets of $E$. 
		Suppose that $|E|\geq 1$ and that
	  the claim holds for all finite hypergraphs with boundary that have fewer edges than $H$.
% 
% \medskip
% 
% 
By assumption, $\bareta_d(H') \leq 0$ for all subhypergraphs $H'$ of $H$. Thus, it follows from the induction hypothesis that
\[\log_2\bbW^{H'}_x(0,N+|E|+2) \lesssim  -{\bareta}_d(H') \, N + (|E' \cup \partial V'|^2+1)\log_2 N \]
for each proper subhypergraph $H'$ of $H$, and hence that
\begin{multline*}
% \log_2 \sum_{\xi \in \Xi^{E',m}} W^H(x,\xi) 
% \lesssim 
		     \log_2 \bbW^{H'}_x(0,N+|E|+2) 
		     - \left[\bareta_d(H) -\bareta_d(H') \right] N + |E\setminus E'|^2\log_2 N \\
		      \lesssim - \bareta_d(H) \, N + (|E' \cup \partial V|^2+1 + |E\setminus E'|^2) \log_2 N. 
\end{multline*}
(Note that the implicit constants depending on $H'$ from the induction hypothesis are bounded by a constant depending on $H$ since $H$ has only finitely many subhypergraphs.) 
Observe that whenever $E' \subsetneq E$ we have that 
\[ |E' \cup \partial V|^2+1 + |E\setminus E'|^2 \leq |E\cup \partial V|^2,\]
and so we deduce that
\begin{multline*}
% \log_2 \sum_{\xi \in \Xi^{E',m}} W^H(x,\xi) 
% \lesssim 
		     \log_2 \bbW^{H'}_x(0,N+|E|+2) 
		     - \left[\bareta_d(H) -\bareta_d(H') \right] N + |E\setminus E'|^2\log_2 N
		     \\ \lesssim - \bareta_d(H) \, N + |E \cup \partial V|^2 \log_2 N 
\end{multline*}
for every proper subhypergraph $H'$ of $H$.  Thus, we have that
\begin{align*}
\log_2 \left[ \bbW^H_x(0,N+1) - \bbW^H_x(0,N) \right]
&\leq \log_2 \left[ \bbW^H_x(0,N+|E|+2) - \bbW^H_x(0,N) \right]
\\
&\lesssim  - \bareta_d(H) \, N + |E \cup \partial V|^2 \log_2 N
\end{align*}
for all $N \geq n$, where we applied \cref{lem:inductionestimate} in the second inequality.  
Summing from $n$ to $N$ we deduce that
\begin{align*}
\bbW^H_x(0,N) - \bbW^H_x(0,n) &\preceq \sum_{i=n}^N \exp_2\left[- \bareta_d(H)\,i + |E\cup \partial V|^2\log_2 i\right]\\
% &\leq \exp_2\left[ |E\cup \partial V|^2\log_2 N \right] \sum_{i=n}^N \exp_2\left[- \bareta_d(H)\,i \right]\\
 &\preceq \exp_2\left[- \bareta_d(H)\,N + (|E\cup \partial V|^2+1)\log_2 N \right].
\end{align*}
Using \cref{lem:firstmomentclose} to control the term $\bbW^H_x(0,n)$ completes the induction.
\end{proof}

We are now ready to perform a similar induction for the non-buoyant case. Note that in this case the induction hypothesis concerns probabilities rather than expectations. This is necessary because the expectations can grow as $N\to\infty$ for the wrong reasons if $H$ has a buoyant coarsening but has a subhypergraph that does not have a buoyant coarsening (e.g.\ the tree in \cref{fig:unbalanced}).

\begin{lem}[Every scale, non-buoyant case]
\label{lem:firstmoment3}
Let $H$ be a finite hypergraph with boundary such that $E\neq \emptyset$, let $m \geq 1$, and suppose that
 $H$ has a subhypergraph that does not have any $d$-buoyant coarsenings.
Then there exist positive constants $c_1=c_1(\bbG,H,m)$ and $c_2=c_2(\bbG,H.m)$ such that
% \vspace{0.02em}
\[
\vspace{0.25em}
\log_2\P(S^H_{x}(0,\infty) > 0) \leq  -c_1 \, n + |E\cup \partial V|^2\log_2 n + c_2\]
for all 
$x=(x_u)_{u\in \partial V} \in \bbV^{\partial V}$ such that $2^{n-m} \leq \langle x_u x_v \rangle \leq 2^{n-1}$ for all $u,v \in \partial V$.
% if $\nu(H_0)\geq 0$.
% If $n$ is such that $\langle x \rangle$
\end{lem}

\begin{proof}

We induct on the number of edges in $H$. For the base case, suppose that $H$ has a single edge.
 % so that $\hat{\bareta}_d(H)=\eta_d(H)$ and
% \[\bbW^H_{x}(n,N) \leq \sum_{i=n+1}^N\bbW^H_x(i-1,i).\]
In this case we must have that $\eta_d(H)>0$, and we deduce from \cref{lem:firstmoment,lem:firstmomentclose} that 
% since in this case
\begin{align*}\bbW^H_x(0,N) &\leq \bbW^H_x(0,n)+ \sum_{i=n+1}^N \bbW^H_x(i-1,i)\\ &\preceq \exp_2\left[-\bareta_d(H)\, n + |E\cup \partial V|^2\log_2 n \right] + \sum_{i=n+1}^N \exp_2\left[ -\bareta_d(H) \, i +  |E|^2\log_2 i \right]\\
&\preceq \exp_2\left[-\bareta_d(H) \, n + |E\cup \partial V|^2\log_2 n\right],\end{align*}
so that the claim follows from Markov's inequality. 
This establishes the base case of the induction. 
% \medskip

Now suppose that $|E|>1$ and that
  the claim holds for all finite hypergraphs with boundary that have fewer edges than $H$. If $H$ has a proper subhypergraph $H'$ with $\bareta_d(H')>0$, then $S^{H'}_x(0,\infty)$ is positive if $S^{H}_x(0,\infty)$ is, and so the claim follows from the induction hypothesis, letting $c_1(\bbG,H,m)=c_1(\bbG,H',m)$ and $c_2(\bbG,H,m)=c_2(\bbG,H',m)$. 

% \medskip

Thus, it suffices to consider the case that $\bareta_d(H) >0$ but that $\bareta_d(H')\leq 0$ for every proper subhypergraph $H'$ of $H$. 
In this case, we apply \cref{lem:inductionestimate} 
to deduce that
\begin{multline*}
\log_2 \left[ \bbW^H_x(0,N+1) - \bbW^H_x(0,N) \right] \leq \log_2 \left[ \bbW^H_x(0,N+|E|+2) - \bbW^H_x(0,N) \right] \\ 
\lesssim 
% \max_{\substack{H' \text{ a proper edge-induced} \\ \mathrm{subhypergraph} }}
\max_{E'\subsetneq E}
\left\{
		    \log_2 \bbW^{H(E')}_x(0,N+|E|+2) 
		     - \left[\bareta_d(H) -\bareta_d(H(E')) \right]\, N + |E\setminus E'|^2\log_2 N 
			\right\}.
\end{multline*}
\cref{lem:firstmoment2} then yields that
\begin{multline*}
\log_2 \left[ \bbW^H_x(0,N+1) - \bbW^H_x(0,N) \right] 
			\lesssim - \bareta_d(H) \, N + (|E'\cup \partial V|^2+1 +|E\setminus E'|^2 ) \log_2 N\\
			\lesssim - \bareta_d(H) \, N + |E\cup \partial V|^2 \log_2 N.
\end{multline*}
% and hence that, since $\eta_d(H)>0$,
Finally, combining this with \cref{lem:firstmomentclose} yields that, since $\bareta_d(H)>0$,
% 
% 
 % \cref{lem:firstmoment2} and \eqref{eq:induction} imply that
\begin{align*}\bbW_x^H(0,N) &\preceq \exp_2\left[-\bareta_d(H)\, n + |E\cup \partial V|^2\log_2 n \right] +  \sum_{i=n}^N \exp_2\left[ -\bareta_d(H)\,i+|E\cup \partial V|^2\log_2 i\right]\\
&\preceq \exp_2\left[-\bareta_d(H)\, n + |E\cup \partial V|^2\log_2 n\right], \end{align*}
and the claim follows from Markov's inequality.
\end{proof}

\begin{proof}[Proof of \cref{prop:nonubiquity}]
Let $H$ be a finite hypergraph with boundary that has a subhypergraph that does not have any $d$-buoyant coarsenings, so that in particular $H$ has at least one edge. 
 \cref{lem:firstmoment3} and \cref{prop:sdim2} imply that for every $\eps>0$, there exists $x=(x_v)_{v\in \partial V}$ such that each of the points $x_v$ are in different components of $\F$ with probability at least $1-\eps$, but $H$ has probability at most $\eps$ to be faithfully present at $x$ in the  component hypergraph $\cC^{hyp}_r(\F)$. It follows that $H$ is not faithfully ubiquitous in the component graph $\cC^{hyp}_r(\F)$ a.s.

Now suppose that $H$ is a hypergraph with boundary such that every quotient $H'$ of $H$ such that $R_\bbG(H') \leq r$ has a subhypergraph that does not have any $d$-buoyant coarsenings. Note that if $H'$ is a quotient of $H$ such that $R_\bbG(H') > r$ then $H'$ is not faithfully present anywhere in $\bbG$ a.s. This follows immediately from the definition of $R_\bbG(H')$. On the other hand,  \cref{lem:firstmoment3} and \cref{prop:sdim2} imply that for every $\eps>0$, there exists $x=(x_v)_{v\in \partial V}$ such that each of the points $x_v$ are in different components of $\F$ with probability at least $1-\eps$, but, for each quotient $H'$ of $H$ with $R_\bbG(H') \leq r$, the hypergraph $H'$ has probability at most $\eps / |\{$quotients of $H\}|$ to be faithfully present at $x$ in the component hypergraph $\cC^{hyp}_r(\F)$, since $H'$ must have a subhypergraph none of whose coarsenings are $d$-buoyant by assumption. It follows by a union bound that $H$ has probability at most $\eps$ to be present in $\cC^{hyp}_r(\F)$ at this $x$. 
 % $x$ 
 It follows as above that $H$ is not ubiquitous in the component hypergraph $\cC^{hyp}_r(\F)$ a.s. 
\end{proof}

\subsection{Positive probability of robust faithful presence in low dimensions}
\label{sec:2ndmoment}

Recall that if $\bbG$ is a $d$-dimensional transitive graph, $H=(\partial V, V_\circ, E)$ is a finite hypergraph with boundary, that $r\geq 1$ and that $(x_v)_{v\in \partial V}$ is a collection of points in $\bbG$, we say that $H$ is $r$-\textbf{robustly faithfully present} at $x=(x_v)_{v\in V}$  if there is an infinite collection  $\{ \xi^i = (\xi^i_{(e,v)})_{(e,v)\in E_\bullet} : i \geq 1 \}$ 
such that $\xi^i$ is a witness for the $r$-faithful presence of $H$ at $x$ for every $i$, and $\xi^j_{(e,v)} \neq \xi^j_{(e',v')}$ for every $i, j \geq 1$ and $(e,v),(e',v') \in E_\bullet$ such that $i \neq j$.  As in the introduction, for each $M\geq 1$ we let $R_\bbG(M)$ be minimal such that it is possible for a set of diameter $R_\bbG(M)$ to intersect $M$ distinct components of the uniform spanning forest of $\bbG$, and let $R_\bbG(H) = R_\bbG(\max_{e\in E} \deg(e) ).$

\medskip

% \medskip

We say that a set $W \subset \bbV$ is \textbf{well-separated} if the vertices of $W$ are all in different components of the uniform spanning forest $\F$ with positive probability.

\begin{lemma}
\label{lem:annoyinglemma}
Let $\bbG$ be a $d$-dimensional transitive graph with $d>4$, and let $\F$ be the uniform spanning forest of $\bbG$. Then a finite set $W \subset \bbV$ is well-separated if and only if when we start a collection of independent simple random walks $\{X^v : v \in W\}$ at the vertices of $W$, the event that $\{
X^u_i : i\geq 0
\} \cap \{X^v_i :i \geq 0 \} = \emptyset$ for every distinct $u,v\in W$ has positive probability.
\end{lemma}

\begin{proof}
We will be brief since the statement is intuitively obvious from Wilson's algorithm and the details are somewhat tedious. 
The `if' implication follows trivially from Wilson's algorithm. To see the reverse implication, suppose that $W$ is well-separated and consider the paths $\{(\Gamma^v_i)_{i\geq 0} : v\in W\}$ from the vertices of $W$ to infinity in $\F$. Using Wilson's algorithm and the Green function estimate \eqref{eq:HSC}, it is easily verified that 
\begin{equation}
\label{eq:supressed}
\lim_{i\to\infty} \sum_{v\in W} \sum_{u\in W \setminus \{v\}} \left[\langle \Gamma^v_i \Gamma^u_j \rangle^{-d+4} + \sum_{j=0}^{i-1} \langle \Gamma^v_i \Gamma^u_j \rangle^{-d+2}\right] =0
\end{equation}
almost surely on the event that the vertices of $W$ are all in different components of $\F$. Let $i\geq 1$ and consider the  collection of simple random walks $Y^{v,i}$ started at $\Gamma^v_i$ and conditionally independent of each other and of $\F$ given $(\Gamma^v_i)_{v\in W}$, and let $\tilde Y^{v,i}$ be the random path formed by concatenating $(\Gamma^{v}_j)_{j=1}^i$ with $Y^{v,i}$. It follows from \eqref{eq:supressed} and Markov's inequality that 
\begin{equation}
\label{eq:supressed2}
\limsup_{i\to\infty}\P\left( \bigl\{\tilde Y^{v,i}_j : j \geq 0\bigr\} \cap 
% \hspace{-0.2cm} \bigcup_{ u \in W \setminus \{v\}} \left[ \{\Gamma^u_k : 0 \leq k \leq i\} \cup \{Y^u_j : j \geq 0\} \right]
\bigl\{\tilde Y^{u,i}_j: j \geq 0\bigr\}
 = \emptyset \text{ for every $v\in W$}\right)
 = \P(\sF(W))>0,
\end{equation}
where we recall that $\sF(W)$ is the event that all the vertices of $W$ are in different components of $\F$.  In particular, it follows that the probability appearing on the left hand side of \eqref{eq:supressed2} is positive for some $i_0\geq 0$. The result now follows since the walks $\{X^v : v \in W\}$ have a positive probability of following the paths $\Gamma^v$ for their first $i_0$ steps, and on this event their conditional distribution coincides with that of $\{\tilde Y^{v,i_0} : v\in W\}$. 
 % We conclude by noting that the walks $\{X^v: v\in W\}$ can be coupled with the paths $\{\tilde Y^{v,i_0} : v\in W\}$ so that the two collections of paths coincide with positive probability
\end{proof}
 
The goal of this subsection is to prove criteria for robust faithful presence to occur with positive probability. 
We begin with the case that $d/(d-4)$ is not an integer (i.e., $d\notin \{5,6,8\}$), which is technically simpler. The corresponding proposition for $d=5,6,8$ is given in \cref{prop:ubiquityspeciald}.

\begin{prop}\label{prop:ubiquity}
Let $\bbG$ be a $d$-dimensional transitive graph with $d>4$ such that $d/(d-4)$ is not an integer, and let $\F$ be the uniform spanning forest of $\bbG$. Let $H$ be a finite hypergraph with boundary with at least one edge, and suppose that $H$ has a coarsening all of whose subhypergraphs are $d$-buoyant. 
Then for every $r\geq R_\bbG(H)$ and every well-separated collection of points $(x_v)_{v\in \partial V}$ in $\bbV$,
 % such that 
% \[\P(\text{the vertices $x_u$ are all in different components of $\F$})>0,\]
 there is a positive probability that
% \[\P(\text{
the vertices $x_u$ are all in different components of $\F$ and that $H$ is robustly faithfully present at $x$ in $\Comp^{hyp}_r(\F)$.
% })>0.\]
 % $H$ is robustly faithfully present at $x$ in $\Comp^{hyp}_r(\F)$ with positive probability.
\end{prop}

The proof of \cref{prop:ubiquity} will employ the notion of \emph{constellations}. 
The reason we work with constellations is that a constellation of witnesses for the presence of $H$ (defined below) necessarily contains a witness for every refinement of $H$. This allows us to pass to a coarsening and work in the setting that every subhypergraph of $H$ is $d$-buoyant. 

For each finite set $A$, we define the \textbf{rooted powerset} of $A$, denoted $\cP_\bullet(A)$, to be 
\[\cP_\bullet(A) := \{(B,b) : B \text{ is a subset of $A$ and $b \in B$}\}.\]
We call a set of vertices $y=(y_{(B,b)})$ of $\bbG$ indexed by $\cP_\bullet(A)$ an $A$\textbf{-constellation}. Given an $A$-constellation $y$, we define  $\sA_r(y)$ to be the event that $y_{(B,b)}$ and $y_{(B',b')}$ are connected in $\F$ if and only if $b=b'$, and in this case they are connected by a path in $\F$ with diameter at most $r$.
We say that an $A$-constellation $y$ in $\bbG$ is \textbf{$r$-good} if it satisfies the following conditions.
 % defined to be a set of points $(x_{(B,b)})_{(B,b)\in \cP_\bullet(A)}$ satisfying the following conditions:
\begin{enumerate}
	\itemsep0.31em
	\item 
	$\langle y_{(B,b)} y_{(B',b')} \rangle \leq r$ for every $(B,b),(B',b') \in \cP_\bullet(A)$.
	 % \subseteq A$, $b\in B$ and $b' \in B$, 
	\item
$\langle y_{(B,b)} y_{(B,b')} \rangle \leq R_{\bbG}(|B|) +1$ for every $B \subseteq A$ and $b,b' \in B$, and
\item $\P(\sA_r(y)) \geq 1/r$.
\end{enumerate}

The proof of the following lemma is deferred to \cref{Sec:technical}.

\begin{lemma}\label{lem:constellations}
Let $\bbG$ be a $d$-dimensional transitive graph with $d>4$. 
Let $A$ be a finite set. Then there exists $r=r(|A|)$ such that for every vertex $x$ of $\bbG$, there exists an $r$-good $A$-constellation contained in the ball of radius $r$ around $x$.
\end{lemma}

Let $H=(\partial V,V_\circ,E)$ be a finite hypergraph with boundary with at least one edge, and let $r=r(\max_e\deg(e))$ be as in \cref{lem:constellations}. We write $\cP_\bullet(e) = \cP_\bullet(\{v \in V : v \perp e\})$ for each $e\in E$. For each $\xi=(\xi_e)_{e\in E}\in \bbV^E$ and each $e\in E$,  we let $(\xi_{(e,B,v)})_{(B,v) \in \cP_\bullet(e)}$ be an $r$-good $e$-constellation contained in the ball of radius $r$ about $\xi_e$, whose existence is guaranteed by \cref{lem:constellations}.

For each $x=(x_v)_{v\in \partial V}$ and $\xi=(\xi_e)_{e\in E}$, we define $\tilde \Witness(x,\xi)$ to be the event that  the following conditions hold:
 \begin{enumerate}[leftmargin=*]
 \item For each boundary vertex $v \in \partial V$, every point in the set $\{x_v\} \cup \{\xi_{(e,A,v)} : e\in E, (A,v) \in \cP_\bullet(e)\}$ is in the same component of $\F$,
 \item For each interior vertex $v \in V_\circ$, every point in the set $\{\xi_{(e,A,v)} : e\in E, (A,v) \in \cP_\bullet(e)\}$ is in the same component of $\F$, and
 \item For any two distinct vertices $u,v \in V$, the components of $\F$ containing the sets $\{\xi_{(e,A,u)} : e\in E, (A,u) \in \cP_\bullet(e)\}$ and $\{\xi_{(e,A,v)} : e\in E, (A,v) \in \cP_\bullet(e)\}$ are distinct. 
 \end{enumerate}
 Thus, on the event $\tilde \Witness(x,\xi)$ every refinement $H'$ of $H$ 
is $R_\bbG(H')$-faithfully present at $x$: Indeed, letting $\phi_V: V'\to V$ and $\phi_E: E'\to E$ be as in the definition of a coarsening and letting $A(e') = \{v \in V : \phi^{-1}_V(v) \perp' e'\}$ for each $e\in E'$, the collection $(\xi_{(e',v')})_{(e',v') \in E_\bullet '} = 
(\xi_{(\phi_E(e'),A(e')\phi_V(v'))})_{(e',v') \in E_\bullet '}$ is a witness for the $R_\bbG(H')$-faithfully presence of $H'$ at $x$.

\medskip

 For each $n\geq 0$, let $\Omega_x(n)$ be the set
\[\Omega_x(n) =  \left\{(\xi_{e})_{e\in E } \in \Lambda_x(n,n+1)^{E} : \langle \xi_e \xi_{e'} \rangle \geq 2^{n-C_1} \text{ for all distinct $e,e' \in E$}\right\}, \]
where $C_1=C_1(E)$ is chosen  
so that $\log_2|\Omega_x(n)|\approx nd|E|$ for all $n$ sufficiently large and all $x$. It is easy to see that such a constant exists using the $d$-dimensionality of $\bbG$.
For each $n\geq 0$ we define $\tilde S_x(n)$ to be the random variable
\begin{equation*} \tilde S_x(n) :=
 % \sum_{\xi \in K_n} 
 % \prod_{u \in P}\mathbbm{1}[\cF(x_u,\{\xi_{(e,u)} : e \perp u\})] \prod_{u \in F}\mathbbm{1}[\cF_u(\{\xi_{(e,u)} : e\perp u \})]\\
% \cdot\prod \mathbbm{1}[\neg\cF(\xi_{(e,u)},\xi_{(f,v)}) \text{ for all $u \neq v \in V$, $e,f\in E$}]
 \sum_{\xi \in \Omega_x(n)}\mathbbm{1}(\tilde\Witness(x,\xi)),\end{equation*}
 so that every refinement $H'$ of $H$ is $R_\bbG(H')$-faithfully present at $x$ on the event that $\tilde S_x(n)$ is positive for some $n\geq 0$, and every refinement $H'$ of $H$ is $R_\bbG(H')$-robustly faithfully present at $x$ on the event that $\tilde S_x(n)$ is positive for infinitely many $n\geq 0$.  
% and let $n$ be such that $\langle x_u x_v \rangle \leq 2^{n}$ for all $u,v \in P$. 

\medskip

% Our goal in this subsection is to prove the following estimate. 

The following lemma lower bounds the first moment of $\tilde S_n$.

\begin{lem}\label{prop:restrictedfirstmoment}
Let $\bbG$ be a $d$-dimensional transitive graph with $d>4$. 
Let $H$ be a finite hypergraph with boundary with at least one edge, let $\eps>0$, and suppose that $x=(x_v)_{v\in \partial V}$ is such that $\langle x_u x_v \rangle \leq 2^{n-1}$ for all $u,v \in \partial V$ and satisfies
\[
\P\Bigl(\{X^u_i : i \geq 0 \} \cap \{X^v_i : i \geq 0\} =\emptyset \text{ for every distinct $u,v\in \partial V$}\Bigr)\geq  \eps 
\]
when $\{ X^v : v \in \partial V\}$ are a collection of independent simple random walks started at $(x_v)_{v\in \partial V}$. 
 Then there exist constants $c=c(\bbG,H,\eps)$ and $n_0=n_0(\bbG,H,\eps)$ such that if $n\geq n_0$ then
 \[
 \log_2 \P(\tilde\Witness(x,\xi)) \geq -(d-4)\left(\Delta -|V_\circ|\right) \, n -c\]
 for every $\xi \in \Omega_x(n)$ and hence that
\[\log_2\E[\tilde S_x(n)] \geq -\eta_d(H)\, n - c.\]
% for all sufficiently large $n$ and all  $x=(x_u)_{u\in \partial V} \in \bbV^{\partial V}$ such that $\langle x_u x_v \rangle \leq 2^{n-1}$ for all $u,v \in \partial V$. 
\end{lem}

\medskip

The proofs of \cref{lem:constellations,prop:restrictedfirstmoment} are unfortunately rather technical, and are deferred to \cref{Sec:technical}. For the rest of this section, we will take these lemmas as given, and use them to prove \cref{prop:ubiquity}. The key remaining step is to upper bound the second moment of the random variable $\tilde S_x(n)$.

\begin{lem}[Restricted second moment upper bound]\label{lem:secondmoment}
Let $\bbG$ be a $d$-dimensional transitive graph with $d>4$ such that $d/(d-4)$ is not an integer. 
	Let $H$ be a hypergraph with boundary with at least one edge. Suppose that every subhypergraph of $H$ is $d$-buoyant. Then there exists a positive constant $c=c(\bbG,H)$ such that
	% \[\P(S_n(x) >0) \geq \delta\]
	\vspace{0.3em}
	\[
		\vspace{0.3em}
		\log_2\E[\tilde S_x(n)^2] \leq -2\eta_d(H)\, n +  c
	\]
	for all $x=(x_u)_{u\in \partial V} \in (\bbV)^{\partial V}$ and all $n$ such that $\langle x_u x_v \rangle \leq 2^{n-1}$ for all $u,v \in \partial V$. 
\end{lem}

\begin{proof}
Observe that if $\xi,\zeta \in \Omega_x(n)$ are such that the events $\tilde \sW(x,\xi)$ and $\tilde \sW(x,\zeta)$ both occur, then the following hold:
	   \begin{enumerate}
		   	\item For each $v \in V$, there is at most one $v' \in V$ such that $\xi_{(e,A,v)}$ and  $\zeta_{(e',A',v')}$ 
	   	 % \rangle \leq 2^{n-3}$. 
	   	 are in the same component of $\F$ for some (and hence every) $e,e' \in E$ and $(A,v) \in \cP_\bullet(e)$, $(A',v')\in \cP_\bullet(e')$.
	   	    	\item For each $e \in E$, there is at most one $e'$ such that $\langle \xi_e \zeta_{e'} \rangle \leq 2^{n-C_1-1}$. 
	   \end{enumerate}
	   As a bookkeeping tool to account for the first of these degrees of freedom, we define $\Phi$ be the set of functions $\phi: V_\circ \to V_\circ \cup \{\star\}$ such that the preimage $\phi^{-1}(v)$ has at most one element for each $v\in V_\circ$. We write $\phi^{-1}(v)=\star$ if $v$ is not in the image of $\phi \in \Phi$, and write $\phi(v)=v$ for every $v\in \partial V$. (Here and elsewhere, we use $\star$ as a dummy symbol so that we can encode partial bijections by functions.) For each $\phi \in \Phi$, and $\xi,\zeta \in \bbV$, define the event $\tilde\Witness_\phi(\zeta,\xi)$ to be the event that both the event $\tilde\Witness(x,\xi)\cap \tilde\Witness(x,\zeta)$ occurs, and that
	for any two distinct vertices $u,v \in V_\circ$ the components of $\F$ containing $\{\xi_{(e,A,u)} : e\in E, (A,u)\in \cP_\bullet(e)\}$ and $\{\zeta_{(e,A,v)}: e\in E, (A,v)\in \cP_\bullet(e) \}$ coincide if and only if $v =\phi(u)$. Thus, we have that 
	\[\tilde \Witness(x,\xi)\cap \tilde \Witness(x,\zeta) = \bigcup_{\phi\in \Phi} \tilde \Witness_{\phi}(\xi,\zeta)\]
	and hence that
	\[\tilde S_x(n)^2 = \sum_{\xi,\zeta \in \Omega_x(n)} \mathbbm{1}[\tilde \Witness(x,\xi) \cap \tilde \Witness(x,\zeta)] \leq \sum_{\phi \in \Phi}\sum_{\xi,\zeta\in \Omega_x(n)} \mathbbm{1}[\tilde \Witness_\phi(\xi,\zeta)].\]

	\medskip

	% Now, \cref{prop:sdim2} implies that
	% For each $e\in E$, choose an arbitrary $u_e \perp e$ and let $\xi_e=\xi_{(e,u_e)}$ and $\zeta_e = \zeta_{(e,u_e)}$. 
	It follows from \cref{prop:sdim2} that
	\vspace{0.2em}
	\begin{multline}
		\vspace{0.2em}
		\P(\tilde \Witness_\phi(\xi,\zeta)) \preceq \prod_{u \in \partial V}\langle \{x_u\}\cup\{\xi_{e} : e \perp u\},\{\zeta_{e} : e\perp u\} \rangle^{-(d-4)}
		 \cdot \prod_{u \in V_\circ,\, \phi(v)=\star}\langle\{\xi_{e} : e \perp u\}\rangle^{-(d-4)}\\
		 \cdot
		\prod_{u \in V_\circ,\, \phi^{-1}(v)=\star}\langle\{\xi_{e} : e \perp u\}\rangle^{-(d-4)} \prod_{u \in V_\circ,\, \phi(v) \neq \star}\langle\{\xi_{e} : e \perp u\} \cup \{\zeta_e : e \perp \phi(u)\}\rangle^{-(d-4)}
		\label{eq:Hphi}
	\end{multline}
	We define $R_\phi(\xi,\zeta)$ to be the expression on the right hand side of \eqref{eq:Hphi}, so that
	\[\E\left[\tilde S_x(n)^2\right] \preceq \sum_{\phi\in \Phi}\sum_{\xi,\zeta \in \Omega_x(n)} R_\phi(\xi,\zeta).\]

	\medskip
	We now account for the second of the two degrees of freedom above. 
	 Let $\Psi$ be the set of functions $\psi: E \to E \cup \{\star\}$ such that the preimage $\psi^{-1}(e)$ has at most one element for every $e\in E$. 
	% For each $\xi,\zeta \in \Omega_x(n)$ and $e \in E$, there is at most one $e' \in E$ such that $\langle \zeta_e \xi_e' \rangle \leq 2^{n-2}$.
	 % Thus, we can define a function $\psi=\psi_{\xi,\zeta} \in \Psi$ by letting $\psi(e)$ be the unique edge such that $\langle \zeta_e \xi_{\psi(e)} \rangle \leq 2^{n-2}$ if such an edge exists, and otherwise set $\psi(e)=\star$. 
	For each $\psi \in \Psi$ and $k = (k_e)_{e \in E} \in \{0,\ldots,n\}^{E}$, let
	\begin{multline*}\Omega^{\psi,k} = \\\left\{(\xi,\zeta) \in (\Omega_x(n))^2 : 
	\begin{array}{l} 2^{n-k_e} \leq \langle \zeta_e \xi_{\psi(e)} \rangle \leq 2^{n-k_e+2} \text{ for all $e\in E$ such that $\psi(e) \neq \star$,} 
	\vspace{0.3em} \\
	 \text{and }\langle \zeta_e \xi_{e'} \rangle \geq 2^{n-C_1-2} \text{ for all $e,e'\in E$ such that $e' \neq \psi(e)$} 
	\end{array}
	\right\},
	\end{multline*}
	where $C_1$ is the constant from the definition of $\Omega_x(n)$, 
	and observe that
	\begin{equation}\label{eq:Omegapsikvolume}\log_2|\Omega^{\psi,k}| \lesssim 2d|E|n - d\sum_{\psi(e)\neq\star} k_e. \end{equation}
	For each $\xi,\zeta \in \Omega_x(n)$ and $e \in E$, there is at most one $e' \in E$ such that $\langle \zeta_e \xi_{e'} \rangle \leq 2^{n-C_1-2}$, and it follows that
	\[\left(\Omega_x(n)\right)^2 = \bigcup_{\psi,k} \Omega^{\psi,k},\]
	where the union is taken over $\psi \in \Psi$ and $k \in \{0,\ldots,n\}^E$.

	\medskip

	Now, for any $\xi,\zeta \in \Omega^{\psi,k}$ and $u \in V_\circ$ with $\phi(u)\neq \star$, we have  
	that 
	% \[\log_2 \langle\{\xi_e : e \perp u\}\rangle^{-(d-4)}\langle\{\zeta_e : e \perp \phi(u)\}\rangle^{-(d-4)} \approx -(d-4)(\deg(u)+\deg(\phi(u))-2)\, n  \]
	% and that
	\begin{multline*}
	\log_2 \langle\{\xi_e : e \perp u\}\cup\{\zeta_e : e \perp \phi(u)\}\rangle^{-(d-4)} \approx -(d-4)\left(\deg(u)+\deg(\phi(u))-1\right)\, n\\ + (d-4) \sum_{e \perp u} \mathbbm{1}[\psi(e)\perp \phi(u)] \, k_e. \end{multline*}
	Meanwhile, we have that 
	\[\log_2 \langle \{\xi_e : e \perp u\} \rangle^{-(d-4)} \approx 
	\log_2 \langle \{\zeta_e : e \perp u\} \rangle^{-(d-4)} \approx
	-(d-4)(\deg(u)-1)\, n\]
	for every $u\in V_\circ$, and
	\begin{multline*}
	\log_2 \langle x_u,\{\xi_e : e \perp u\},\{\zeta_e : e\perp u\} \rangle^{-(d-4)} \approx -2(d-4)\deg(u)n + (d-4)\sum_{e \perp u}\mathbbm{1}[\psi(e)\perp u]\,k_e
	\end{multline*}
	for every $u\in \partial V$.
	Summing these estimates yields
	\begin{multline*}\log_2 R_\phi(\xi,\zeta) \approx -2(d-4)\Delta n + 2(d-4)|V_\circ|n - (d-4)|\{ v \in V_\circ : \phi(v) \neq \star\}|\,n
	 \\+ (d-4)\sum_{e}|\{u \perp e : \phi(u) \perp \psi(e)\}|k_e.
	 % \sum_{e}\mathbbm{1}[\psi(e) \neq e,\; \text{$\psi(e) \cap e \neq \emptyset$\,}]\,k_e
	% + (d-4)\sum_{e}\mathbbm{1}[\psi(e) = e]\,k_e \deg(e).
	\end{multline*}
	% which simplifies to
	%  \begin{multline*} \log_2 R_\phi(\xi,\zeta) \approx  -2 \eta_d(H)\, n - (d-4)|\{ v \in V_\circ : \phi(v) \neq \star\}|\,n\\
	%  + (d-4)\sum_{e}|\{u \perp e : \phi(u) \perp \psi(e)\}|k_e.\end{multline*}
	Thus, using the volume estimate \eqref{eq:Omegapsikvolume}, we have that
	\begin{multline*}\log_2 \sum_{(\xi,\zeta) \in \Omega^{\psi,k}}R_\phi(\xi,\zeta) \lesssim
	-2\eta_d(H)n - (d-4)|\{ u \in V_\circ : \phi(u) \neq \star\}|\,n
	 \\+ (d-4)\sum_{\psi(e)\neq \star}|\{u \perp e : \phi(u) \perp \psi(e)\}|k_e - d \sum_{\psi(e)\neq\star}k_e.
	\end{multline*}
	Observe that for every $\psi\in \Psi$ and $e\in E$, we have that
	\begin{multline*}
	\sum_{k_e=0}^n \exp_2\left(\left[(d-4)|\{u\perp e:\phi(u)\perp \psi(e)\}| - d\right]k_e \right) \\\preceq \begin{cases}
	\exp_2\left(\left[(d-4)|\{u\perp e:\phi(u)\perp \psi(e)\}| - d\right]n\right) & \text{ if $(d-4)|\{u\perp e:\phi(u)\perp \psi(e)\}|  >d$}\\
	n & \text{ if $(d-4)|\{u\perp e:\phi(u)\perp \psi(e)\}|  =d$}\\
	1 & \text{ if $(d-4)|\{u\perp e:\phi(u)\perp \psi(e)\}| <d$}.
	\end{cases}
	\end{multline*}
	Thus, summing over $k$, we see that for every $\psi \in \Psi$ and $\phi \in \Phi$ we have that
	\begin{multline}
	\log_2 \sum_{k\in\{0,\ldots,n\}^E}\sum_{(\xi,\zeta) \in \Omega^{\psi,k}}R_\phi(\xi,\zeta) \lesssim -2\eta_d(H)\, n - (d-4)|\{ u \in V_\circ : \phi(v) \neq \star\}| n\\
	 % \ni v \text{ for some $e \perp v$}\}|
	 % \\+ (d-4)\sum_{e}\mathbbm{1}[\psi(e) \neq e,\, \text{ $\psi(e)$ shares an endpoint with $e$}]k_e
	% + (d-4)\sum_{e}\mathbbm{1}[\psi(e) = e]\deg(e) -d\sum_{e}\mathbbm{1}[\psi(e) \neq \star].
	+\sum_{e\in E}\left[(d-4)|\{u \perp e : \phi(u) \perp \psi(e)\}|-d\right]
	\mathbbm{1}\left(|\{u \perp e : \phi(u) \perp \psi(e)\}| >  d/(d-4) \right) n\\
	+\sum_{e\in E}\mathbbm{1}\left(|\{u\perp e : \phi(u) \perp \psi(e)\}| = d/(d-4)\right)\log_2 n.
	\label{eq:Rallterms}
	% -d \sum_{e:\psi(e)\neq \star}k_e.
	\end{multline}

Since $d/(d-4)$ is not an integer, the last term is zero, so that if we define $Q : \Phi \times \Psi \to \R$ by
	\begin{multline}Q(\phi,\psi) = - (d-4)|\{ u \in V_\circ : \phi(v) \neq \star\}|\\
	 % \ni v \text{ for some $e \perp v$}\}|
	 % \\+ (d-4)\sum_{e\in E}\mathbbm{1}[\psi(e) \neq e,\, \text{ $\psi(e)$ shares an endpoint with $e$}]k_e
	% + (d-4)\sum_{e\in E}\mathbbm{1}[\psi(e) = e]\deg(e) -d\sum_{e\in E}\mathbbm{1}[\psi(e) \neq \star].
	+\sum_{e\in E}\left[(d-4)|\{u \perp e : \phi(u) \perp \psi(e)\}|-d\right]
	\mathbbm{1}[|\{u \perp e : \phi(u) \perp \psi(e)\}| >  d/(d-4) ],
	\label{eq:Qdef2}
	\end{multline}
	then we have that
	\begin{equation*}
	\log_2 \sum_{k \in \{0,\ldots,n\}^E} \sum_{(\xi,\zeta)\in \Omega^{\psi,k}}R_\phi(\xi,\zeta) 
	 % -2\eta_d(H) + Q(\phi,\psi)\\
	% &\hspace{1.7cm} + \sum_{e\in E}\mathbbm{1}\left(|\{u\perp e : \phi(u) \perp \psi(e)\}| = d/(d-4)\right)\log_2 n\\
	% & \leq
	 \lesssim -2\eta_d(H)n + Q(\phi,\psi)n.
	 % |\{e\in E : \deg(e) = d/(d-4)\}|\log_2 n.
	\end{equation*}
	Thus, since $|\Phi \times \Psi|$ does not depend on $n$, we have that 
	\begin{multline*}\log_2 \E[\tilde S_x(n)^2] \lesssim \log_2  \sum_{\phi \in \Phi}\sum_{\psi \in \Psi} \sum_{k \in \{0,\ldots,n\}^E} \sum_{(\xi,\zeta)\in \Omega^{\psi,k}}R_\phi(\xi,\zeta) \\
	\lesssim 
\max_{\phi,\psi} \log_2 \sum_{k \in \{0,\ldots,n\}^E} \sum_{(\xi,\zeta)\in \Omega^{\psi,k}}R_\phi(\xi,\zeta) \lesssim
	-2\eta_d(H)n + \max_{\phi,\psi}Q(\phi,\psi)n, 
	\end{multline*}
and so it suffices to prove that $Q(\phi,\psi)\leq 0$ for every $(\phi,\psi)\in \Phi\times\Psi$.

To prove this, first observe that we can bound
\begin{multline*}Q(\phi,\psi) \leq \tilde Q(\phi):=  - (d-4)|\{ u \in V_\circ : \phi(v) \neq \star\}|\\
	+\sum_{e\in E}\left[(d-4)|\{u \perp e : \phi(u) \neq \star \}|-d\right]
	\mathbbm{1}[|\{u \perp e : \phi(u) \neq \star \}| >  d/(d-4) ].
\end{multline*}
Let $H'$ be the subhypergraph of $H$ with boundary vertices given by the boundary vertices of $H$, edges given by the set of edges of $H$ that have $|\{u\perp e:\phi(u)\neq \star\}|>d/(d-4)$, and interior vertices given by the set of interior vertices $u$ of $H$ for which $\phi(u)\neq \star$ and $\phi(u)\perp e$ for some $e\in E'$. Then we can rewrite
\begin{equation}
\tilde Q(\phi) = \eta_d(H')-(d-4)\bigl|\{v\in V_\circ : \phi(v)\neq \star\}\setminus V'\bigr| \leq 0,
\label{eq:bad2}
\end{equation}
where the second inequality follows by the assumption that every subhypergraph of $H$ is $d$-buoyant. This completes the proof.
\end{proof}

\begin{proof}[Proof of \cref{prop:ubiquity}]
Suppose that the finite hypergraph with boundary $H$ has a $d$-optimal coarsening all of whose subhypergraphs are $d$-buoyant. Then the lower bound on the square of the first moment of $\tilde S^{H'}_x(n)$ provided by \cref{prop:restrictedfirstmoment} and the upper bound on the second moment of $\tilde S^{H'}_x(n)$ provided by \cref{lem:secondmoment} coincide, so that the Cauchy-Schwarz inequality implies that
	\[\P\left(\tilde S^{H'}_x(n) > 0\right) \geq \myfrac[0.5em]{\E\left[\tilde S^{H'}_x(n)\right]^2}{\E\left[\tilde S^{H'}_x(n)^2\right]} \succeq 1\]
	for every $n$ such that $\langle x_u x_v \rangle \leq 2^{n-1}$ for every $u,v \in \partial V$. It follows from Fatou's lemma that
	\[\P\left(\tilde S^{H'}_x(n)>0 \text{ for infinitely many $n$}\right) \geq \limsup_{n\to\infty} \P\left(\tilde S^{H'}_x(n) > 0\right) \succeq 1, \]
	so that $H$ is robustly faithfully present at $x$ with positive probability as claimed. 
\end{proof}

\subsubsection{The cases $d=5,6,8$.}
\label{sec:speciald}
We now treat the cases in which $d/(d-4)$ is an integer. This requires somewhat more care owing to the possible presence of the logarithmic term in \eqref{eq:Rallterms}. Indeed, we will only treat certain special `building block' hypergraphs directly via the second moment method. We will later build other hypergraphs out of these special hypergraphs in order to to prove the main theorems.

Let $H=(\partial V,V_\circ, E)$ be a finite hypergraph with boundary. We say that a subhypergraph $H'=(\partial V',V_\circ',E')$ of $H$ is \textbf{bordered} if $\partial V'=\partial V$ and every vertex $v\in V \setminus V'$ is incident to at most one edge in $E'$. For example, every full subhypergraph containing every boundary vertex is bordered. We say that a subhypergraph of $H$ is \textbf{proper} if it is not equal to $H$ and \textbf{non-trivial} if it has at least one edge. We say that $H$ is $d$\textbf{-basic} if it does not have any edges of degree less than or equal to $d/(d-4)$ and does not contain any proper, non-trivial bordered subhypergraphs $H'$ with $\eta_d(H')=0$.

% If $\bbG$ is a $d$-dimensional transitive graph, $H=(\partial V, V_\circ, E)$ is a finite hypergraph with boundary, $r\geq 1$ and $(x_v)_{v\in \partial V}$ is a collection of points in $\bbG$, we say that $H$ is \textbf{constellationally robustly faithfully present} at $x=(x_v)_{v\in V}$  if there is an infinite collection  $\{ \xi^i = (\xi^i_{e})_{e\in E_\bullet} : i \geq 1 \}$ 
% such that the event $\tilde \sW(x,\xi^i)$ occurs for every $i\geq 1$ and $\xi^i_{e} \neq \xi^j_{e'}$ for every $1\leq i< j$ and $e,e' \in E$. Thus, if $H$ is constellationally robustly faithfully present at $x$ then every refinement $H'$ of $H$ is $R_\bbG(H')$-robustly faithfully present at $x$. 

\begin{prop}\label{prop:ubiquityspeciald}
Let $\bbG$ be a $d$-dimensional transitive graph with $d\in\{5,6,8\}$, and let $\F$ be the uniform spanning forest of $\bbG$. Let $H$ be a finite hypergraph with boundary with at least one edge.
Suppose additionally that one of the following assumptions holds:
\begin{enumerate}
	\item
$H$ is a refinement of a hypergraph with boundary that has exactly one edge, the unique edge contains exactly $d/(d-4)$ boundary vertices, and every interior vertex is incident to the unique edge.
\end{enumerate}
or
\begin{enumerate}
\item[(2)]  
$H$ has a $d$-basic coarsening with more than one edge, all of whose subhypergraphs are $d$-buoyant.
\end{enumerate}
% and let $r \geq R_\bbG(M)$. 
Then for every $r\geq R_\bbG(H)$ and every well-separated collection of points $(x_v)_{v\in \partial V}$ in $\bbV$ 
% such that 
% \[\P(\text{the vertices $x_u$ are all in different components of $\F$})>0,\] 
there is a positive probability that
% \[\P(\text{
the vertices $x_u$ are all in different components of $\F$ and that $H$ is  
% constellationally 
robustly faithfully present at $x$.
% })>0.\]
 % $H$ is robustly faithfully present at $x$ in $\Comp^{hyp}_r(\F)$ with positive probability.
\end{prop}

% Note that the hypergraphs described in case $(1)$ of \cref{prop:ubiquityspeciald} are exactly those hypergraphs with boundary that have a coarsening that has exactly one edge and all of whose subhypergraphs are $d$-buoyant.

The proof of \cref{prop:ubiquityspeciald} will apply the following lemma, which is the analogue of \cref{lem:secondmoment} in this context.

\begin{lem}
\label{lem:secondmomentspeciald}
Let $\bbG$ be a $d$-dimensional transitive graph with $d\in\{5,6,8\}$.
Let $H$ be a hypergraph with boundary with at least one edge such that every subhypergraph of $H$ is $d$-buoyant.
\begin{enumerate}
	\itemsep0.5em
\item If $H$ has exactly exactly one edge, this unique edge is incident to exactly $d/(d-4)$ boundary vertices, and every interior vertex is incident to this unique edge, 
 then there exists a constant $c=c(\bbG,H)$ such that
\vspace{0.3em}
	\[
		\vspace{0.3em}
		\log_2\E[\tilde S_x(n)^2] \leq \log_2 n + c
	\]
for all $x=(x_u)_{u\in \partial V} \in (\bbV)^{\partial V}$ and all $n$ such that $\langle x_u x_v \rangle \leq 2^{n-1}$ for all $u,v \in \partial V$. 
\item If $H$  is $d$-basic, then there exists a constant $c=c(\bbG,H)$ such that
\vspace{0.3em}
	\[
		\vspace{0.3em}
		\log_2\E[\tilde S_x(n)^2] \leq  -2\eta_d(H) +c
	\]
for all $x=(x_u)_{u\in \partial V} \in (\bbV)^{\partial V}$ and all $n$ such that $\langle x_u x_v \rangle \leq 2^{n-1}$ for all $u,v \in \partial V$. 
\end{enumerate}
\end{lem}

\begin{proof} Note that in both cases we have that every subhypergraph of $H$ is $d$-buoyant. 
We use the notation of the proof of \cref{prop:ubiquity}. As in equation \eqref{eq:Rallterms} of that proof, we have that
	\begin{multline}
	\log_2 \sum_{k\in\{0,\ldots,n\}^E}\sum_{(\xi,\zeta) \in \Omega^{\psi,k}}R_\phi(\xi,\zeta)\\ \lesssim -2\eta_d(H)\, n +Q(\phi,\psi)n 
	+\left|\left\{e \in E : \left|\left\{u\perp e : \phi(u) \perp \psi(e)\right\}\right| = d/(d-4)\right\}\right|\log_2 n,
	% -d \sum_{e:\psi(e)\neq \star}k_e.
	\label{eq:Rallterms2}
	\end{multline}
	where $Q(\phi,\psi)$ is defined as in \eqref{eq:Qdef2}. Moreover, the same argument used in that proof shows that $Q(\phi,\psi)\leq 0$ for every $(\phi,\psi)\in\Phi\times\Psi$.  In case $(1)$ of the lemma, in which $H$ has a single edge, we immediately obtain the desired bound since $\eta_d(H)=0$ and the coefficient of the $\log_2 n$ term is either $0$ or $1$.

Now suppose that $H$ is $d$-basic. Let $L(\phi,\psi)$ be the coefficient of $\log_2 n$ in \eqref{eq:Rallterms2}.
Note that $H$ cannot have an edge whose intersection with $\partial V$ has  $(d-4)/d$ elements or more, since otherwise the subhypergraph $H'$ of $H$ with that single edge and with no internal vertices is proper, bordered, and has $\eta_d(H')\geq 0$. Thus, we have that if $\phi_0$ is defined by $\phi_0(v)=\star$ for every $v\in V_\circ$ then 
\[L(\phi_0,\psi)\leq 
\left|\left\{e \in E : |\psi(e) \cap \partial V| \geq d/(d-4)\right\}\right|
=0\]
for every $\psi \in \Psi$. 

Let $\operatorname{Isom} \subseteq \Phi\times \Psi$ be the set of all $(\phi,\psi)$ such that $\phi(u)\perp \psi(e)$ for every $e\in E$ and $v\perp e$. Since $H$ is $d$-basic we have that if $(\phi,\psi)\in \operatorname{Isom}$ then
\[
L(\phi,\psi) =
\left|\left\{e \in E : \deg(e) = d/(d-4)\right\}\right| =0.
\]

% Let $L(\phi,\psi)$ be the coefficient of $\log_2 n$ in \eqref{eq:Rallterms2}.
  We claim that  $Q(\phi,\psi)\leq -(d-4)$ unless either $\phi=\phi_0$ or $(\phi,\psi)\in \operatorname{Isom}$.  % $L(\phi,\psi)>0$ then $Q(\psi,\phi) \leq -C_d$
   Once proven this will conclude the proof, since we will then have that 
\begin{equation*}
	\log_2 \sum_{k\in\{0,\ldots,n\}^E}\sum_{(\xi,\zeta) \in \Omega^{\psi,k}}R_\phi(\xi,\zeta)\\ \lesssim -2\eta_d(H)\, n  + \max\{ -(d-4) n + |E|\log_2 n, 0\} \lesssim -2\eta_d(H) n
	% -d \sum_{e:\psi(e)\neq \star}k_e.
	\end{equation*}
for every $(\phi,\psi)\in\Phi\times \Psi$, from which we can conclude by summing over $\Phi\times\Psi$ as done previously.

% Since $d-4$ divides $d$, the $d$-apparent weight of every hypergraph with boundary is a multiple of $d-4$, and so we must have that $\eta_d(H')\leq -d+4 <0$ for every proper bordered subhypergraph $H'$ of $H$. 
% To see this, first observe that the inequality
% \begin{multline*}Q(\phi,\psi) \leq \tilde Q(\phi):=  - (d-4)|\{ u \in V_\circ : \phi(v) \neq \star\}|\\
% 	 % \ni v \text{ for some $e \perp v$}\}|
% 	 % \\+ (d-4)\sum_{e\in E}\mathbbm{1}[\psi(e) \neq e,\, \text{ $\psi(e)$ shares an endpoint with $e$}]k_e
% 	% + (d-4)\sum_{e\in E}\mathbbm{1}[\psi(e) = e]\deg(e) -d\sum_{e\in E}\mathbbm{1}[\psi(e) \neq \star].
% 	+\sum_{e\in E}\left[(d-4)|\{u \perp e : \phi(u) \neq \star \}|-d\right]
% 	\mathbbm{1}[|\{u \perp e : \phi(u) \neq \star \}| >  d/(d-4) ].
% \end{multline*}
% can be improved as follows: Setting
% \[C_d = \min\Bigl\{\lfloor (2d-4)/(d-4)\rfloor(d-4)-d, d-4 \Bigr\} >0,\]
% we have that
% \begin{multline}
% Q(\phi,\psi) \leq  \tilde Q(\phi) \\
% 	-C_d \mathbbm{1}\Bigl[ \exists \text{$e\in E$ s.t.\ $|\{u\perp e :\phi(u)\neq \star\}|>(d-4)/d$ and $v\perp e$ s.t.\ $\phi(u)\notin \psi(e)\cup\{\star\}$}\Bigr]
% 	\label{eq:bad1}
% \end{multline}
% Using this bound together with \eqref{eq:bad2}, we have that

We first prove that $Q(\phi,\psi)\leq -(d-4)$ unless either $\phi=\phi_0$ or $\phi(v) \neq \star$ for every $v\in V$. 
Note that since $d-4$ divides $d$, the $d$-apparent weight of every hypergraph with boundary is a multiple of $d-4$, and so we must have that $\eta_d(H')\leq -(d-4)$ for every subhypergraph $H'$ of $H$ with $\eta_d(H')<0$. 
As in \eqref{eq:bad2}, we have that 
$Q(\phi,\psi)\leq \eta_d(H')$, 
% \begin{multline*}
% Q(\phi,\psi) \leq  \eta_d(H')  \\
% 	-C_d \mathbbm{1}\Bigl[ \exists \text{$e\in E$ s.t.\ $|\{u\perp e :\phi(u)\neq \star\}|>(d-4)/d$ and $v\perp e$ s.t.\ $\phi(u)\notin \psi(e)\cup\{\star\}$}\Bigr]\\
% 	-(d-4)\mathbbm{1}\Bigl[ \exists \text{$u\in V_\circ$ such that $\phi(v) \neq \star$ and $v\notin V'$}\Bigr]\\
% \end{multline*}
 where $H'=H'(\phi)$ is the subhypergraph of $H$ with boundary vertices given by the boundary vertices of $H$, edges given by the  set of edges of $H$ that have $|\{u\perp e:\phi(u)\neq \star\}|\geq d/(d-4)$, and interior vertices given by the set of interior vertices $u$ of $H$ for which $\phi(u)\neq \star$ and $\phi(u)\perp e$ for some $e\in E'$. 

% Thus, to complete the proof, it suffices to show that
We claim that 
 if $\phi$ is  such that $\eta_d(H')=0$ then $H'$ is bordered, and consequently is either equal to $H$ or does not have any edges by our assumptions on $H$. To see this, suppose for contradiction that $H'$ is not bordered, so that there exists a vertex $v\in V_\circ \setminus V_\circ'$ that is incident to more than one edge of $H'$. Let $H''$ be the subhypergraph of $H'$ obtained from $H'$ by adding the vertex $v$. Then we have that $|E(H'')|=|E(H')|$, $|V_\circ(H'')|=|V_\circ(H')|+1$ and $\Delta(H'')\geq \Delta(H'')+2$, and consequently that $\eta_d(H'')\geq \eta_d(H')+(d-4)$. Since every subhypergraph of $H$ is $d$-buoyant, we have that $\eta_d(H'')\leq 0$ and consequently that $\eta_d(H')\leq -(d-4)$, a contradiction. This establishes that 
 $Q(\phi,\psi)\leq -(d-4)$ unless either $\phi=\phi_0$ or $\phi(v) \neq \star$ for every $v\in V$, as claimed. 

 It remains to show that if $\phi(v) \neq \star$ for every $v\in V$ then $Q(\phi_1,\psi) \leq -(d-4)$ unless $(\phi,\psi)\in \operatorname{Isom}$. 
Since every edge of $H$ has degree strictly larger than $d/(d-4)$, we have that
\begin{multline*}
\left[(d-4)|\{u \perp e : \phi(u) \perp \psi(e)\}|-d\right]
	\mathbbm{1}[|\{u \perp e : \phi(u) \perp \psi(e)\}| >  d/(d-4) ]\\ \leq 
	\left[(d-4)\deg(e)-d\right]
	-(d-4) 
\end{multline*}
for every $e\in E$ and every $(\phi,\psi)\in \Phi\times\Psi$ such that $|\{u\perp e : \phi(u) \perp \psi(e)\}| < \deg(e)$. It follows easily from this and the definition of $Q(\phi,\psi)$ that if $\phi$ has $\phi(v) \neq \star$ for every $v\in V$, then
\begin{align*}
Q(\phi,\psi) \leq 
 \eta_d(H) - (d-4) |\{ e \in E : |\{u \perp e : \phi(u) \perp \psi(e)\}| < \deg(e) \}|.
\end{align*}
Since $\eta_d(H)\leq 0$ by assumption, it follows that $Q(\phi,\psi)\leq -(d-4)$ unless $(\phi,\psi)\in \operatorname{Isom}$. This concludes the proof.
\end{proof}

\cref{lem:secondmoment} (together with \cref{prop:restrictedfirstmoment}) is already sufficient to yield case (2) of \cref{prop:ubiquityspeciald}. To handle case (1), we will require the following additional estimate.
% , which holds for all $d>4$.

\begin{lem}[Different scales are uncorrelated]\label{lem:secondmoment2}
Let $\bbG$ be a $d$-dimensional transitive graph with $d>4$. 
Let $H$ be a hypergraph with boundary.
  Then there exists a positive constant $c=c(\bbG,H,r)$ such that
% \[\P(S_n(x) >0) \geq \delta\]
\[\log_2\E[\tilde S_x(n) \tilde S_{x}(n+m)] \leq -\eta_d(H)\, (2n+m) + c\]
for all $x=(x_u)_{u\in \partial V} \in (\bbV)^{\partial V}$, all $m\geq 2$, and all $n$ such that $\langle x_u x_v \rangle \leq 2^{n-1}$ for all $u,v \in \partial V$. 
% If $n$ is such that $\langle x \rangle$
\end{lem}

\begin{proof}
Let $\Phi$ and $\tilde \Witness_\phi(\xi,\zeta)$ be defined as in the proof of \cref{lem:secondmoment}.
 % for each $\xi,\zeta \in \bbV^{E_\bullet}$. 
 % For each $e\in E$, choose an arbitrary $u_e \perp e$ and let $\xi_e = \xi_{(e,u(e))}$ and $\zeta_e = \zeta_{(e,u(e))}$.
  For every $\xi \in \Omega_x(n)$ and $\zeta \in \Omega_x(n+m)$, we have that all distances relevant to our calculations are on the order of either $2^n$ or $2^{n+m}$. That is,
\begin{align*}
\log_2 \langle \xi_e \xi_{e'} \rangle,\,  \log_2 \langle \xi_e x_v \rangle  \approx n
 % \approx n\\
 % \approx n
% \end{align*}
\quad  \text{ and } \quad
% \[
\log_2 \langle \xi_e \zeta_{e'} \rangle,\, \log_2 \langle \zeta_e \zeta_{e'} \rangle,\, \log_2 \langle \zeta_e x_v \rangle \approx n+m
\end{align*}
 for all $e,e' \in E$ and $v\in \partial V$. Thus, using \eqref{eq:Hphi}, can estimate
\begin{multline*}\frac{1}{d-4}\log_2\P(\tilde \Witness_\phi(\xi,\zeta)) \lesssim -\sum_{u\in\partial V}|\{e \in E : e\perp u\}|\,(2n+m) \\-\sum_{u \in V_\circ,\, \phi(u) =\star} (|\{e\in E : e \perp u\}|-1)\, n  
-\sum_{u \in V_\circ,\, \phi^{-1}(u) =\star} (|\{e\in E : e \perp u\}|-1)\, (n+m) 
\\ -\sum_{u \in V_\circ,\, \phi(u) \neq \star} \left(\left|\{e\in E : e \perp u\}\right|n-n +\left|\{e\in E : e \perp \phi(u)\}\right|(n+m)\right)\,
\\ = -\Delta (2n+m) + |V_\circ|\,(2n+m) - |\{v \in V_\circ: \phi(v)\neq \star \}|\, (n+m),
 % \cdot \prod_{u \in V_\circ,\, \phi(v)=\star}\langle\{\xi_{e} : e \perp u\}\rangle^{-}\\
 % \cdot
% \prod_{u \in V_\circ,\, \phi^{-1}(v)=\star}\langle\{\xi_{e} : e \perp u\}\rangle^{-(d-4)} \prod_{u \in V_\circ,\, \phi(v) \neq \star}\langle\{\xi_{e} : e \perp u\} \cup \{\zeta_e : e \perp \phi(u)\}\rangle^{-(d-4)}
\end{multline*}
which is maximized when $\phi(v)=\star$ for all $v\in V_\circ$. Now, since 
\[\log_2|\Omega_x(n)\times \Omega_x(n+m)| \lesssim |\Lambda(n-1,n)^{E} \times \Lambda(n+m-1,n+m)^{E}| \lesssim d(2n+m),\]
we deduce that
\begin{align*}
\log_2\E[\tilde S_x(n)\tilde S_x(n+m)] 
&\leq \log_2|\Omega_x(n)\times \Omega_x(n+m)| + \log|\Phi| \\&\hspace{2cm}+ \max\{\P(\tilde \Witness_\phi(\xi,\zeta)): \xi \in \Omega_x(n),\zeta\in \Omega_x(n+m), \phi \in \Phi \}
\\ & \lesssim d|E|(n+m) - (d-4)\Delta(2n+m) +(d-4) |V_\circ|(2n+m)\\ &= -\eta_d(H) (2n+m)
\end{align*}
as claimed.
\end{proof}

\begin{proof}[Proof of \cref{prop:ubiquityspeciald} given \cref{lem:constellations,prop:restrictedfirstmoment}]
	The second case, in which $H$ has a $d$-basic coarsening with more than one edge all of whose subhypergraphs are $d$-buoyant, follows from \cref{lem:constellations,prop:restrictedfirstmoment,lem:secondmomentspeciald} exactly as in the proof of \cref{prop:ubiquity}. Now suppose that $H$ is a refinement of a hypergraph with boundary $H'$ that has $d/(d-4)$ boundary vertices and a single edge incident to every vertex. Then $\eta_d(H')=0$ and every subhypergraph of $H'$ is $d$-buoyant. Applying \cref{prop:restrictedfirstmoment,lem:secondmomentspeciald,lem:secondmoment2}, we deduce that
	\[\E\left[\sum_{k=n}^{2n} \tilde S^{H'}_x(2k)\right] \succeq n, \quad \text{ and } \quad
	\E\left[\left(\sum_{k=n}^{2n} \tilde S^{H'}_x(2k)\right)^2\right] \preceq n^2, \]
	for every $n$ such that $\langle x_u x_v \rangle \leq 2^{n-1}$ for every $u,v \in \partial V$, from which it follows by Cauchy-Schwarz that
	\[\P\left( \sum_{k=n}^{2n}\tilde S^{H'}_x(2k) >0 \right)\succeq 1.\]
	for every $n$ such that $\langle x_u x_v \rangle \leq 2^{n-1}$ for every $u,v \in \partial V$.
	The proof can now be concluded as in the proof of \cref{prop:ubiquity}. 
\end{proof}

% % \section{Proof of \cref{thm:indist}}

\subsection{Proof of Lemmas \ref{lem:constellations} and 
\ref{prop:restrictedfirstmoment}
}
\label{Sec:technical}

% We shall require the following estimate of Hebisch and Saloff-Coste \cite{HebSaCo93}. The estimate is classical in the case $\bbG=\Z^d$.
% \[G(x,y) = \]

% \begin{defn}
In this section we prove \cref{lem:constellations,prop:restrictedfirstmoment}. 
We begin with some background on random walk estimates. Given a graph $G$ and a vertex $u$ of $G$, we write $\bP_u$ for the law of the random walk on $G$ started at $u$.
Let $G$ be a graph, and let $p_n(x,y)$ be the probability that a random walk on $G$ started at $x$ is at $y$ at time $n$. Given positive constants $c$ and $c'$, we say that $G$ satisfies $(c,c')$\textbf{-Gaussian heat kernel estimates} if 
\begin{align} \frac{c}{|B(x,n^{1/2})|}e^{-c d(x,y)^2/n} \leq p_n(x,y) + p_{n+1}(x,y) \leq \frac{c'}{|B(x,n^{1/2})|}e^{-d(x,y)^2/(c' n)} 
\label{eq:GHKE}
\end{align}
for every $n\geq 0$ and every pair of vertices $x,y$ in $G$ with $d(x,y)\leq n$. We say that $G$ satisfies Gaussian heat kernel estimates if it satisfies $(c,c')$-Gaussian Heat Kernel Estimates for some positive constants $c$ and $c'$.

\begin{thm}[Hebisch and Saloff-Coste \cite{HebSaCo93}]
	\label{thm:HSCGreen}
	Let $\bbG$ be a $d$-dimensional transitive graph. Then $G$ satisfies Gaussian heat kernel estimates.
\end{thm}

Hebisch and Saloff-Coste proved their result only for Cayley graphs, but the general case can be proven by similar methods\footnote{In fact, the general case can also be deduced from the case of Cayley graphs, since if $\bbG$ is a $d$-dimensional transitive graph then the product of $G$ with a sufficiently large complete graph is a Cayley graph \cite{trofimov1985graphs,godsil1989note}, but taking such a product affects the random walk in only a very trivial way.}, see e.g.\ \cite[Corollary 14.5 and Theorem 14.19]{Woess}.

Now, recall that two graphs $G=(V,E)$ and $G'=(V',E')$ are said to be $(\alpha,\beta)$\textbf{-rough isometric} if there exists a function $\phi:V \to V'$ such that the following conditions hold.
\begin{enumerate}
	\item $\phi$ roughly preserves distances: The estimate \[\alpha^{-1} d(x,y) - \beta \leq d'(\phi(x),\phi(y)) \leq \alpha d(x,y) + \beta\] holds for all $x,y \in V$.
	\item $\phi$ is roughly surjective: For every $x \in V'$, there exists $y \in V$ such that $d'(x,\phi(y)) \leq \beta$.
\end{enumerate}
The following stability theorem for Gaussian heat kernel estimates follows from the work of Delmotte \cite{delmotte1999parabolic}; see also \cite[Theorem 3.3.5]{KumFlour}.

\begin{thm}\label{thm:GHKEstability}
	Let $G$ and $G'$ be $(\alpha,\beta)$-roughly isometric graphs for some positive $\alpha,\beta$, and suppose that the degrees of $G$ and $G'$ are bounded by $M<\infty$ and that $G$ satisfies  $(c,c')$-Gaussian heat kernel estimates for some positive $c,c'$. Then there exist $\tilde c = \tilde c(\alpha,\beta,M,c,c')$ and 
	$\tilde c' = \tilde c'(\alpha,\beta,M,c,c')$ such that $G'$ satisfies $(\tilde c, \tilde c')$-Gaussian heat kernel estimates.
\end{thm}

Recall that a function $h:V\to\R$ defined on the vertex set of a graph is said to be \textbf{harmonic} on a set $A \subseteq V$ if 
\[h(v)=\frac{1}{\deg(v)} \sum_{u \sim v} h(u)\]
for every vertex $v\in A$, where the sum is taken with appropriate multiplicities if there are multiple edges between $u$ and $v$. The graph $G$ is said to satisfy an \textbf{elliptic Harnack inequality} if for every $\alpha>1$, there exist a constant $c(\alpha) \geq 1$ such that \[c(\alpha)^{-1} \leq h(v)/h(u) \leq c(\alpha)\] for every two vertices $u$ and $v$ of $G$ and every positive function $h$ that is harmonic on the set \[\left\{w \in V : \min \{d(u,w),d(w,v)\} \leq \alpha d(u,v)\right\},\]
in which case we say that $G$ satisfies an elliptic Harnack inequality with constants $c(\alpha)$. 

The following theorem also follows from the work of Delmotte \cite{delmotte1999parabolic}, and was implicit in the earlier work of e.g.\ Fabes and Stroock \cite{FabStro86}; see also \cite[Theorem 3.3.5]{KumFlour}. Note that these references all concern the \emph{parabolic} Harnack inequality, which is stronger than the elliptic Harnack inequality. 

\begin{thm}\label{thm:GHKEimpliesEHI}
	Let $G$ be a graph. If $G$ satisfies $(c_1,c_1')$-Gaussian heat kernel estimates, then there exists $c_2(\alpha)=c_2(\alpha,c_1)$ such that $G$ satisfies an elliptic Harnack inequality with constants $c_2(\alpha)$.
\end{thm}

We remark that the elliptic Harnack inequality has recently been shown to be stable under rough isometries in the breakthrough work of Barlow and Murugan \cite{barlow2016stability}. 

\medskip

Recall that a graph is said to be \textbf{$d$-Ahlfors regular}  if there exists a positive constant $c$ such that $c^{-1} r^d \leq |B(x,r)| \leq cr^d$ for every $r\geq 1$ and every $x \in V$ (in which case we say $G$ is $d$-Ahlfors regular with constant $c$). Ahlfors regularity is clearly preserved by rough isometry, in the sense that if $G$ and $G'$ are $(\alpha,\beta)$-rough isometric graphs for some positive $\alpha,\beta$, and $G$ is $d$-Ahlfors regular with constant $c$, then there exists a constant $c'=c'(\alpha,\beta,c)$ such that $G'$ is $d$-Ahlfors regular with constant $c'$.

\medskip

Observe that if the graph $G$ is $d$-Ahlfors regular for some $d>2$ and satisfies a Gaussian heat kernel estimate, then summing the estimate \eqref{eq:GHKE} yields that
\[1 \leq  \sum_{n\geq0} p_n(v,v) \preceq 1 \]
for every vertex $v$, and that
\begin{equation}
\label{eq:***}
\bP_u(\text{hit v}) = \frac{\sum_{n\geq0}{p_n(u,v)}}{\sum_{n\geq0} p_n(v,v)} \asymp \langle uv\rangle^{-(d-2)}
\end{equation}
for all vertices $u$ and $v$ of $G$.

% \medskip

 \medskip

We now turn to the proofs of \cref{lem:constellations,prop:restrictedfirstmoment}. The key to both proofs is the following lemma.

\begin{lemma}
\label{lem:firstmomentlowerboundestimate}
Let $\bbG$ be a $d$-Ahlfors regular graph with constant $c_0$ for some $d>4$, let $\F$ be the uniform spanning forest of $\bbG$, and suppose that $\bbG$ satisfies $(c_0^{-1},c_0)$-Gaussian heat kernel estimates. 
Let $K_1,\ldots,K_N$ be a collection of finite, disjoint sets of vertices, and let $K = \bigcup_{i=1}^k K_N$. Let $\{X^v : v \in K\}$ be a collection of independent simple random walks started from the vertices of $K$. If
\begin{equation}\label{eq:lowerboundhypothesis}\P\Bigl(
% \text{no two vertices in $K$ are in the same component of $\F$}
\{X^u_i : i \geq 0\} \cap \{ X^v_i : i\geq0\} = \emptyset \text{ for all $u\neq v \in K$}
\Bigr) \geq \eps > 0,\end{equation}
 then there exist constants $c=c(\bbG,H,\eps,|K|,c_0)$ and $C=C(\bbG,H,\eps,|K|,c_0)$ such that
\begin{multline} 
\log_2\P\left(
\begin{array}{l}
\sF(K_i \cup K_j) \text{ if and only if $i=j$, and each two points in $K_i$ are connected}\\
\text{by a path in $\F$ of diameter at most $C \operatorname{diam}(K)$ for each $1 \leq i \leq N$}
% \big
\end{array}
\right)\\ \geq -(d-4)(|K| -N)\log_2 \operatorname{diam}(K) + c.\end{multline}
\end{lemma}

Here we are referring to the diameter of the path considered as a subset of $\bbG$.

Before proving \cref{lem:firstmomentlowerboundestimate}, let us see how it implies \cref{lem:constellations,prop:restrictedfirstmoment}. 

\begin{proof}[Proof of \cref{lem:constellations} given \cref{lem:firstmomentlowerboundestimate}]
Let $r'$ be a large constant to be chosen later. 
By definition of $R_\bbG$ and \cref{lem:annoyinglemma}, there exists $\eps=\eps(|A|)>0$ such that for each set $B \subseteq A$, there exists a set $\{\xi_{(B,b)} : b \in B\} \subset \bbV$ of diameter at most $R_\bbG(|B|)$ such that if $\{X^{(B,b)}:b\in B\}$ are independent simple random walks then
 % the points $\xi_{(B,b)}$ are all in different components of $\F$ with probability at least $(2\eps)^{2^{-|A|}}$.
\[\P\left(\{X^{(B,b)}_i:i\geq 0\} \cap \{X^{(B,b')}_i:i\geq 0\} = \emptyset \text{ for every $b\neq b' \in B$}\right) \geq (2\eps)^{2^{-|A|}}.\]
Take such a set for each $B$ in such a way that the set $\{\xi_{(B,b)} : (B,b) \in \cP_\bullet(A)\}$ is contained in the ball of radius $r'$ around $x$, and for each distinct $B,B' \subseteq A$, the sets $\{\xi_{(B,b)} : b\in B\}$ and  
 $\{\xi_{(B',b)} : b\in B'\}$ have distance at least $r'/2$ from each other. Clearly this is possible for sufficiently large $r'$. We have by independence that
 \[\P\left( \bigcap_{B \subseteq A} \left\{ \{X^{(B,b)}_i:i\geq 0\} \cap \{X^{(B,b')}_i:i\geq 0\} = \emptyset \text{ for every $b\neq b' \in B$}\right\}\right) \geq 2\eps.\]
On the other hand, it follows easily from the Greens function estimate \eqref{eq:HSC} that if $r'$ is sufficiently large (depending on $|A|$ and $\eps$) then 
\[
\P\biggl( \begin{array}{l} \{X^{(B,b)}_i:i\geq 0\} \cap \{X^{(B',b')}_i:i\geq 0\} = \emptyset \text{ for}\\ \text{some $B,B' \subseteq A$, $b \in B$ and $b'\in B$ with $B\neq B'$} \end{array}\biggr) \leq \eps,
\]
and we deduce that 
\[
\P\left(  \{X^{(B,b)}_i:i\geq 0\} \cap \{X^{(B',b')}_i:i\geq 0\} = \emptyset \text{ for every distinct $(B,b),(B',b')\in \cP_\bullet(A)$}\right) \geq \eps
\]
for such $r'$. Applying \cref{lem:firstmomentlowerboundestimate}, we deduce that
 $\P( \sA_{Cr'}(\xi) ) \geq c$
for some $C=C(\bbG,|A|,\eps,r')$ and $c=c(\bbG,|A|,\eps)$. It follows that
% so that
 $(\xi_{(B,b)})_{(B,b) \in \cP_\bullet(A)}$ is an $r$-good $A$ constellation for some $r=r(|A|)$ sufficiently large. \qedhere

\end{proof}

\begin{proof}[Proof of \cref{prop:restrictedfirstmoment} given \cref{lem:firstmomentlowerboundestimate}]
Let $\bbG$ be a $d$-dimensional transitive graph for some $d>4$. Let $x=(x_v)_{v\in \partial V}$ be such that $\langle x_u x_v \rangle \leq 2^{n-1}$ for every $u,v \in \partial V$,  let $\xi=(\xi_e)_{e\in E} \in \Omega_x(n)$, and let $r=r(H)$ and $(\xi_{(e,A,v)})_{e \in E, (A,v) \in \cP_\bullet(e)}$ be as in \cref{sec:2ndmoment}. 
% \medskip

% 

For each edge $e$ of $H$, write $\sA_e(\xi)$ for the event $\sA_r((\xi_{(e,A,v)})_{(A,v) \in \cP_\bullet(e)})$, which has probability at least $1/r$ by definition of the $r$-good constellation $(\xi_{(e,A,v)})_{(A,v) \in \cP_\bullet(e)}$. 
Since the number of subtrees of a ball of radius $r$ in $\bbG$ is bounded by a constant, it follows that there exists a constant $\eps=\eps(\bbG,H)$ and
% that $y_{(A,u)}$ is connected to $y_{(B,v)}$ in $\F$ if and only if $u=v$, and in this case $Y_e$ . 
 a collection of disjoint subtrees $(T_{(e,v)}(\xi))_{(e,v) \in E_\bullet}$ of $\bbG$ such that the tree $T_{(e,v)}(\xi)$ has diameter at most $r$ and contains each of the vertices $\xi_{(e,A,v)}$ with $(A,v)\in \cP_\bullet(e)$ for every $(e,v)\in E_\bullet$,  and the estimate
 \[\P\left(\sA_r(\hat\xi_e) \cap \bigcap_{v\perp e} \{T_{(e,v)}(\xi) \subset \F\}\right) \geq (2\eps)^{1/|E|} \]
 holds for every $e\in E$.
Fix one such collection $(T_{(e,v)}(\xi))_{(e,v) \in E_\bullet}$ for every $\xi \in \Omega_x(n)$, and for each $e\in E$ let $\sB_e(\xi)$ be the event that $T_{(e,v)}(\xi)$ is contained in $\F$ for every $v\in E$. Let $\sB(\xi) = \bigcap_{e\in E} \sB_e(\xi)$. Considering generating $\F$ using Wilson's algorithm, starting with random walks $\{X^{(e,A,v)} : e \in E,$ $(A,v) \in \cP_\bullet(e)\}$ such that $X^{(e,A,v)}_0=\xi_{(e,A,v)}$ for every $e \in E$ and $(A,v) \in \cP_\bullet(e)$, we observe that
\begin{equation}
\label{eq:tailtrivB2}
\Big|\P\left( \sB(\xi) \right) - \prod_{e\in E} \P\left(  \sB_e(\xi) \right)\Big| \leq \P\left( \begin{array}{l} X^{(e,A,v)} \text{ and } X^{(e',A',v')} \text{ intersect for some distinct} \\ \text{$e,e' \in E$ and some $(A,v) \in \cP_\bullet(e)$, $(A',v') \in \cP_\bullet(e')$} \end{array}\right)
\end{equation}
and hence that
\begin{equation}
\label{eq:tailtrivB}
\P\left( \sB(\xi) \right) \geq \frac{1}{2} \prod_{e\in E} \P\left(  \sB_e(\xi) \right) \geq \eps
\end{equation}
for all $n$ sufficiently large and $\xi \in \Omega_x(n)$.
% , and let $T_{(e,v)}(\xi)=\gamma_{\xi_e}(T_{(e,v)}(\xi))$. 

\medskip

% \medskip

Let $\bbG_\xi$ be the graph obtained by contracting the tree $T_{(e,v)}(\xi)$ down to a single vertex for each $(e,v) \in E_\bullet$. The spatial Markov property of the USF (see e.g.\ \cite[Section 2.2.1]{HutNach2016b}) implies that the law of $\F$ given the event $\sB(\xi)$ is equal to the law of the union of $\bigcup_{(e,v) \in E_\bullet} T_{(e,v)}(\xi)$ with the uniform spanning forest of $\bbG_\xi$. Observe that $\bbG_\xi$ and $\bbG$ are rough isometric, with constants depending only on $\bbG$ and $H$, and that $\bbG_\xi$ has degrees bounded by a constant depending only on $\bbG$ and $H$. Thus, it follows from \cref{thm:HSCGreen,thm:GHKEstability,thm:GHKEimpliesEHI} that $\bbG_\xi$ is $d$-Ahlfors regular, satisfies Gaussian heat kernel estimates, and satisfies an elliptic Harnack inequality, each with constants depending only on $H$ and $\bbG$.

% Let 
\medskip

Let $E_\star = E \cup \{\star\}$, and let $K = E_\bullet \cup \{(\star,v) : v \in \partial V\}$.
For each $(e,v)\in E_\bullet$, let $x_{(e,v)}$ be the vertex of $\bbG_\xi$ that was formed by contracting $T_{(e,v)}(\xi)$, and let $x_{(\star,v)} = x_v$ for each $v\in \partial V$.
For each vertex $v$ of $H$, choose an edge $e_0(v)\perp v$ arbitrarily from $E_\star$, and let $K' = K \setminus \{x_i: v \in V\}$. Let $\F_\xi$ be the uniform spanning forest of $\bbG_\xi$, and let $\tilde \sW'(x,\xi)$ be the event that for each $(e,v),(e',v') \in K$ the vertices $x_{(e,v)}$ and $x_{(e',v')}$ are in the same component of $\F_\xi$ if and only if $v=v'$. The spatial Markov property and \eqref{eq:tailtrivB} imply that
\[\P\left(\tilde \sW(x,\xi)\right) \geq \eps \P\left(\bar \Witness'(x,\xi)\right) \succeq \P\left(\tilde \Witness'(x,\xi)\right)\]
whenever $n$ is sufficiently large and $\xi \in \Omega_x(n)$. Thus, applying \cref{lem:firstmomentlowerboundestimate} to $\bbG_\xi$ by setting $N=|V|$, enumerating $V=\{v_1,\ldots,v_N\}$ and setting $K_i = \{ x_{(e,v_i)} : (e,v_i) \in K \}$ for each $v\in V$ yields that
% and so it suffices to prove that
\[ \log_2 \P\left(\tilde \Witness(x,\xi)\right) \gtrsim \log_2 \P\left(\tilde \Witness'(x,\xi)\right) \gtrsim  -(d-4)\left(\Delta -|V_\circ|\right) \, n, \]
completing the proof. \qedhere

\end{proof}

\medskip

We now start working towards the proof of  \cref{lem:firstmomentlowerboundestimate}. 
We begin with the following simple estimate.

\begin{lemma}\label{lem:hitwnotB}
Let $G$ be $d$-Ahlfors regular with constant $c_1$, and suppose that $G$ satisfies $(c_2,c_2')$-Gaussian heat kernel estimates. Then 
there exist a positive constant $C=C(d,c_1,c_2,c_2')$ such that
\vspace{0.2em}
\[
\vspace{0.2em}
C^{-1} \langle u w \rangle^{-(d-2)} \leq \bP_u\left(\text{hit } w \text{ before }\Lambda_x(n+3c,\infty) \text{, do not hit  } \Lambda_x(0,n)\right) \leq C \langle u w \rangle^{-(d-2)}\]
for every $c \geq C$, every vertex $x$, every $n\geq 1$, and every $u,w \in \Lambda_x(n+c,n+2c)$.
\end{lemma}

\begin{proof}
The upper bound follows immediately from \eqref{eq:***}. We now prove the lower bound. 
	For every $c \geq 1$ and every $u,w \in \Lambda_x(n+c,\infty)$, we have that
	\[
		\bP_u(\text{hit }\Lambda_x(0,n)) = \frac{\bP_u(\text{hit } x)}{\bP_u(\text{hit } x \mid \text{ hit }\Lambda_x(0,n))} \asymp \frac{\langle u x \rangle^{-(d-2)}}{2^{-(d-2)n}} \preceq 2^{-(d-2)c}. 
	\]
	Thus, we have that
	\begin{align*}
		\bP_u(\text{hit } w \text{ and } \Lambda_x(0,n)) &\leq \bP_u(\text{hit $\Lambda_x(0,n)$ after hitting $w$}) +  
		\bP_u(\text{hit $w$ after hitting $\Lambda_x(0,n)$})\\
		&\preceq \langle u w \rangle^{-(d-2)}2^{(d-2)n}\langle wx\rangle^{-(d-2)} + 2^{(d-2)n}\langle u x \rangle^{-(d-2)} \langle w x \rangle^{-(d-2)},
	\end{align*}
	where the second term is bounded by conditioning on the location at which the walk hits $\Lambda_x(0,n)$ and then using the strong Markov property. 
	By the triangle inequality, we must have that at least one of $\langle u x \rangle$ or $\langle w x \rangle$ is greater than $\frac{1}{2}\langle u w \rangle $. This yields the bound
	\begin{align*}
		\bP_u(\text{hit } w \text{ and } \Lambda_x(0,n))
		&\preceq \left(2^{(d-2)n}\langle wx\rangle^{-(d-2)} + 2^{(d-2)n}\left(\min \left\{\langle u x \rangle,\, \langle w x \rangle\right\}\right)^{-(d-2)}\right)  \langle uw \rangle^{-(d-2)}\\
		&\preceq 2^{-(d-2)c}\langle u w \rangle^{-(d-2)}.
	 \end{align*}
	 On the other hand, if $u,w \in \Lambda_x(n+c,n+2c)$ then conditioning on the location at which the walk hits $\Lambda_x(n+3c,\infty)$ yields that
	 \[ \bP_u(\text{hit } w \text{ after } \Lambda_x(n+3c,\infty)) \preceq 2^{-(d-2)(n+3c)} \preceq \langle uw \rangle^{-(d-2)}.\]
	 The claim now follows easily.
 \end{proof}

\begin{proof}[Proof of \cref{prop:restrictedfirstmoment}]

For each $1 \leq i \leq N$, let $x_i$ be chosen arbitrarily from the set $K_i$.
Let $(X^{x})_{x \in K}$ be a collection of independent random walks on $\bbG$, where $X^{x}$ is started at $x$ for each $x\in K$, and write $X^i=X^{x_i}$. 
Let $K'_i=K_i \setminus \{x_i\}$ for each $1 \leq i \leq N$ and let $K'=\bigcup_{i=1}^N K'_i$. In this proof, implicit constants will be functions of $|K|, N, c_0,$ and $d$. We take $n$ such that $2^{n-1} \leq \diam(K) \leq 2^{n}$.

\medskip

Let $c_1,c_2,c_3$ be constants to be determined. 
For each $y=(y_{x})_{x\in K} \in (\Lambda(n+c_1,n+c_3))^{K}$, let $\sY_y$ be the event
\[\sY_y = \{ X^{x}_{2^{2(n+c_2)}} = y_{x} \text{ for each $x\in K$}\}.\]
Let $\sC(c_2)$ be the event that none of the walks $X^{x}$ intersect each other before time $2^{2(n+c_2)}$, so that $\P(\sC(c_2)) \geq \eps$ for every $c_2 \geq 0$ by assumption.
% Our assumptions imply that there exists $\delta=\delta(\bbG,N,\eps)$ such that if $n$ is sufficiently large then, by Wilson's algorithm,
% \[\P(\sC(c_2)) \geq \P(\text{ the vertices $x_{(e,v)}$ are all in different components of $\F$}) \geq \delta \]
% for every $c_2 \in \Z$.
% 
% \medskip
% 
For each $x\in K$, let $\sD_{x}(c_1,c_3)$ be the event that $X^{x}_{2^{2(n+c_2)}}$ is in $\Lambda(n+c_1,n+c_3)$ and that $X^{x}_m \in \Lambda(n,\infty)$ for all $m \geq 2^{2(n+c_2)}$, and let $\sD(c_1,c_3) = \bigcap \sD_{x}(c_1,c_3)$.
It follows by an easy application of the Gaussian heat kernel estimates that we can choose $c_2=c_2(\bbG,N,\eps)$ and $c_3=c_3(\bbG,N,\eps)$ sufficiently large that 
\begin{equation}
\label{eq:DgivenY}
\P(\sD(c_1,c_3) \mid \sY_y) \geq 1- \eps/2
\end{equation}
for every $y=(y_{x})_{x\in K} \in (\Lambda(n+c_1,n+c_3))^{K}$, and in particular 
 so that $\P(\sC(c_2) \cap \sD(c_1,c_3)) \geq \eps$.  We fix some such sufficiently large $c_1,c_2,$ and $c_3$, and also assume that $c_1$ is larger than the constant from \cref{lem:hitwnotB}. We write $\sC=\sC(c_2)$, $\sD_{x}=\sD_{x}(c_1,c_3)$, and $\sD=\sD(c_1,c_3)$.
 % Let $c_2,c_3$ be constants to be chosen later. 

\medskip

For each $1 \leq i \leq N$ and $x\in K'_i$, we define $\sI_{x}$ to be the event that the walk $X^{x}$ hits the set
\begin{multline*}
L^i_{\text{good}}=\\
\left\{ \LE(X^i)_m : \LE(X^i)_m \in \Lambda(n+2c_3, n+ 4c_3),\, \LE(X^i)_{m'} \in \Lambda(0, n+ 6c_3) \text{ for all $ 0 \leq m' \leq m$} \right\}
\end{multline*}
before hitting $\Lambda(n + 6c_3, \infty)$, and let $\sI = \bigcap_{x\in K'} \sI_{x}$.

\medskip

For each $x$ and $x'$ in $K$, we define $\sE_{x,x'}$ to be the event that the walks $X^{x}$ and $X^{x'}$ intersect, and let
 \[\sE = \bigcup\left\{ \sE_{x,x'} : 1 \leq i < j \leq N,\, x \in K_i,\, x'\in K_j \right\} \cup \bigcup \left\{ \sE_{x,x'} : x,x' \in K' \right\}.\]
% 
% \medskip
These events have been defined so that, if we sample $\F$ using Wilson's algorithm, beginning with the walks $\{ X^v : v \in V\}$ (in any order) and then the walks $\{ X^{x} : x\in K\}$ (in any order), we have that
\begin{multline*}
\left\{
\begin{array}{l}
 \sF(K_i \cup K_j) \text{ if and only if $i=j$, and each two points in $K_i$ are connected}\\
\text{by a path in $\F$ of diameter at most $2^{6c_3} \operatorname{diam}(K)$ for each $1 \leq i \leq N$}
% \big
\end{array}\right\}\\ \supseteq (\sC \cap \sD \cap \sI) \setminus \sE.
\end{multline*}
 Thus, it suffices to prove that
 \[\log_2 \P\left(\left(\sC \cap \sD \cap \sI\right) \setminus \sE\right) \gtrsim -(d-4)\left(K -N\right) \, n = -(d-4)|K'|\,n .\]
 We break this estimate up into the following two lemmas: one lower bounding the probability of the good event $\sC \cap \sD \cap \sI$, and the other upper bounding the probability of the bad event $\sC \cap \sD \cap \sI \cap \sE$.

\begin{lemma}
\label{lem:Iev}
The estimate
\[\log_2 \P(\sI_{x} \mid \sC \cap \sD \cap \sY_y) \gtrsim -(d-4)n\] holds for every $x \in K'$ and $y=(y_{x})_{x\in K} \in (\Lambda(n+c_1,n+c_3))^{K}$.
\end{lemma}

The proof uses techniques from \cite{lyons2003markov} and the proof of \cite[Theorem 4.2]{BeKePeSc04}.

\begin{proof}[Proof of \cref{lem:Iev}]

Fix $x \in K'$, and let $1\leq i \leq N$ be such that $x\in K'_i$. Write $Y=X^i$ and $Z=X^{x}$. 
% For each $v\in V$, 
Let $L = ( L(k) )_{k\geq 0}$ be the loop-erasure of $(Y_k)_{k\geq 0}$ and, for each $m\geq 0$, let $L_m= (L_m(k))_{k= 0}^{q_m}$ be the loop-erasure of $( Y_k )_{k=0}^m$. Define 
\[\tau(m) = \inf\{ 0 \leq r \leq q_m : L_m(r) = Y_k \text{ for some $k \geq m$}\} \]
and 
% \vspace{0.1em}
\[
\vspace{0.4em}
\tau(m,\ell) = \inf\{ 0 \leq r \leq q_m : L_m(r) = Z_k \text{ for some $k \geq \ell$}\}.\]
The definition of $\tau(m)$ ensures that $L_m(k)=L(k)$ for all $k\leq \tau(m)$.  We define the indicator random variables
\begin{multline*}
I_{m,\ell} =\\ \mathbbm{1}\left(Y_m = Z_\ell \in \Lambda(n+2c_3,n+4c_3), \text{ and } Y_{m'}, Z_{\ell'} \in \Lambda(0,n+6c_3) \text{ for all $m' \leq m$, $\ell'\leq \ell$}\right)
\end{multline*}
and
\begin{align*}
% \hat I_{m,\ell} &= \mathbbm{1}(X_m = Y_\ell)\mathbbm{1}(\sA),\\
J_{m,\ell} &= I_{m,\ell} \, \mathbbm{1} \!\big(\tau(m,\ell) \leq \tau(m)\big).
\end{align*}
Observe that 
\[\sI_x \subseteq \left\{ J_{m,\ell} =1 \text{ for some } m, \ell \geq 2^{2(n+c_2)} \right\}. \]
Moreover, for every $m,\ell \geq 2^{2(n+c_2)}$ and every $y \in (\Lambda(n+c_1,n+c_3))^{K}$,  the walks $\langle Y_k \rangle_{k\geq m}$ and $\langle Z_k \rangle_{k\geq \ell}$ have the same distribution conditional on the event \[\sC \cap \sD \cap \sY_y \cap \{I_{m,\ell}=1\}.\] 
Thus, we deduce that
\[\P\left(\tau(m) \geq \tau(m,\ell) \mid \sC \cap \sD \cap \sY_y \cap \{I_{m,\ell}=1\}\right) \geq 1/2 \]
whenever the event being conditioned on has positive probability,
and therefore that 
\begin{equation*}\E[ I_{m,\ell} \mid \sC \cap \sD \cap \sY_y] \; \geq \; \E[J_{m,\ell}\mid \sC \cap \sD \cap \sY_y] \; \geq \; \frac{1}{2}\E[ I_{m,\ell} \mid \sC \cap \sD\cap \sY_y].\end{equation*}

Let 
\[I = \sum_{\ell \geq 2^{2(n+c_2)}}\sum_{m \geq 2^{2(n+c_2)}}  I_{m,\ell} 
\quad \text{ and } \quad
% and let
% \[
J = \sum_{\ell \geq 2^{2(n+c_2)}}\sum_{m \geq 2^{2(n+c_2)}}  J_{m,\ell}, \]
and note that the conditional distribution of  $I$ given the event $\sC\cap\sD\cap\sY_y$ is the same as the conditional distribution of $I $ given the event $\sD \cap \sY_y $.
For every $y \in (\Lambda(n+c_1,n+c_3))^{K}$, we have that, decomposing $\E[I \mid \sD \cap \sY_y]$ according to the location of the intersections and applying the estimate \cref{lem:hitwnotB},
\begin{multline*}
\E[J \mid \sC \cap \sD \cap \sY_y] \asymp \E[I \mid \sD \cap \sY_y] \succeq\\
% \sum_{\ell \geq 2^{2n}}\sum_{m \geq 2^{2n}}  \E [ I_{m,\ell} \mid \sC \cap \sD \cap \sY_y] \\
 \sum_{w \in \tilde \Lambda(n+2c_3,n+4c_3)} \bP_{y_{x_i}}(\,\text{hit } w \text{ before $\Lambda(n+6c_3,\infty)$} \mid \text{do not hit } \Lambda(0,n))\\\hspace{5cm} \cdot \bP_{y_{x}}(\,\text{hit } w \text{ before $\Lambda(n+6c_3,\infty)$} \mid \text{do not hit } \Lambda(0,n))\\
% \succeq \sum_{w \in \Lambda_{n+r,n+r+1}} 2^{-2(d-2)n}\\
\succeq 2^{-2(d-2)n} | \Lambda(n+2c_3,n+4c_3)| \asymp 2^{-(d-4)n}.
% \P(X, Y \text{ both hit $w$ but do not hit $B_{2^{n-k}}$ after time $2^{2n}$}\} \mid  \sB, X_{2^{2n}} = u, Y_{2^{2n}}=v)
 \end{multline*}
% Using the fact that \[\P(\sB\cap \sA \cap \{Y_{2^{2n}}, Z_{2^{2n}} \in \Lambda_{n+r,n+r+1} \}) \succeq 1\], we have that $\E[J] \succeq 2^{-(d-4)n}$. 
On the other hand, we have that 
\begin{align*}
\E[J^2 \mid \sC \cap \sD \cap \sY_y] 
% \,\leq\, \E[J^2 \mid  \sY_y] 
\, \leq \, \E[I^2 \mid \sC \cap \sD \cap \sY_y] \, = \,
\E[I^2 \mid  \sD \cap \sY_y] \preceq \E[I^2 \mid \sY_y].
\end{align*}
% and we can calculate that
Meanwhile, decomposing $\E[I^2 \mid \sY_y]$ according to the location of the intersections and applying the Gaussian heat kernel estimates yields that 
\begin{multline*}
 \E[I^2 \mid  \sY_y] 
 % &\leq \sum_{w,z\in \bbV} G(u,w)G(w,z)\left[G(v,w)G(w,z) + G(v,z)G(z,w) \right]\\
 \preceq \sum_{w,z \in \Lambda(n+2c_3,n+4c_3)} \langle y_{x_i} w \rangle^{-(d-2)}\langle w z \rangle^{-(d-2)}\langle y_{x} w \rangle^{-(d-2)} \langle wz \rangle^{-(d-2)}\\
+
\sum_{w,z \in \Lambda(n+2c_3,n+4c_3)} \langle y_{x_i} w \rangle^{-(d-2)}\langle w z \rangle^{-(d-2)}\langle y_{x} z \rangle^{-(d-2)} \langle z w \rangle^{-(d-2)},
% &...
\end{multline*}
where the two different terms come from whether $Y$ and $Z$ hit the points of intersection in the same order or not. With the possible exception of $\langle wz \rangle$, all the distances involved in this expression  are comparable to $2^n$. Thus, we obtain that
\[\E[I^2 \mid \sY_y] \preceq 2^{-2(d-4)n} \sum_{w,z \in \Lambda(n+2c_3,n+4c_3)} \langle w z \rangle^{-2(d-2)}.\]
For each $w \in \bbV$, considering the contributions of dyadic shells centred at $w$ yields that, since $d>4$,
\begin{align*}
\sum_{z\in \bbV} \langle w z\rangle^{-2(d-2)} \preceq \sum_{n\geq 0}2^{dn}2^{-2(d-2)n} \leq \sum_{n\geq 0} 2^{-(d-4)n} \preceq 1,
\end{align*}
and we deduce that
\[\E[I^2 \mid \sY_y] \preceq 2^{-2(d-4)n} |\Lambda(n+2c_3,n+4c_3)| \preceq 2^{-(d-4)n}.\]
Thus, the Cauchy-Schwarz inequality implies that
\begin{align*}\P(\sI_{x}\mid \sC \cap \sD \cap \sY_y) \geq \P(J > 0 \mid \sC \cap \sD \cap \sY_y) \succeq \myfrac[0.2em]{\E\left[J \mid \sC \cap \sD \cap \sY_y\right]^2}{\E\left[J^2 \mid \sC \cap \sD \cap \sY_y\right]} \succeq 2^{-(d-4)n}.
\end{align*}
as claimed.
\end{proof}

\medskip

We next use the elliptic Harnack inequality to pass from an estimate on $\sI_x$ to an estimate on $\sI$.

\begin{lemma}\label{lem:farintersections}
$\log_2 \P(\sC\cap \sD \cap \sI) \gtrsim -(d-4)|K'|\, n$
\end{lemma}

\begin{proof}

For each $1\leq i \leq N$, let $x'_i$ be chosen arbitrarily from $K'_i$.
To deduce \cref{lem:farintersections} from \cref{lem:Iev}, it suffices to prove that
\[\P\Bigg(\bigcap_{x \in K'} \sI_{x} \mid \sC \cap \sD \cap \sY_y \Bigg) \succeq \prod_{i=1}^N \P\left( \sI_{x_i'} \mid \sC \cap \sD \cap \sY_y\right)^{|K_i'|}\]
for every $y=(y_{x})_{x\in K} \in (\Lambda(n+c_1,n+c_3))^{K}$.

\medskip

% where $e_1(v) \ni v$ is chosen arbitrarily from the set $\{e \in E : e \perp v \text{ and } e \neq e_0(v)\}$.
Let $\cX$ be the $\sigma$-algebra generated by the random walks $(X^{i})_{i=1}^N$. 
% Given $y \in (\Lambda(n+c_1,n+c_3))^{K}$ and $x \in K'$, 
% Observe that, for each $x \in K'$, if we fix the values $y_{x'}$ for all $x' \neq x$ and consider 
% \[\P(\sI_{(e,)}\]
Observe that for each $x\in K'$ we have
\begin{align*}\P(\sI_{x} \mid \cX,\, \sC\cap\sD\cap\sY_y) &= \frac{\bP_{y_{x}}
\left( \text{hit $L^i_{\text{good}}$ before $\Lambda(0,n+6c_3)$, never leave $\Lambda(n,\infty)$} \right)}{\bP_{y_{x}}
\left( \text{never leave $\Lambda(n,\infty)$} \right)}\\
& \asymp \bP_{y_{x}}
\left( \text{hit $L^i_{\text{good}}$ before $\Lambda(0,n+6c_3)$, never leave $\Lambda(n,\infty)$} \right).
\end{align*}
The right hand side of the second line is a positive harmonic function of $y_{x}$ on $\Lambda(n+c_1,n+c_3+1)$, and so the elliptic Harnack inequality implies that for every $y,y' \in (\Lambda(n+c_1,n+c_3))^{K}$ and every $x\in K'$, we have that
% $e \perp v$ distinct from $e_0(v)$, we have that
\begin{equation*}\P\left(\sI_{x} \mid \cX,\, \sC \cap \sD \cap \sY_y\right) \asymp \P(\sI_{x} \mid \cX,\, \sC \cap \sD \cap \sY_{y'}).  \end{equation*}
Furthermore, if $y'$ is obtained from $y$ by swapping $y_{x}$ and $y_{x'}$ for some $1\leq i \leq N$ and $x,x' \in K'_i$, then clearly
\begin{equation*}\P(\sI_{x} \mid \cX,\, \sC\cap\sD\cap\sY_y) = \P(\sI_{x'} \mid \cX,\, \sC\cap\sD\cap\sY_{y'}).  \end{equation*}
% for every $j,j'$ and $u$. 
Therefore, it follows that
\begin{equation*}\P(\sI_{x} \mid \cX,\, \sC\cap\sD\cap\sY_y) \asymp \P(\sI_{x'} \mid \cX,\, \sC\cap\sD\cap\sY_{y})  \end{equation*}
for all $1\leq i \leq N$ and $x,x' \in K'_i$.
 % sharing a common vertex.

Since the events $\sI_{x}$ are conditionally independent given the $\sigma$-algebra $\cX$ and the event $\sC \cap \sD \cap \sY_y$, we deduce that
\begin{align*}
\P(\sI \mid  \sC \cap \sD \cap \sY_y) &= \E\left[ \P(\sI \mid \cX,\, \sC \cap \sD \cap \sY_y) \mid \sC \cap \sD \cap \sY_y \right]\\
& = \E\left[ \prod_{x\in K'}\P(\sI_{x} \mid \cX,\, \sC \cap \sD \cap \sY_y) \mid \sC \cap \sD \cap \sY_y\right]\\
& \asymp \E\left[ \prod_{i=1}^N\P(\sI_{x'_i} \mid \cX,\, \sC \cap \sD \cap \sY_y)^{|K'_i|} \mid \sC \cap \sD \cap \sY_y\right].
\end{align*}

Now, the random variables $\P(\sI_{x'_i} \mid \cX,\, \sC \cap \sD \cap \sY_y)^{|K_i'|}$ are independent conditional on the event $\sC \cap \sD \cap \sY_y$, and so we have that
\begin{align*}
\P(\sI \mid  \sC \cap \sD \cap \sY_y) 
% \E\left[ \P(\sI \mid \cX,\, \sC \cap \sD \cap \sY_y) \mid \sC \cap \sD \cap \sY_y \right]\\
% & = \E\left[ \prod_{(i,j)}\P(\sI_{i,j} \mid \cX,\, \sC \cap \sD \cap \sY_y) \mid \sC \cap \sD \cap \sY_y\right]\\
& \asymp \prod_{i=1}^N \E\left[ \P(\sI_{x'_i} \mid \cX,\, \sC \cap \sD \cap \sY_y)^{|K'_i|} \mid \sC \cap \sD \cap \sY_y\right]\\
&\geq \prod_{i=1}^N  \P(\sI_{x'_i} \mid  \sC \cap \sD \cap \sY_y)^{|K'_i|},
\end{align*}
as claimed, where the second line follows from Jensen's inequality. 
\end{proof}

\medskip

Finally, it remains to show that the probability of getting  unwanted intersections in addition to those that we do want is of lower order than the probability of just getting the intersections that we want.

\begin{lemma} \label{lem:ABIC}
We have that
\[\log_2 \P(\sC \cap \sD \cap \sI \cap \sE) \lesssim   - \bigl[(d-4)|K'|+2\bigr] \, n + {|K'|^2}\log_2 n.\] 
\end{lemma}

\begin{proof}
% The second sentence of the lemma follows from the first together with \cref{lem:farintersections}. We now prove the first sentence. 
For each $w \in \bbV$ and $x,x' \in K$, let $\sE_{x,x'}(w)$ be the event that $X^{x}$ and $X^{x'}$ both hit $w$. 
Let $\zeta=(\zeta_{x})_{x\in K'}$ and let $\sigma = (\sigma_i)_{i=1}^N$ be such that $\sigma_v$ is a bijection from  $\{1,\ldots,|K'_i|\}$ to $K'_i$  for each $1 \leq i \leq N$. 
We define $\sR_\sigma(\zeta)$ to be the event that  for each $1 \leq i \leq N$ the walk $X^{i}$ passes through the points $\{ \zeta_{x} : x \in K'_i\}$ in the order given by $\sigma$ and that for each $x\in K'$ the walk $X^{x}$ hits the point $\zeta_{x}$. We also define
\begin{equation*}
	R_\sigma(\zeta) =
		  \prod_{i=1}^N\left\langle x_i \zeta_{\sigma_i(1)} \right\rangle^{-(d-2)}\prod_{j=1}^{|K'_i|} \left\langle \zeta_{\sigma_i(j-1)} \zeta_{\sigma_i(j)} \right\rangle ^{-(d-2)} \left\langle \sigma_i(j) \zeta_{\sigma_i(j)} \right\rangle^{-(d-2)}, 
\end{equation*}
so that 
$\P(\sR_\sigma(\zeta)) \asymp R_\sigma(\zeta)$ for every $\zeta \in \bbV^{K'}$.

Let $\Lambda_\zeta = \Lambda(n+c_1,n+c_1+c_2)^{K'}$, $\Lambda_{w,1}= \Lambda(n,n+c_2+1),$ $\Lambda_{w,2} = \Lambda(n+c_2+1,\infty)$, and $\Lambda_w=\Lambda_{w,1}\cup\Lambda_{w,2}$. (Note that these sets are not functions of $\zeta$ or $w$, but rather are the sets from which $\zeta$ and $w$ will be drawn.)
We also define 
\[O=K^2 \setminus \left[ \{(x,x) : x \in K\} \cup \bigcup_{i=1}^N\left[\{(x_i,x) : x \in K_i \} \cup \{(x,x_i) : x \in K_i \} \right] \right]. \]
To be the set of pairs of points at least one of which must have their associated pair of random walks intersect in order for the event $\sE$ to occur. 
 Define the random variables $M_{\sigma,0}$, $M_{\sigma,1}$, and $M_{\sigma,2}$ to be
\begin{align*}
\vspace{0em}
M_{\sigma,0} &=   \sum_{\zeta \in \Lambda_\zeta} \mathbbm{1}\big[\sR_\sigma(\zeta)\big]\\
\vspace{0em}\\
M_{\sigma,1} &=  \sum_{(x,x') \in O}\; \sum_{w \in \Lambda_{w,1}} \sum_{\zeta \in \Lambda_\zeta} \mathbbm{1}\big[\sR_\sigma(\zeta) \cap \sE_{x,x'}(w)\big], \qquad \text{ and}\\
\vspace{0em}\\
M_{\sigma,2} &=  \sum_{(x,x')\in O} \;\sum_{w \in \Lambda_{w,2}} \sum_{\zeta \in \Lambda_\zeta} \mathbbm{1}\big[\sR_\sigma(\zeta) \cap \sE_{x,x'}(w)\big].
\end{align*}
Observe that  $\sum_{\sigma}(M_{\sigma,1} + M_{\sigma,2}) \geq 1$ on the event 
 $\sC \cap \sB\cap\sI\cap\sE$, and so
 to prove \cref{lem:ABIC} it suffices to prove that
\begin{equation}
	\label{eq:M1M2estimate}
	\log_2\E\left[M_{\sigma,1}+M_{\sigma,2}\right]
	\lesssim
	-\bigl[(d-4)|K'|+2\bigr]\, n + 2\log_2 n
\end{equation}
for every $\sigma$. We will require the following estimate.

\begin{lem}
\label{lem:REcases}
The estimate
\begin{equation}\P\left(\sR_\sigma(\zeta) \cap \sE_{x,x'}(w) \right) \preceq R_\sigma(\zeta) \langle w \zeta_{x} \rangle^{-(d-2)}\langle w \zeta_{x'} \rangle^{-(d-2)}. \end{equation}
holds for every $(x,x') \in O$, every $\zeta \in \Lambda_\zeta$, every $w \in \Lambda_w$, and every collection $\sigma=(\sigma_i)_{i=1}^N$ where $\sigma_i : \{1,\ldots,|K'_i|\} \to K'_i$ is a bijection for each $1\leq i \leq N$.
\end{lem}

\begin{proof}
Unfortunately, this proof requires a straightforward but tedious case analysis. We will give details for the simplest case, in which both $x,x'\in K'$. 
A similar proof applies in the cases that one or both of $x$ or $x'$ is not in $K'$, but there are a larger amount of subcases to consider according to when the intersection takes place. In the case that $x,x' \in K'$, let $\sE^{-,-}(\zeta,w)$, $\sE^{-,+}(\zeta,w)$, $\sE^{+,-}(\zeta,w)$ and $\sE^{+,+}(\zeta,w)$ be the events defined as follows:
% that 
\begin{itemize}[leftmargin=2.5cm]
\itemsep1em
\item[$\sE^{-,-}(\zeta,w)$:] The event $\sR_\sigma(\zeta)$ occurs, and $X^{x}$ and $X^{x'}$ both hit $w$ before they hit $\zeta_{x}$ and $\zeta_{x'}$ respectively. 
\item[$\sE^{-,+}(\zeta,w)$:] The event $\sR_\sigma(\zeta)$ occurs, $X^{x}$ hits $w$ before hitting $\zeta_{x}$, and $X^{x'}$ hits $w$ after  hitting  $\zeta_{x'}$.
\item[$\sE^{+,-}(\zeta,w)$:] The event $\sR_\sigma(\zeta)$ occurs, $X^{x}$ hits $w$ after hitting $\zeta_{x}$, and $X^{x'}$ hits $w$ before  hitting  $\zeta_{x'}$.
\item[$\sE^{+,+}(\zeta,w)$:] The event $\sR_\sigma(\zeta)$ occurs, and $X^{x}$ and $X^{x'}$ both hit $w$ after they hit $\zeta_{x}$ and $\zeta_{x'}$ respectively.
\end{itemize}
We have the estimates
\begin{align*}
\P(\sE^{-,-}(\zeta,w)) &\asymp  R(x,\zeta) \frac{\langle x  w \rangle^{-(d-2)}\langle w  \zeta_{x} \rangle^{-(d-2)} \langle x'  w \rangle^{-(d-2)}\langle w  \zeta_{x'} \rangle^{-(d-2)}}
 {\langle x \zeta_{x} \rangle^{-(d-2)}\langle x' \zeta_{x'} \rangle^{-(d-2)}},
 \\\\
 \P(\sE^{-,+}(\zeta,w)) &\asymp  R(x,\zeta) \frac{\langle x  w \rangle^{-(d-2)}\langle w  \zeta_{x} \rangle^{-(d-2)}}
{\langle x \zeta_{x} \rangle^{-(d-2)}}
\langle \zeta_{x'} w \rangle^{-(d-2)}, \\\\
\P(\sE^{+,-}(\zeta,w)) &\asymp R(x,\zeta)\frac{\langle x'  w \rangle^{-(d-2)}\langle w  \zeta_{x'} \rangle^{-(d-2)}}{\langle x' \zeta_{x'} \rangle^{-(d-2)}}\langle \zeta_{x} w \rangle^{-(d-2)},\\
\text{and}\hspace{3.5cm}&\\
\P(\sE^{+,+}(\zeta,w)) &\asymp  R(x,\zeta)\langle \zeta_{x} w \rangle ^{-(d-2)}\langle \zeta_{x'} w \rangle^{-(d-2)}.\end{align*}
% \vspace{0.5em}

\noindent
 In all cases, a bound of the desired form follows since $\langle w x \rangle \succeq \langle \zeta_{x} x \rangle$ and $\langle w x' \rangle \succeq \langle \zeta_{x'} x' \rangle$ for every $x,x'\in K'$, $\zeta\in \Lambda_\zeta$, and $w\in \Lambda_w$, and we conclude by summing these four bounds. \qedhere
 % and $(i',j')\in $,
%   we deduce that
%   \vspace{0.5em}
% \[
% \vspace{0.5em}
% \P(\sE^{\pm,\pm}(\zeta,w)) \preceq R(x,\zeta) \langle w  \zeta_{x} \rangle^{-(d-2)}\langle w \zeta_{x'} \rangle^{-(d-2)}\]
% as claimed.   \qedhere

\end{proof}

Our aim now is to prove \cref{eq:M1M2estimate} by an appeal to \cref{lem:firstmomentgeneral}. To do this, we will encode the combinatorics of the potential ways that the walks can intersect via hypergraphs. 
To this end, let $H_\sigma$ be the finite hypergraph with boundary that
has vertex set
\[V(H_\sigma) = \left(\{1\} \times K\right) \cup \left(\{2\} \times K'\right),\]
boundary set
\[\partial V(H_\sigma) = 
\left(\{1\} \times \{x_i : 1 \leq i \leq N \}\right) \cup \left(\{2\}\times K'\right),\]
% \{(1,v,0) : v\in V(H) \}\cup \{(2,i,j) : 1\leq i \leq |I|,\, 0 \leq j \leq |W_i| -1 \},\]
% \[V_\circ(L) = \{(1,i,j) : 1\leq i \leq |I|,\, 1 \leq j \leq |W_i| -1  \}\]
and edge set
\begin{multline*}
E(H_\sigma) =
 \left\{ \left\{(2,\sigma_i(j)), (1,\sigma_i(j)), (1,i,\sigma_i(j+1))\right\} : 1 \leq i \leq N,\, 1 \leq j \leq |K'_i|-1 \right\}\\ \cup \left\{\left\{(2,\sigma_i(|K'_i|)), (1,\sigma_i(|K'_i|)\right\} : 1 \leq i \leq N\right\}.
\end{multline*}
See \cref{fig:inthyp} for an illustration. 
Note that the isomorphism class of $H_\sigma$ does not depend on $\sigma$.
The edge set $E(H_\sigma)$ can be identified with $K'$ by taking the intersection of each edge with the set $\{2\}\times K'$. 
% which can be identified with the set $\{1\leq i \leq |I|,\, 1 \leq j \leq |W_i|\}$.
Under this identification, the definition of $H_\sigma$ ensures that 
\[R_\sigma(\zeta) = W^{H_\sigma,2}(x,\zeta)\]
and consequently that
\[\E[M_{\sigma,0}] \preceq \bbW^{H_\sigma,2}_x(n,n+c_1+c_2).\]

\begin{figure}
\includegraphics[width=0.28\textwidth]{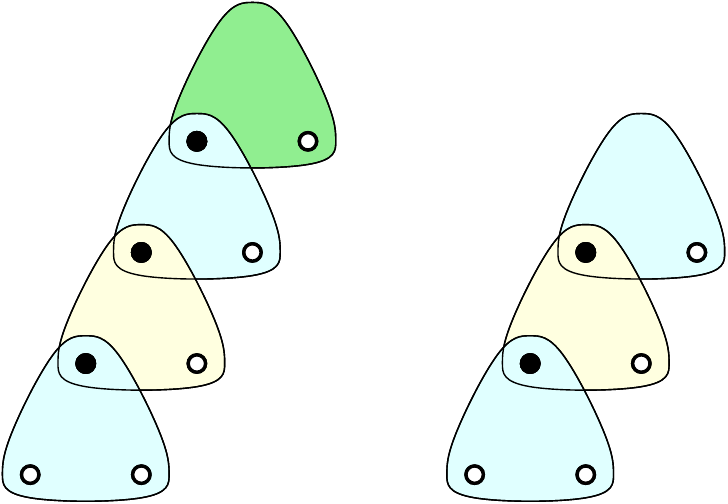}
\hspace{1cm}
\includegraphics[width=0.28\textwidth]{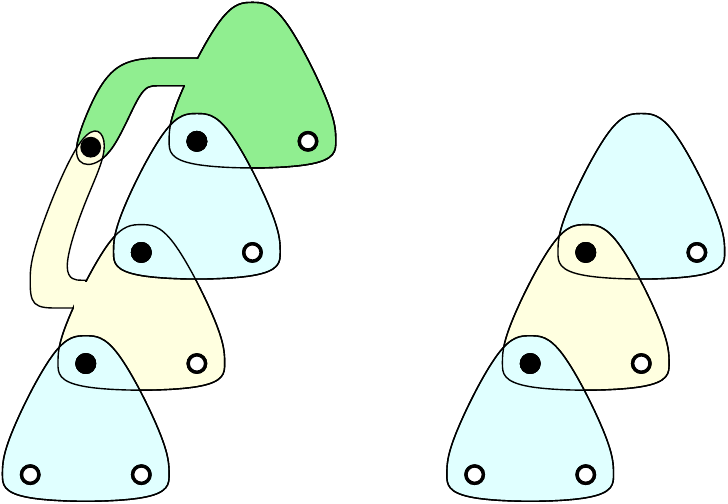}
\hspace{1cm}
\includegraphics[width=0.28\textwidth]{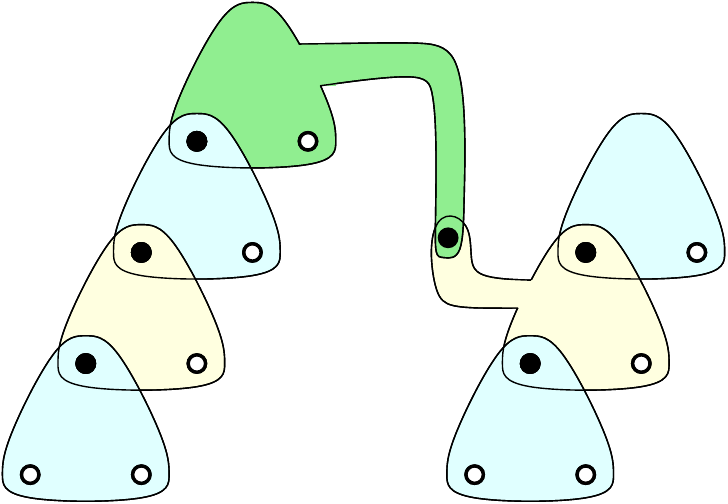}
\caption{Left: The hypergraph $H_\sigma$ in the case that $N=2$, 
$|K_1|=5$, and $|K_2|=4$.
% 
% 
 % has two vertices of with $\deg_\star$ equal to $5$ and $4$, respectively.
Note that the isomorphism class of $H_\sigma$ does not depend on $\sigma$.
 % $I=\{1,2\}$, $|W_1|=5$ and $|W_2|=4$. 
 Centre: Letting $K_1=\{x_{1,1},\ldots,x_{1,5}\}$, and $K_2=\{x_{2,1},\ldots,x_{2,4}\}$, this is the hypergraph $H_\sigma(x_{1,2},x_{1,4})$. 
% Examples of modifications of $H$}
Right: The hypergraph $H_\sigma(x_{1,4},x_{2,2})$.
}
\label{fig:inthyp}
\end{figure}

\medskip

We claim that 
\begin{equation}
\label{eq:Lprime}
 \eta_{d,2}(H_\sigma') \geq \eta_{d,2}(H_\sigma)+2
\end{equation}
for any coarsening $H_\sigma'$ of $H_\sigma$, so that
\begin{align*}\bareta_{d,2}(H_\sigma) = \eta_{d,2}(H_\sigma) &= (d-2)(3|K'|-|V|) -d|K'| -(d-2)(|K'|-|V|)\\
&= (d-4)|K'|.\end{align*}
and hence that
\begin{equation}
\label{eq:M0estimate}
\log_2 \E[M_{\sigma,0}] \lesssim -(d-4) |K'| \, n + |K'|^2 \log_2(n)
\end{equation}
by \cref{lem:firstmomentgeneral}.
Indeed, suppose that $\coarse{H_\sigma}{\bowtie}$ is a proper coarsening of $H_\sigma$ corresponding to some equivalence relation $\bowtie$ on $E(H_\sigma)$, and that the edge corresponding to $x=\sigma_i(j) \in K'$ is maximal in its equivalence class in the sense that there does not exist $\sigma_i(j')$ in the equivalence class of $\sigma_i(j)$ with $j' > j$. Clearly such a maximal $x$ must exist in every equivalence class. Moreover, for such a maximal $x = \sigma_i(j)$ there can be at most one edge of $H_\sigma$ that it shares a vertex with and is also in its class, namely the edge corresponding to $\sigma_i(j-1)$. Thus, if $x$ is maximal and its equivalence class is not a singleton, let $\coarse{H_\sigma}{\bowtie'}$ be the coarsening corresponding to the equivalence relation $\bowtie'$ obtained from $\bowtie$ by removing $x$ from its equivalence class. Then we have that $\Delta(\coarse{H_\sigma}{\bowtie'}) \leq \Delta(\coarse{H_\sigma}{\bowtie})+1$ and that $|E(\coarse{H_\sigma}{\bowtie'})| = |E(\coarse{H_\sigma}{\bowtie})|+1$, so that
\begin{equation}
\label{eq:plusfour}
\eta_d(\coarse{H}{\bowtie}) \geq \eta_d(\coarse{H_\sigma}{\bowtie'}) + d - (d-2)= \eta_d(\coarse{H_\sigma}{\bowtie'}) + 2, 
\end{equation}
and the claim follows by inducting on the number of edges in non-singleton equivalence classes.

\medskip

To obtain a bound on the expectation of $M_{\sigma,2}$, considering the contribution of each shell $\Lambda(m,m+1)$ yields the estimate
\begin{align*}
\sum_{w \in \Lambda_{w,2}}  \langle \zeta_{x} w \rangle^{-(d-2)}\langle \zeta_{x'} w \rangle^{-(d-2)}
 % &\preceq  \sum_{m=n+c_2+1}^\infty \exp_2\left( -\bareta_{d,2}(H)\, n + \log_2(n)\right)\\
& \preceq \sum_{m \geq n+ c_2 + 1} 2^{dn} 2^{-2(d-2)n} \preceq 2^{-(d-4)n}
\end{align*}
for every $\zeta \in \Lambda_\zeta$,
% we have that, by \cref{lem:firstmomentgeneral},
and it follows from \cref{lem:REcases} and \eqref{eq:M0estimate} that
\begin{align}
\log_2 \E[M_{\sigma,2}] &\lesssim \log_2 \E[M_{\sigma,0}] - (d-4)\, n
\nonumber
\\
 &\lesssim -(d-4)(|K'|+1)\, n + |K'|^2 \log_2 n.
 \label{eq:M2estimate}
\end{align}

\medskip

It remains to bound the expectation of $M_{\sigma,1}$. 
For each two distinct $x,x' \in K'$, let $H_\sigma(x,x')$ be the hypergraph with boundary obtained from $H_\sigma$ by adding a single vertex, $\star$, and adding this vertex to the two edges corresponding to $x$ and $x'$ respectively. These hypergraphs are defined in such a way that, by \cref{lem:REcases}, 
\[\E[M_{\sigma,1}] \preceq \sum_{(x,x')\in O} \bbW^{H_\sigma(x,x'),\, 2}_x(n+c_1,n+c_1+c_2)\]
We claim that 
\vspace{0.2cm}
\begin{equation}
\label{eq:Lprime2}
\vspace{0.2cm}
\bareta_{d,2}(H_\sigma(x,x')) \geq \bareta_{d,2}(H) +2 = (d-4)|K'|+2
\end{equation}
for every two distinct $x,x' \in K$.
First observe that 
 coarsenings of $H_\sigma$ and of $H_\sigma(x,x')$ both correspond to equivalence relations on $K$. Let $\bowtie$ be an equivalence relation on $K$, and let $H'_\sigma(x,x')$ and $H_\sigma'$ be the corresponding coarsenings. 
Clearly $|E(H'_\sigma(x,x'))|=|E(H_\sigma')|$ and $|V_\circ(H'_\sigma(x,x'))|=|V_\circ(H_\sigma')|+1$. 
If  $x$ and $x'$ are related under $\bowtie$, then we have that $\Delta(H'_\sigma(x,x')) = \Delta(H_\sigma')+1$, while if 
$x$ and $x'$ are not related under $\bowtie$, then we have that $\Delta(H'_\sigma(x,x')) = \Delta(H_\sigma')+2$.
We deduce that
\[\eta_{d,2}(H'_\sigma(x,x')) \geq \begin{cases} \eta_{d,2}(H_\sigma') &\text{ if $x \bowtie x'$}\\
\eta_{d,2}(H_\sigma') + 2 &\text{ otherwise.}  
\end{cases}
\]
If $x \bowtie x'$ then $H_\sigma'$ must be a proper coarsening of $H_\sigma$, and we deduce from \eqref{eq:Lprime} that the inequality $\eta_{d,2}(H'_\sigma(x,x')) \geq \eta_{d,2}(H_\sigma) +2$ holds for every coarsening $H'_\sigma(x,x')$ of $H_\sigma(x,x')$, yielding the claimed inequality \eqref{eq:Lprime2}.
 Using \eqref{eq:Lprime2}, we deduce from \cref{lem:firstmomentgeneral} that
\begin{equation}
\label{eq:M1estimate}
\log_2\E[M_{\sigma,1}] \lesssim -\bigl[(d-4)|K'|+2\bigr]\, n + |K'|^2\log_2 n. 
 \end{equation}
 Combining \eqref{eq:M2estimate} and \eqref{eq:M1estimate} yields the claimed estimate \eqref{eq:M1M2estimate}, completing the proof.
\qedhere
% \medskip

\end{proof}
\noindent
\emph{Completion of the proof of \cref{lem:firstmomentlowerboundestimate}.}
Since the upper bound given by \cref{lem:ABIC} is of lower order than the lower bound given by \cref{lem:farintersections}, it follows that 
there exists $n_0=n_0(|K|,N,d,c_1,c_2)$ such that
 \[\P(\sC \cap \sD \cap \sI \cap \sE) \leq \frac{1}{2}\P(\sC\cap\sD\cap \sI)\]
if $n \geq n_0$, and hence that
\[\log_2 \P(\sC \cap \sD \cap \sI \setminus \sE) \gtrsim \log_2 \P(\sC \cap \sD \cap \sI) \gtrsim -(d-4)|K'|\, n \]
for sufficiently large $n$ as claimed. \qedhere
% and \cref{lem:ABIC}

\end{proof}

\section{Proof of the main theorems}
\label{sec:wrappingup}

We now complete the proof of \cref{thm:mainhyper}. We begin with the simpler case in which $d/(d-4)$ is not an integer.

\begin{proof}[Proof of \cref{thm:mainhyper} for $d\notin \{5,6,8\}$]

We begin by analyzing faithful ubiquity. 
Let $\bbG$ be a $d$-dimensional transitive graph, and let $H$ be a finite hypergraph with boundary. If $H$ has a subhypergraph none of whose coarsenings are $d$-buoyant, then \cref{prop:nonubiquity} implies that $H$ is not faithfully ubiquitous in $\cC^{hyp}_r(\F)$ almost surely for any $r\geq 1$. 

% \medskip

Otherwise, by \cref{lem:maxminswap}, $H$ has a coarsening all of whose subhypergraphs are $d$-buoyant. If $d/(d-4)$ is not an integer, then it follows from \cref{prop:ubiquity} that 
there exist vertices $(x_v)_{v\in \partial V}$ in $G$ such that with positive probability, the vertices $x_v$ are in different components of $\F$ and $H$ is $R_\bbG(H)$-robustly faithfully present at $(x_v)_{v\in V}$. The set
\[\left\{\left(\omega,(x_v)_{ v\in \partial V}\right) \in \{0,1\}^{E(\bbG)} \times \bbV^{\partial V}: H \text{ is $R_\bbG(H)$-robustly faithfully present at $x$}\right\}\]
is a tail multicomponent property, and it follows from \cref{thm:indist} that 
$H$ is faithfully ubiquitous in $\Comp^{hyp}_{r}(\F)$ for every $r \geq R_\bbG(H)$ a.s.

% \medskip

We now turn to ubiquity. Let $r \geq 1$. It follows immediately from the definitions that if $H$ has a quotient that is faithfully ubiquitous in $\cC_r^{hyp}(\F)$ almost surely then $H$ is ubiquitous in $\cC_r^{hyp}(\F)$ almost surely, and so it suffices to prove the converse. 
If every quotient $H'$ of $H$ with $R_\bbG(H') \leq r$ has a subhypergraph none of whose coarsenings are $d$-buoyant, 
then $H$ is not ubiquitous in $\cC^{hyp}_r(\F)$  almost surely by \cref{prop:nonubiquity}. Otherwise, by \cref{lem:maxminswap}, $H$ has a quotient $H'$ with $R_\bbG(H') \leq r$ that has a coarsening all of whose subgraphs are $d$-buoyant, so that $H'$ is faithfully ubiquitous in  $\cC^{hyp}_r(\F)$ almost surely and therefore $H$ is ubiquitous in  $\cC^{hyp}_r(\F)$ almost surely by the above. This concludes the proof.
\qedhere

\end{proof}

\begin{proof}[Proof of \cref{thm:mainhyper} for $d\in \{5,6,8\}$]
The only part of the proof that requires modification in this case is the proof that if $H$ has a coarsening all of whose subhypergraphs are $d$-buoyant then $H$ is faithfully ubiquitous in $\cC^{hyp}_{R_\bbG(H)}(\F)$  almost surely. To show this, we will prove by induction on $|E(H)|$ that if every subhypergraph of $H$ is $d$-buoyant then every refinement $H'$ of $H$ is faithfully ubiquitous in $\cC^{hyp}_{R_\bbG(H')}(\F)$  almost surely. 

Let us first consider the base case $|E(H)|=1$. Since every subhypergraph of $H$ is $d$-buoyant, the unique edge of $H$ must contain at most $d/(d-4)$ boundary vertices. Let $H'$ be obtained from $H$ by deleting all internal vertices that are not in the unique edge of $H$, and, if necessary, adding additional new boundary vertices to the unique edge so that it contains exactly $d/(d-4)$ boundary vertices. Then it follows from \cref{prop:ubiquityspeciald} and \cref{thm:indist} that every refinement $H''$ of $H'$ is faithfully ubiquitous in $\cC^{hyp}_{R_\bbG(H'')}(\F)$ almost surely. It is easily verified from the definitions that this implies that every refinement $H'''$ of $H$ is is faithfully ubiquitous in $\cC^{hyp}_{R_\bbG(H'')}(\F)$ almost surely also. In particular, it follows that for every $n \leq d/(d-4)$, every set of $n$ trees of $\F$ are contained in an edge of $\cC^{hyp}_{r}(\F)$ for every $r\geq R_\bbG(n)$ almost surely.

Let $H$ be a finite hypergraph with boundary all of whose subhypergraphs are $d$-buoyant. Suppose that $|E(H)|\geq 2$ and that the claim has been established for all hypergraphs with fewer edges than $H$.  
If $H$ is $d$-basic then we are already done, so assume not. Then at least one of the following must occur:
\begin{enumerate}
\item $H$ has an edge of degree less than or equal to $d/(d-4)$.
\item $H$ has a proper, non-trivial bordered subhypergraph $H'$ with $\eta_d(H')=0$.
\end{enumerate}

Let us first consider the case that $H$ has an edge of degree less than or equal to $d/(d-4)$. Let $e_0$ be an edge of $H$ with $\deg(e_0)\leq d/(d-4)$ and let $H_1$ be the subhypergraph of $H$ with $\partial V(H_1)=\partial V(H)$, $V_\circ(H_1)=V_\circ(H)$, and $E(H_1)=E(H)\setminus\{e_0\}$. By the induction hypothesis, every refinement $H_1'$ of $H_1$ is faithfully ubiquitous in $\cC^{hyp}_{R_\bbG(H_1')}(\F)$ almost surely. 
Let $H_2$ be a refinement of $H$, and let $H_3$ be obtained from $H_2$ by deleting every edge of $H_2$ which corresponds to $e_0$ under the refinement. Then $H_3$ is a refinement of $H_1$, and so is faithfully ubiquitous in $\cC^{hyp}_{R_\bbG(H_3)}(\F)$ almost surely. On the other hand, every edge of $H_2$ that was deleted to form $H_3$ has degree at most $d/(d-4)$, and since $\cC^{hyp}_{R_\bbG(H_2)}(\F)$ contains every possible edge of these sizes almost surely, we deduce that $H_2$ is faithfully ubiquitous on $\cC^{hyp}_{R_\bbG(H_2)}(\F)$ almost surely.

Now suppose that $H$ has a proper, non-trivial bordered subhypergraph $H_1$ with $\eta_d(H_1)=0$. 
Let $H_2$ be the hypergraph with boundary that has $\partial V(H_2)=V(H_1)$, $V_\circ(H_2)=V_\circ(H)\setminus V_\circ (H_1)$, and $E(H_2)=E(H_1)\setminus E(H_1)$. We claim that every subhypergraph of $H_2$ is $d$-buoyant. Indeed, suppose that $H_3$ is a subhypergraph of $H_2$, and let $H_4$ be the subhypergraph of $H_1$ that includes all the edges and vertices of $H_1$ that are included in either $H_1$ or $H_3$ (noting that some of the boundary vertices of $H_3$ will become interior vertices of $H_4$). Let $N$ be the number of boundary vertices of $H_3$ that are interior vertices of $H_1$. Then we can compute that
$|E(H_4)|=|E(H_1)|+|E(H_3)|$, $|V_\circ(H_4)|=|V_\circ(H_1)|+|V_\circ(H_4)|+N$, and $\Delta(H_4)=\Delta(H_1)+\Delta(H_4)+N$, so that
\[
\eta_d(H_3)=\eta_d(H_3)+\eta_d(H_1)=\eta_d(H_4) \leq 0 
\]
since $\eta_d(H_1)=0$ and every subhypergraph of $H_1$ is $d$-buoyant.
Thus, we deduce from the induction hypotheses that every refinement $H'$ of either  $H_1$ or $H_2$ is faithfully ubiquitous in $\cC^{hyp}_r(\F)$ almost surely for every $r\geq R_\bbG(H) \geq \max\{R_\bbG(H_1),R_\bbG(H_2)\}$. It is easily verified that this implies that every refinement $H'$ of $H$ is faithfully ubiquitous in $\cC^{hyp}_r(\F)$ for every $r\geq R_\bbG(H')$ almost surely. \qedhere
% We claim that this implies that $H$ is also faithfully ubiquitous in 
% in $\cC^{hyp}_r(\F)$ almost surely for every $r\geq R_\bbG(H)$. 

 % Let $H''$ be the subhypergraph of $H'$ induced by those edges of $H'$ that do not have degree $d/(d-4)$, so that $H''$ and all its subhypergraphs are $d$-buoyant  by \cref{lem:optimalisbest}. By the cases already considered above, $H''$ is faithfully ubiquitous in $\cC^{hyp}_{r}(\F)$ almost surely for every $r\geq R_\bbG(H'')$. On the other hand, we also know from above that for every $r \geq R_\bbG(d/(d-4))$ and every set of $d/(d-4)$ vertices $\{x_1,\ldots, x_{d/(d-4)}\}$ in $\cC^{hyp}_{r}(\F)$, the subhypergraph of $\cC^{hyp}_{r}(\F)$ induced by this set of vertices is almost surely  the hypergraph containing as an edge every subset $A$ of $\{x_1,\ldots, x_{d/(d-4)}\}$. Thus, it follows that if $r \geq R_\bbG(H)$ then $H$ is faithfully present at a $\partial V$-tuple of points of $\cC^{hyp}_{r}(\F)$ if and only if $H'$ is present at the same tuple of points a.s. We deduce that $H$ is faithfully ubiquitous in $\cC^{hyp}_{r}(\F)$ for every $r \geq R_\bbG(H)$ almost surely as claimed, completing the part of the proof concerning faithful ubiquity.  

\end{proof}

\begin{proof}[Proof of \cref{thm:maintree}]
We begin by proving the claim about faithful ubiquity. Applying \cref{thm:main,lem:maxminswap}, and since every subgraph of a tree is a forest, it suffices to prove that if $T$ is a finite forest with boundary then $\eta_d(T') \geq \eta_d(T)$ whenever $d\geq4$ and $T'$ is a coarsening of $T$, so that, in particular,
\[\bareta_d(T) = \eta_d(T) = (d-8)|E| -(d-4)|V_\circ|\]
for every $d\geq 4$.

Indeed, suppose that $T'=\coarse{T}{\bowtie}$ is a proper coarsening of  a finite forest with boundary $T$. 
Since $T$ is a finite forest, the subgraph of $T$ spanned by each equivalence class of $\bowtie$ is also a finite forest, and therefore must contain a leaf. Choose a non-singleton equivalence class of $\bowtie$ and an edge $e$ of this equivalence relation that is incident to a leaf of the spanned forest.
Thus, $e$ has the property that one of the endpoints of $e$ is not incident to any other edge in $e$'s equivalence class. Let $\bowtie'$ be the equivalence relation obtained from $\bowtie$ by removing $e$ from its equivalence class and placing it in a singleton class by itself.  Then we have
that $|E(\coarse{T}{\bowtie'})| = |E(\coarse{T}{\bowtie})|+1$
and $\Delta(\coarse{T}{\bowtie'}) \leq \Delta(\coarse{T}{\bowtie})+1$ so that
\[\eta_d(\coarse{T}{\bowtie'})  \leq  \eta_d(\coarse{T}{\bowtie}) -4.\]
Thus, it follows by induction on the number of edges of $T$ in non-singleton equivalence classes that $\eta_d(\coarse{T}{\bowtie}) \geq \eta_d(T)$ for every coarsening $\coarse{T}{\bowtie}$ of $T$ as claimed. This establishes the claim about faithful ubiquity.

We now turn to ubiquity. Let $\bbG$ be a $d$-dimensional transitive graph for some $d>8$, let $r\geq 1$, and let $\F$ be the uniform spanning forest of $\bbG$.
 Let $T$ be a finite tree with boundary that is not faithfully ubiquitous in $\cC_r(\F)$, and let $T'$ be a subgraph of $T$ such that $(d-8)|E(T')| -(d-4)|V_\circ(T')| >0$, which exists by the previous paragraph. Since 
 \[(d-8)|E(T')| -(d-4)|V_\circ(T')| = \sum_{\substack{T'' \text{ a connected}\\\text{component of $T'$}}} (d-8)|E(T'')| -(d-4)|V_\circ(T'')|,\]
 we deduce that $T'$ has a connected subgraph $T''$ with $(d-8)|E(T'')| -(d-4)|V_\circ(T'')| >0$. Let $H$ be a quotient of $T$, let $H'$ be the image of $T''$ under the quotient map, and let $S$ be a spanning tree of $H'$, so that $|V_\circ(S)|\leq |V_\circ(T'')|$ and $|\partial V(S)| = |\partial V(T'')|$. Since $S$ and $T''$ are both trees, we have that $|E(S)| = |\partial V(S)|+|V_\circ(S)|-1$ and
 $|E(T'')| = |\partial V(T'')|+|V_\circ(T'')|-1$. We easily deduce that $\eta_d(S) \geq \eta_d(T'')>0$, and consequently that $S$ is not faithfully ubiquitous in $\cC_r(\F)$ almost surely. 
% 
% 
% 
%  For each edge $e$ of $H$, let $\pi^{-1}(e)$ be an edge of $T$ that is mapped to $e$ under the quotient. For each subgraph $S'$ of $S$, $\pi^{-1}(S')$ spans a subgraph of $T$ with at least as many interior vertices as $S'$ and the same number of edges. Thus, we have that
% \[ \max\left\{\frac{|E(S')|}{|V_\circ(S')|} : S' \text{ a subgraph of $S$}\right\}  \geq 
% \max\left\{\frac{|E(T')|}{|V_\circ(T')|} : T' \text{ a subgraph of $T$}\right\}.\]
% It therefore follows from the previous paragraph that if $S$ is faithfully ubiquitous in $\cC_r(\F)$ almost surely then $T$ is also. 
On the other hand, since $S$ is a subgraph of $H$, we have that if $H$ is faithfully ubiquitous in $\cC_r(\F)$ almost surely then $S$ is also. Since the quotient $H$ was arbitrary, it follows from \cref{thm:main} that $T$ is ubiquitous in $\cC_r(\F)$ if and only if it is faithfully ubiquitous in $\cC_r(\F)$ almost surely, completing the proof. \qedhere

\end{proof}

\begin{proof}[Proof of \cref{thm:mainsimple}]
To deduce item (1) from \cref{thm:main} and \cref{lem:maxminswap}, we need only prove that 
\[ f(d):= \min\left\{ \max \left\{\eta_d (H'') : H'' \text{ is a subhypergraph of $H'$}\right\} : H' \text{ is a coarsening of } H\right\}\]
is a non-decreasing function of $d \geq 4$ for every finite hypergraph with boundary $H$. 
Suppose that $H'$ is a subhypergraph of a coarsening of $H$. Let $H''$ be the largest subhypergraph of $H'$ that contains no edges or interior vertices of degree strictly less than $2$. In other words, $H''$ is obtained from $H'$ by recursively deleting edges and interior vertices of $H'$ that have degree strictly less than $2$ until no such edges or vertices remain.
It is easily verified that 
deleting edges or interior vertices of degree less than $2$ does not decrease the $d$-apparent weight when $d\geq 4$, and hence that
$\eta_d(H'') \geq \eta_d(H')$. 
% (Note that when deleting vertices we are required to identify edges that are incident to the same set of retained vertices, and delete edges that are only incident to deleted vertices; this is not a problem.)
 Thus, we have that
\begin{multline} 
\label{eq:etahatdegree2}
f(d) 
=\\ \min\left\{ \max \left\{ \eta_d(H''): \begin{array}{l}\text{ a subhypergraph of $H'$  with no edges}\\ \text{or interior vertices of degree $<2$}\end{array} \right\} : H' \text{ a coarsening of $H$}  \right\}
\end{multline}
for every $d\geq 4$.
% \]
If $H''$ is a finite hypergraph with boundary such that every edge and interior vertex of $H''$ has degree at least $2$, then $\Delta(H'') \geq 2|E(H'')|$ and $\Delta(H'') \geq 2|V_\circ(H'')|$, so that $\Delta(H'') \geq |E(H'')|+|V_\circ(H'')|$, and hence the coefficient of $d$ in $\eta_d(H'')$ is positive. Thus, the claimed monotonicity  follows from \eqref{eq:etahatdegree2}.

For item (2) it suffices by \cref{thm:maintree} to construct a family of finite trees with boundary $(T_d)_{d\geq 9}$ such that
\[
		\min\left\{\frac{|V_\circ(T'_d)|}{|E(T'_d)|}\,:\, T'_d \text{ is a subgraph of $T_d$}\right\} = \frac{d-8}{d-4}
\]
for each $d\geq 9$. We will use the family of trees pictured in \cref{fig:sepfamily}. Write $d= 4 + 5k + \ell$ where $0 \leq \ell < 5$ and let $T_d$ be the tree that has one vertex of degree five connected to $\ell$ paths of length $k+1$ and $5-\ell$ paths of length $k$. $T_d$ has five leaves, which we declare to be in its boundary, and declare all the other vertices to be in its interior. 
Clearly any subgraph $T'_d$ of $T_d$ maximizing $|V_\circ(T'_d)|/|E(T'_d)|$ must be induced by a union of geodesics joining the boundary vertices, and it is easily verified that, amongst these subgraphs, it is the full graph $T_d$ that maximizes $|V_\circ(T'_d)|/|E(T'_d)|$. To conclude, we compute that 
\[|V_\circ(T_d)| = 1 + 5(k-1) + \ell = d-8 
\quad
\text{ and }
\quad
|E(T_d)| = 5k + \ell = d-4,\]
so that $|V_\circ(T_d)|/|E(T_d)|=(d-8)/(d-4)$ as required.
\end{proof}

\section{Closing remarks and open problems}
\label{sec:closing}

\subsection{The number of witnesses}
\label{subsec:witnessesdiscussion}

The proof of \cref{thm:mainhyper} also yields the following result. 
If $\bbG$ is a $d$-dimensional transitive graph, $\F$ is the uniform spanning forest of $\bbG$, $H=(\partial V, V_\circ, E)$ is a finite hypergraph with boundary, and $r\geq 1$, then the following hold almost surely:
\begin{enumerate}[leftmargin=*]
\itemsep0.5em
\item If $H$ is faithfully ubiquitous in $\cC_r^{hyp}(\F)$, then for every collection $(x_u)_{u\in \partial V}$ of distinct vertices of $\cC_r^{hyp}(\F)$, there exists  a collection $(x^i_u)_{u \in V_\circ}$ of distinct vertices of $\cC_r^{hyp}(\F)$ for each $i \geq 1$ such that   $\{ x^i_u : u\in V_\circ, u \perp e \} \cup \{x_u : u\in\partial V, u \perp e \}$ is an edge of $\cC_r^{hyp}(\F)$ for every $i \geq 1$ and every $e \in E$, 
 $\{x^i_u : u \in V_\circ\}$ is disjoint from $\{x_u : u \in \partial V\}$ for every $i \geq 1$, and $\{x^i_u : u \in V_\circ\}$ and $\{x^j_u : u \in V_\circ\}$ are disjoint whenever $i> j \geq 1$.

\item
If $H$ is not faithfully ubiquitous in $\cC_r^{hyp}(\F)$, then for every collection $(x_u)_{u\in \partial V}$ of distinct vertices of $\cC_r^{hyp}(\F)$ there exists a finite set of vertices $A$ of $\cC_r^{hyp}(\F)$ such that $\{x_u : u \in V_\circ\}$ intersects $A$ whenever  $(x_u)_{u \in V_\circ}$ is a collection of distinct vertices of $\cC_r^{hyp}(\F)$ disjoint from $(x_u)_{u\in \partial V}$ with the property that
$\{ x^i_u : u\in V_\circ, u \perp e \} \cup \{x_u : u\in \partial V, u \perp e \}$ is an edge of $\cC_r^{hyp}(\F)$ for every $e \in E$.
\end{enumerate}
Indeed, item (2) is an immediate consequence of \cref{thm:indist}.

This has the following interesting consequence. 
For each $d >8$, it follows from \cref{thm:maintree} that the star with $\lceil (d-4)/(d-8) \rceil$ boundary leaves and one internal vertex is not faithfully ubiquitous in the component graph of the uniform spanning forest of $\Z^d$. Thus, we deduce from item (2), above, that if $d > 8$ then for every collection of $\lceil (d-4)/(d-8) \rceil$ distinct vertices of the component graph, there is almost surely some finite $M$ depending on the collection such that any clique containing the collection has size at most $M$. In particular, we conclude that the component graph of the uniform spanning forest of $\Z^d$  does not contain an infinite clique whenever $d>8$ a.s. In contrast, we note that the component graph of the uniform spanning forest of $\Z^d$ \emph{does} contain arbitrarily large cliques almost surely whenever $d \geq 5$. (This follows as a special case of \cref{thm:main} as in \cref{fig:degenerate}, but is also very easy to prove directly.)

\subsection{Further questions about the component graph of the USF.}

It is natural to wonder whether \cref{thm:main} determines the component graph up to isomorphism. It turns out that this is not the case. 
Indeed, observe that faithful ubiquity of a finite graph with boundary $H$ can be expressed as a first order sentence in the language of graphs:
\[\text{for all } (x_v)_{v \in \partial V} 
	\text{ there exists } 
	 (x_v)_{v \in V_\circ} \text{ such that } x_u\sim x_v \text{ for every } u,v \in V \text{ such that } u \sim v.
\]
Ubiquity of $H$ can be expressed similarly. However, even if we knew the almost-sure truth value of \emph{every} first order sentence in the language of graphs, this still would not suffice to determine the graph up to isomorphism. Indeed, recall that a graph $G=(V,E)$ is \textbf{quasi-$k$-transitive} if the action of its automorphism group on $V^k$ has only finitely many orbits. The model-theoretic Ryll-Nardzewski Theorem \cite[Theorem 7.3.1]{HodgesBook} implies that a countably infinite graph is determined up to isomorphism by its first order theory if and only if it is \textbf{oligomorphic}, i.e., quasi-$k$-transitive for every $k\geq 1$. 
% Such a graph must clearly have a very large amount of symmetry.
 By considering sizes of cliques as in \cref{subsec:witnessesdiscussion}, it follows from the discussion in that section that the component graph of the uniform spanning forest of $\Z^d$ is a.s.\ not quasi-$\lceil (d-4)/(d-8) \rceil$-transitive when $d>8$, and hence is a.s.\ not oligomorphic   when $d> 8$.
We conjecture that in fact the component graph has very little symmetry indeed.

\begin{conjecture}
Let $\bbG$ be a $d$-dimensional transitive graph for some $d > 8$, and let $r\geq 1$. Then $\cC_r(\F)$ has no non-trivial automorphisms almost surely. Moreover, there does not exist a deterministic graph $G$ such that $\cC_r(\F)$ is isomorphic to $G$ with positive probability.
\end{conjecture}

Although we do not believe the component graphs of the USF on different transitive graphs of the same dimension to be isomorphic, it seems nevertheless that most properties of the component graph should be determined by the dimension. One way of formalizing such a statement would be to axiomatize entire the almost-sure first order theory of the component graph of the uniform spanning forest and show that this first order theory is the same for different transitive graphs of the same dimension. We expect that \cref{thm:main}, or a slightly stronger variation of it, should play an important role in this axiomatization. 
See \cite{spencer2001strange} for the development of such a theory in the mean-field setting of Erd\H{o}s-R\'enyi graphs.
In particular, we believe the following.

\begin{conjecture}
Let $\bbG_1$ and $\bbG_2$ be $d$-dimensional transitive graphs, let $r_1,r_2\geq 1$, and let $\F_1$ and $\F_2$ be the uniform spanning forests of $\bbG_1$ and $\bbG_2$ respectively. Then the component graphs $\cC_{r_1}(\F_1)$ and $\cC_{r_2}(\F_2)$ are elementarily equivalent almost surely. That is, they satisfy the same set of first order sentences in the language of graphs almost surely.
\end{conjecture}

% \medskip

\subsection{Component graphs of other models and other graphs.}

It would be interesting to study ubiquitous subgraphs in component graphs derived from other models on $\Z^d$. The most tractable of these is likely to be the interlacement process \cite{Sznitman10,rath2010connectivity,procaccia2011geometry}, for which some related results have been proven by Lacoin and Tykesson \cite{lacoin2013easiest}. Here the component graph is defined by considering two trajectories to be adjacent if and only if they intersect.

\begin{question}
Let $d \geq 3$. Which finite graphs with boundary are ubiquitous in the component graph of the random interlacement on $\Z^d$?
\end{question}

The picture should be quite different to ours since the connection probabilities for more than two points are no longer given by a power of the spread.

\medskip

A much more straightforward extension of our results would be to consider uniform spanning forests generated by long-range random walks on $\Z^d$. Similarly, one could consider uniform spanning forests on non-transitive, possibly fractal, graphs that are Ahlfors-regular and satisfy sub-Gaussian heat kernel estimates of some order $\beta \geq 2$ (see e.g.\ \cite[Chapter 3]{KumFlour}). 
The beginnings of this analysis are already present implicitly in \cref{lem:firstmomentgeneral}.

\subsection*{Acknowledgments}
This work was carried out while TH was an intern at Microsoft Research, Redmond. TH thanks Mathav Murugan for many useful discussions on heat kernel estimates. We thank Omer Angel for his comments on an earlier draft of this manuscript, and thank the anonymous referee for many helpful comments and corrections.

\bibliographystyle{abbrv}
\bibliography{unimodular}
\end{document}